\documentclass[11pt,oneside,english]{amsart}
\usepackage[T1]{fontenc}
\usepackage[latin9]{inputenc}
\usepackage{babel}
\usepackage{verbatim}
\usepackage{prettyref}
\usepackage{amstext}
\usepackage{amsthm}
\usepackage{amssymb}
\usepackage{graphicx}
\usepackage{esint}
\usepackage[unicode=true,pdfusetitle,
 bookmarks=true,bookmarksnumbered=false,bookmarksopen=false,
 breaklinks=false,pdfborder={0 0 1},backref=false,colorlinks=false]
 {hyperref}

\makeatletter
\numberwithin{equation}{section}
\numberwithin{figure}{section}
\theoremstyle{plain}
\newtheorem{thm}{\protect\theoremname}[section]
  \theoremstyle{remark}
  \newtheorem{rem}[thm]{\protect\remarkname}
  \theoremstyle{plain}
  \newtheorem{cor}[thm]{\protect\corollaryname}
  \theoremstyle{plain}
  \newtheorem{prop}[thm]{\protect\propositionname}
  \theoremstyle{plain}
  \newtheorem{lem}[thm]{\protect\lemmaname}

\usepackage{mathrsfs,amsthm,amssymb,bbm,fullpage}

\newcommand{\cB}{\mathcal{B}}

\newcommand{\cE}{\mathcal{E}}

\newcommand{\R}{\mathbb{R}}

\newcommand{\N}{\mathbb{N}}

\newcommand{\ind}[1]{\mathbbm{1}_{#1}}

\newcommand{\id}{\text{id}}

\newcommand{\abs}[1]{\lvert#1\rvert}
\newcommand{\norm}[1]{\lvert\lvert#1\rvert\rvert}

\newrefformat{prop}{Proposition \ref{#1}}
\newrefformat{sub}{Section \ref{#1}}
\newrefformat{fact}{Fact \ref{#1}}
\newrefformat{cor}{Corollary \ref{#1}}
\newrefformat{rem}{Remark \ref{#1}}

\@ifundefined{showcaptionsetup}{}{%
 \PassOptionsToPackage{caption=false}{subfig}}
\usepackage{subfig}
\makeatother

  \providecommand{\corollaryname}{Corollary}
  \providecommand{\lemmaname}{Lemma}
  \providecommand{\propositionname}{Proposition}
  \providecommand{\remarkname}{Remark}
\providecommand{\theoremname}{Theorem}

\begin{document}

\title{Axial compression of a thin elastic cylinder: bounds on the minimum
energy scaling law}

\author{Ian Tobasco}
\address[Ian Tobasco]{Courant Institute of Mathematical Sciences, 251 Mercer St.\ NY, NY, USA, 10012}
\email{tobasco@cims.nyu.edu}

%\keywords{Thin elastic cylinder, minimum energy scaling law, }
%\subjclass[2010]{60K35, 82B44, 82D30,  49S05, 49K21} 
%\thanks{The final publication is available at Springer via http://dx.doi.org/10.1007/s00440-015-0691-z}

\date{\today}
\begin{abstract}
We consider the axial compression of a thin elastic cylinder placed
about a hard cylindrical core. Treating the core as an obstacle, we
prove upper and lower bounds on the minimum energy of the cylinder
that depend on its relative thickness and the magnitude of axial compression.
We focus exclusively on the setting where the radius of the core is
greater than or equal to the natural radius of the cylinder. We consider
two cases: the ``large mandrel'' case, where the radius of the core exceeds that of the cylinder, and
the ``neutral mandrel'' case, where the radii of the core and cylinder are the same. In the large mandrel
case, our upper and lower bounds match
in their scaling with respect to thickness, compression, and the magnitude
of pre-strain induced by the core. We construct three types of axisymmetric
wrinkling patterns whose energy scales as the minimum in different parameter regimes, corresponding to the presence of many wrinkles,
few wrinkles, or no wrinkles at all. In the neutral mandrel case,
our upper and lower bounds match in a certain regime in which the compression is small as compared
to the thickness; in this regime, the minimum energy scales as that
of the unbuckled configuration. We achieve these results for both the
von K\'arm\'an-Donnell model and a geometrically nonlinear model
of elasticity. 
\end{abstract}

\maketitle

\section{Introduction}

In many controlled experiments involving the axial compression of
thin elastic cylinders, one observes complex folding patterns (see,
e.g., \cite{donnell1934new,horton1965imperfections,pogorelov1988bendings,seffen2014surface}).
It is natural to wonder if such patterns are required to minimize
elastic energy, or if they are instead due to loading history. Before
we can begin to answer these questions, we need to understand the
minimum energy and in particular its dependence on external parameters.
This paper offers progress towards this goal.

Since the work of Horton and Durham \cite{horton1965imperfections},
it is a common experimental practice to place the elastic cylinder
about a hard inner core that stabilizes its deformation during loading.
In this paper, we consider the minimum energy of a compressed thin
elastic cylinder fit about a hard cylindrical core (which we also
refer to as the ``mandrel''). We prove upper and lower bounds on
the minimum energy which quantify its dependence on the thickness
of the cylinder, $h$, and the amount of axial compression, $\lambda$.
Ultimately, our goal is to identify the first term in the asymptotic
expansion of the minimum energy about $h,\lambda=0$. A more modest
goal, closer to what we achieve, is to prove upper and lower bounds
that match in scaling but not necessarily in pre-factor, e.g., 
\[
Ch^{\alpha}\lambda^{\beta}\leq\min\,E\leq C'h^{\alpha}\lambda^{\beta}.
\]
When our bounds match, which they do in some cases, we will have identified
the minimum energy scaling law along with test functions that achieve
this scaling.

There is a growing mathematical literature on minimum energy scaling
laws for thin elastic sheets. Some recent studies have considered
problems in which the direction of wrinkling is known in advance.
This could be due to the presence of a tensile boundary condition
\cite{bella2014wrinkles}, or a tensile body force such as gravity
pulling on a heavy curtain \cite{bella2015coarsening}. Such a tensile
force acts as a stabilizing mechanism, in that it pulls the wrinkles
taut and sets their direction. Then, the question is typically: how
should the wavelength of the wrinkles change throughout the sheet,
in order to achieve (nearly) minimal energy? Other works concern problems
in which the direction, or even the presence, of wrinkling is unknown
\emph{a priori}. These include works on blistering patterns \cite{belgacem2000rigorous,jin2001energy};
delamination \cite{bedrossian2015blister}; herringbone patterns \cite{kohn2013analysis};
and crumpling and folding of paper \cite{conti2008confining,venkataramani2004lower}.
In these papers, an important point is the construction of energetically
favorable crumpling or folding patterns which accommodate biaxial
compressive loads. 

In our view, the cylinder-mandrel problem belongs to either category,
as a function of whether the cylinder is fit snugly onto the mandrel
or not. Our analysis addresses the following two cases: the ``large
mandrel'' case, in which the natural radius of the cylinder is smaller
than that of the core, and the ``neutral mandrel'' case, in which
the radii of the cylinder and the core are the same. In the first
case, the mandrel pre-strains the cylinder along its hoops and, in
the presence of axial compression, this drives the formation of axisymmetric
wrinkles. In this setting, we prove upper and lower bounds on the
minimum energy that match in their scaling. The neutral mandrel case
is different, as there is no pre-strain to set the direction of wrinkling.
In this case, our best upper and lower bounds do not match (so that
at least one of them is suboptimal). Nevertheless, our lower bound
is among the few examples thus far of ansatz-free lower bounds in
problems involving confinement with the possibility of crumpling.
The cylinder-mandrel problem is similar in spirit to that of \cite{kohn2013analysis}:
in some sense, the obstacle in our analysis plays the role of their
elastic substrate. A key difference, however, is that in this paper
the cost of deviating from the mandrel is felt internally by the elastic
cylinder, whereas in \cite{kohn2013analysis} the cost of deviating
from the substrate is included as separate bulk effect. In this sense,
our discussion is also similar to that in \cite{bedrossian2015blister},
where the delaminated set is unknown.

These problems belong to a larger class in which the emergence of
``microstructure'' is modeled using a nonconvex variational problem
regularized by higher order terms (see, e.g., \cite{desimone2006recent,kohn1994surface,serfaty2006vortices}).
While we would like to understand energy minimizers, and eventually
local minimizers, a natural first step is to understand how the value
of the minimum energy depends on the problem's external parameters.
Proving upper bounds is conceptually straightforward, as it involves
evaluating the energy of suitable test functions; proving lower bounds
is more difficult, as the argument must be ansatz-free.

The presence of the inner obstacle in the cylinder-mandrel setup has
a stabilizing effect. This has been exploited in experiments which
explore both the incipient buckling load \cite{horton1965imperfections},
as well as buckled states deep into the bifurcation diagram \cite{seffen2014surface}.
In practice, there is a gap between the cylinder and the core (we
call this the ``small mandrel'' case). In the recent experimental
work \cite{seffen2014surface}, the authors explore the effect of
this gap size on the resulting buckling patterns. The character of
the observed patterns depends strongly on the size of the gap between
the cylinder and the core: in some cases the resulting structures
resemble origami folding patterns (e.g., the Yoshimura pattern), while
in other cases they resemble delamination patterns (e.g., the ``telephone-cord''
patterns discussed in \cite{moon2002characterization}). 

The effect of imposing a cylindrical geometry on confined thin elastic
sheets has also been explored in the literature. In the experimental
work \cite{roman2012stress}, Roman and Pocheau consider the axial
compression of a sheet trapped between two cylindrical obstacles.
The authors explore the effect of the size of the gap between the
obstacles on the compression-driven deformation of the sheet. When
the gap is large, the sheet exhibits crumples and folds; as the gap
shrinks, the sheet ``uncrumples'' in a striking fashion. At the
smallest reported gap sizes, the sheet appears to be (almost) axially
symmetric. This raises the question of whether the deformations from
\cite{seffen2014surface} would also become axially symmetric if the
size of the gap between the cylinder and mandrel were reduced to zero.
In the large mandrel case of the present paper, we prove that axially
symmetric wrinkling patterns achieve the minimum energy scaling law.
Our upper bounds in the neutral mandrel case also use axisymmetric
wrinkling patterns, but we wonder if optimal deformations must be
axisymmetric there. 

In the recent paper \cite{paulsen2016curvature}, Paulsen et al.\
consider the axial compression of a thin elastic sheet bonded to a
cylindrical substrate. The substrate acts as a Winkler foundation,
and sets the effective shape in the vanishing thickness limit. The
effective cylindrical geometry, in turn, gives rise to an additional
geometric stiffness which adds to the inherent stiffness of the substrate.
The authors also consider the effect of applying tension along the
wrinkles; the result is a local prediction for the optimal wavelength
of wrinkles in the sheet via the ``Far-From-Threshold'' approach
\cite{davidovitch2011prototypical}. 

The cylinder-mandrel problem offers a similar opportunity to discuss
the competition between stiffness of geometrical and physical origin.
In particular, in the neutral mandrel case, our lower bounds quantify
the additional stability afforded by the cylindrical obstacle. While
a flat sheet placed along a planar obstacle is immediately unstable
to compressive uniaxial loads, the same is not true in the presence
of cylindrical obstacles: superimposing wrinkles onto a curved shape
costs additional stretching energy. In the large mandrel case,
our upper and lower bounds balance the pre-strain induced stiffness
against the bending resistance. Since the resulting bounds match up
to prefactor, our prediction for the wavelength of wrinkling is optimal
in its scaling.

The present paper is not a study of the buckling load of a thin elastic
cylinder under axial compression, though this is an interesting problem
in its own right. This is the subject of the recent papers by Grabovsky
and Harutyunyan \cite{grabovsky2015rigorous,grabovsky2015scaling},
which give a rigorous derivation of Koiter's formula for the buckling
load from a fully nonlinear model of elasticity. These papers also
discuss the sensitivity of buckling to imperfections; in the context
of the von K\'arm\'an-Donnell equations, this is discussed in \cite{horak2006cylinder}.
(See also \cite{horak2008numerical,hunt2003cylindrical} for related
work.) The existence of a large family of buckling modes associated
with the incipient buckling load of a thin cylinder is consistent
with the development of geometric complexity when buckling first occurs.
One might imagine that the complexity seen experimentally reflects
the initial and perhaps subsequent bifurcations. Nevertheless, it
still makes sense to ask whether this complexity is required for,
or even consistent with, achievement of minimal energy. We cannot
begin to answer this question without first understanding the energy
scaling law.

In this paper, we prove upper and lower bounds on the minimum energy
in the cylinder-mandrel problem. Our upper bounds are ansatz-driven,
and we achieve them by constructing competitive test functions. In
contrast, our lower bounds are ansatz-free. Given enough compression,
low-energy test functions must buckle. Buckling in the presence of
the mandrel requires ``outwards'' displacement, and this leads to
tensile hoop stresses which cost elastic energy at leading order.
Thus, the mandrel drives buckling patterns to refine their length
scales to minimize elastic energy; this is compensated for by bending
effects, which prefer larger length scales overall. Through the use
of various Gagliardo-Nirenberg interpolation inequalities, we deduce
lower bounds by balancing these effects. In the large mandrel case,
this argument proves the minimum energy scaling law. In the neutral
mandrel case, the optimal such argument leads to matching bounds only
when the compression is small as compared to the thickness. For a
more detailed discussion of these ideas, we refer the reader to \prettyref{sub:DiscussionofProofs},
following the statements of the main results.

\subsection{The elastic energies}

We now describe the energy functionals that will be discussed in this
paper. Each is a model for the elastic energy per thickness of a unit
cylinder. Throughout this paper, we let $\theta\in I_{\theta}=[0,2\pi]$
be the reference coordinate along the ``hoops'' of the cylinder
and $z\in I_{z}=[-\frac{1}{2},\frac{1}{2}]$ be the reference coordinate
along the generators. The reference domain is $\Omega=I_{\theta}\times I_{z}$.

\subsubsection{The von K\'arm\'an-Donnell model}

The first model we consider is a geometrically linear model of elasticity,
which we refer to as the von K\'arm\'an-Donnell (vKD) model. Let
$\phi:\Omega\to\R^{3}$ be a displacement field, given in cylindrical
coordinates by $\phi=(\phi_{\rho},\phi_{\theta},\phi_{z})$. Treating
the ``in-cylinder'' displacements, $\phi_{\theta},\phi_{z}$, as
``in-plane'' displacements, the elastic strain tensor is given in
the vKD model by 
\begin{equation}
\epsilon=e(\phi_{\theta},\phi_{z})+\frac{1}{2}D\phi_{\rho}\otimes D\phi_{\rho}+\phi_{\rho}e_{\theta}\otimes e_{\theta}.\label{eq:FvKstrain}
\end{equation}
Assuming a trivial Hooke's law, the elastic energy per thickness is
given in this model by 
\begin{equation}
E_{h}^{vKD}(\phi)=\int_{\Omega}\,\abs{\epsilon}^{2}+h^{2}\abs{D^{2}\phi_{\rho}}^{2}\,d\theta dz.\label{eq:EFvK}
\end{equation}
Here, the symmetric linear strain tensor $e=e\left(\phi_{\theta},\phi_{z}\right)$
is given in $\left(\theta,z\right)$-coordinates by $e_{ij}=\left(\partial_{i}\phi_{j}+\partial_{j}\phi_{i}\right)/2$,
$i,j\in\left\{ \theta,z\right\} $, and the vectors $\left\{ e_{\theta},e_{z}\right\} $
are the reference coordinate basis vectors. The first term in \prettyref{eq:EFvK}
is known as the ``membrane term'', the second is the ``bending
term'', and the parameter $h$ is the (non-dimensionalized) thickness
of the sheet. The primary interest in this functional as a model of
elasticity is in the ``thin'' regime, $h\ll1$.

We note here that, as in \cite{horak2006cylinder,horak2008numerical,hunt2003cylindrical},
we choose to call this the von K\'arm\'an-Donnell model of elasticity.
In doing so, we invite comparison with the well-known F\"oppl-von
K\'arm\'an model for the elastic energy of a thin plate. In the
F\"oppl-von K\'arm\'an model, the elastic strain tensor is given
by
\[
\epsilon=e(u_{x},u_{y})+\frac{1}{2}Dw\otimes Dw,
\]
where $u=(u_{x},u_{y})$ and $w$ are the ``in-plane'' and ``out-of-plane''
displacements respectively. The elastic energy per thickness is then
given by the direct analog of \prettyref{eq:EFvK}. The key difference
between this model and the vKD model described above is the presence
of the last term in \prettyref{eq:FvKstrain}. This term is of geometrical
origin: it arises as $\phi_{\rho}$ describes the radial, or ``out-of-cylinder'',
displacement in the present work.

To model axial confinement of the elastic cylinder in the presence
of the mandrel, we consider the minimization of $E_{h}^{vKD}$ over
the admissible set
\begin{equation}
\begin{split}A_{\lambda,\varrho,m}^{vKD}= & \{\phi:\Omega\to\R^{3}\ :\ \phi_{\rho}\in H_{\text{per}}^{2}(\Omega),\ \phi_{\theta}\in H_{\text{per}}^{1}(\Omega),\ \phi_{z}+\lambda z\in H_{\text{per}}^{1}(\Omega)\}\\
 & \quad\cap\{\phi_{\rho}\geq\varrho-1,\ \max_{\substack{i\in\left\{ \theta,z\right\} ,\,j\in\{\rho,\theta,z\}}
}\norm{\partial_{i}\phi_{j}}_{L^{\infty}(\Omega)}\leq m\}.
\end{split}
\label{eq:AFvK}
\end{equation}
The parameter $\lambda\in(0,1)$ is the relative axial confinement
of the cylinder. The parameter $\varrho\in(0,\infty)$ is the radius
of the mandrel,\footnote{We warn the reader that while we use the subscript $\rho$ to denote
the radial component of a vector in $\R^{3}$, e.g., $x_{\rho}$,
we use the symbol $\varrho$ to denote the radius of the mandrel. } which we treat as an obstacle. The parameter $m\in(0,\infty]$ gives
an \emph{a priori }bound on the ``slope'' of the displacement, $D\phi$.
(As we will show, minimization of $E_{h}^{vKD}$ under axial confinement
prefers unbounded slopes as $h\to0$. We introduce the hypothesis
$m<\infty$ in order to systematically discuss sequences of test functions
which do not feature exploding slopes.) The assumption of periodicity
in the $z$-direction is for simplicity and does not change the essential
features of the problem.

\subsubsection{A nonlinear model of elasticity}

The vKD model described in the previous section fails to be physically
valid when the ``slope'' of the displacement, $D\phi$, is too large.
In this paper, we also consider the following nonlinear model for
the elastic energy per thickness: 
\begin{equation}
E_{h}^{NL}(\Phi)=\int_{\Omega}\,\abs{D\Phi^{T}D\Phi-\id}^{2}+h^{2}\abs{D^{2}\Phi}^{2}\,d\theta dz\label{eq:ENL}
\end{equation}
where $\Phi:\Omega\to\R^{3}$ is the deformation of the cylinder.
This is related to the displacement, $\phi$, through the formulas
\[
\Phi_{\rho}=1+\phi_{\rho},\quad\Phi_{\theta}=\theta+\phi_{\theta},\ \text{and}\quad\Phi_{z}=z+\phi_{z}.
\]
The functional $E_{h}^{NL}$ is a widely-used replacement for the
fully nonlinear elastic energy of a thin sheet (see, e.g., \cite{bella2014metric,conti2008confining}).
We note two simplifications from a fully nonlinear model: the energy
is written as the sum of a membrane term and a bending term; and where
a difference of second fundamental forms between that of the deformed
and that of the undeformed configurations would usually appear, it
has been replaced by the full matrix of second partial derivatives
of the deformation, $D^{2}\Phi$.

In parallel with the vKD model, we consider the minimization of $E_{h}^{NL}$
over the admissible set 
\begin{equation}
\begin{split}A_{\lambda,\varrho,m}^{NL}= & \{\Phi:\Omega\to\R^{3}\ :\ \Phi_{\rho}\in H_{\text{per}}^{2}(\Omega),\ \Phi_{\theta}-\theta\in H_{\text{per}}^{2}(\Omega),\ \Phi_{z}-(1-\lambda)z\in H_{\text{per}}^{2}(\Omega)\}\\
 & \quad\cap\{\Phi_{\rho}\geq\varrho,\ \max_{\substack{i\in\left\{ \theta,z\right\} ,\,j\in\{\rho,\theta,z\}}
}\norm{\partial_{i}\Phi_{j}}_{L^{\infty}(\Omega)}\leq m,\ \partial_{z}\Phi_{z}\geq0\ \text{Leb-a.e.}\}.
\end{split}
\label{eq:ANL}
\end{equation}
As above, $\lambda\in(0,1)$ is the relative axial confinement, $\varrho\in(0,\infty)$
is the radius of the mandrel, and $m\in(0,\infty]$ is an $L^{\infty}$\emph{-a
priori }bound on $D\Phi$. The final hypothesis, on the sign of $\partial_{z}\Phi_{z}$,
has no analog in \prettyref{eq:AFvK}, and deserves some additional
discussion. 

One might imagine that the cylinder should fold over itself to accommodate
axial compression. Indeed, if $z\to\Phi_{z}$ need not be invertible,
one can construct test functions that have significantly lower energy
than given in \prettyref{thm:NLlargemandrelscaling} or \prettyref{thm:NLneutralbounds}.
(In the notation of these results, such test functions can be made
to have excess energy no larger than $C(\varrho_{0})\max\{[(\varrho^{2}-1)\vee h^{2}]^{1/3}h^{4/3},h^{3/2}\}$
whenever $\varrho\in[1,\varrho_{0}]$ and $h,\lambda\in(0,\frac{1}{2}]$.)
In order to avoid this, and to facilitate a direct comparison with
the geometrically linear setting, we introduce the hypothesis that
$\partial_{z}\Phi_{z}\geq0$ in the definition of \prettyref{eq:ANL}.
We remark that such a hypothesis can be relaxed; as discussed in \prettyref{rem:invertibilityhypothesis},
one only needs to prevent $\partial_{z}\Phi_{z}$ from approaching
the well at $-1$ in order to obtain our results.

\subsection{Statement of results}

We prove quantitative bounds on the minimum energy of $E_{h}^{vKD}$
and $E_{h}^{NL}$ in two cases: the ``large mandrel case'', where
$\varrho>1$, and the neutral mandrel case, where $\varrho=1$. The
small mandrel case, where $\varrho<1$, is close to the poorly understood
question of the energy scaling law of a crumpled sheet of paper, which
is still a matter of conjecture (despite significant recent progress
offered in \cite{conti2008confining}).

\subsubsection{The large mandrel case\label{sub:largemandrelresults}}

We begin with the case where $\varrho>1$. In this setting, our methods
prove the minimum energy scaling law. We state the results first for
the vKD model. Define 
\begin{equation}
\cE_{b}^{vKD}(\varrho)=\left|\Omega\right|\left(\varrho-1\right)^{2}\label{eq:EbFVK}
\end{equation}
and let $c_{0}(\lambda,h,m)=\min\{\lambda^{1/2}h^{1/4},m^{1/2}h^{1/2}\}$.
\begin{thm}
\label{thm:FvKlargemandrelscaling} Let $h,\lambda\in(0,\frac{1}{2}]$,
$\varrho\in[1,\infty)$, and $m\in[2,\infty)$. Then we have that
\[
\min_{A_{\lambda,\varrho,m}^{vKD}}E_{h}^{vKD}-\cE_{b}^{vKD}\sim_{m}\min\left\{ \lambda^{2},\max\left\{ (\varrho-1)^{4/7}h^{6/7}\lambda^{5/7}, (\varrho-1)^{2/3}h^{2/3}\lambda\right\} \right\} 
\]
whenever $\varrho-1\geq c_{0}(\lambda,h,m)$. In the case that $m=\infty$,
we have that 
\[
\min_{A_{\lambda,\varrho,\infty}^{vKD}}E_{h}^{vKD}-\cE_{b}^{vKD}\sim\min\left\{ \lambda^{2},(\varrho-1)^{4/7}h^{6/7}\lambda^{5/7}\right\} 
\]
 whenever $\varrho-1\geq c_{0}(\lambda,h,\infty)$.\end{thm}
\begin{rem}
\label{rem:FvKslopeexlposion} Note that the scaling law $(\varrho-1)^{2/3}h^{2/3}\lambda $
disappears from the result when one does not assume an \emph{a priori}
$L^{\infty}$-bound on $D\phi$. Indeed, this assumption changes the
character of minimizing sequences. A consequence of our methods is
a quantification of the blow-up rate of $\norm{D\phi}_{L^{\infty}}$
as $h\to0$. For instance, if we fix $\varrho\in(1,\infty)$ and $\lambda\in(0,\frac{1}{2}]$,
then the minimizers $\{\phi_{h}\}$ of $E_{h}^{vKD}$ over $A_{\lambda,\varrho,\infty}^{vKD}$
satisfy $\norm{D\phi_{h}}_{L^{\infty}}\gtrsim_{\varrho,\lambda}h^{-2/7}$
as $h\to0$. The interested reader is directed to \prettyref{sub:Blowuprate_largemandrel}
for a precise statement of the full result. In any case, we are led
by this observation to include the parameter $m$ in the definition
of the admissible set, $A_{\lambda,\varrho,m}^{vKD}$, in order to
prevent the non-physical explosion of slope that is energetically
preferred in the large mandrel vKD problem.\end{rem}
\begin{proof}
\begin{comment}
By \prettyref{prop:FvKUB}, we have that 
\[
\min_{A_{\lambda,\varrho,m}^{FvK}}E_{h}^{FvK}-\cE_{b}^{FvK}\lesssim\min\left\{ \lambda^{2},\max\left\{ \lambda h,h^{6/7}\lambda^{5/7}(\varrho-1)^{4/7},m^{-1/3}(\varrho-1)^{2/3}\lambda h^{2/3}\right\} \right\} 
\]
whenever $h\in(0,\frac{1}{16}]$, $\lambda\in(0,\frac{1}{2}]$, $\varrho\in[1,\infty)$,
and $m\in[2,\infty]$. By \prettyref{prop:FvKlargemandrelLB}, we
have that 
\[
\min\left\{ \max\left\{ m^{-2/3}(\varrho-1)^{2/3}h^{2/3}\lambda,\lambda^{5/7}(\varrho-1)^{4/7}h^{6/7}\right\} ,\lambda^{2}\right\} \lesssim\min_{A_{\lambda,\varrho,m}^{FvK}}E_{h}^{FvK}-\cE_{b}^{FvK}
\]
whenever $h,\lambda\in(0,\infty)$, $\varrho\in[1,\infty)$, and $m\in(0,\infty]$.
The result now follows from the observation that 
\[
\lambda h\leq\max\{h^{6/7}\lambda^{5/7}(\varrho-1)^{4/7},m^{-1/3}(\varrho-1)^{2/3}\lambda h^{2/3}\}\iff\min\{\lambda^{1/2}h^{1/4},m^{1/2}h^{1/2}\}\leq\varrho-1.
\]
\end{comment}
{} \prettyref{thm:FvKlargemandrelscaling} follows from \prettyref{prop:FvKUB}
and \prettyref{prop:FvKlargemandrelLB}, once we note that 
\[
\lambda h\leq\max\{h^{6/7}\lambda^{5/7}(\varrho-1)^{4/7},m^{-1/3}(\varrho-1)^{2/3}\lambda h^{2/3}\}\iff\min\{\lambda^{1/2}h^{1/4},m^{1/2}h^{1/2}\}\leq\varrho-1.
\]

\end{proof}
This theorem shows that there are three types of patterns (three ``phases'')
which achieve the minimum energy scaling law, and that there are two
types of patterns if $m=\infty$. As we will see in the proof of the
upper bounds, these patterns consist of axisymmetric wrinkles. Roughly
speaking, the phases correspond to the absence of wrinkles, the presence
of one or a few wrinkles, or the presence of many wrinkles. The distinction
between ``few'' and ``many'' is made clear in \prettyref{sec:upperbounds}
(see \prettyref{lem:FvKUB_onewrinkle} and \prettyref{lem:FvKUB_manywrinkles}).
See \prettyref{fig:wrinklingpatterns} for a depiction of these wrinkling
patterns.

\begin{figure}
\subfloat[]{\includegraphics[height=0.18\textheight]{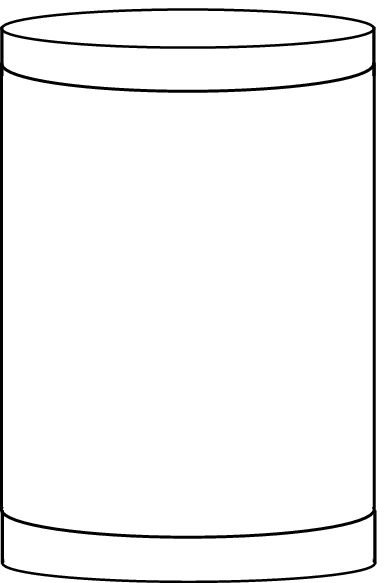}}\hspace{1.5em}\subfloat[]{\includegraphics[height=0.18\textheight]{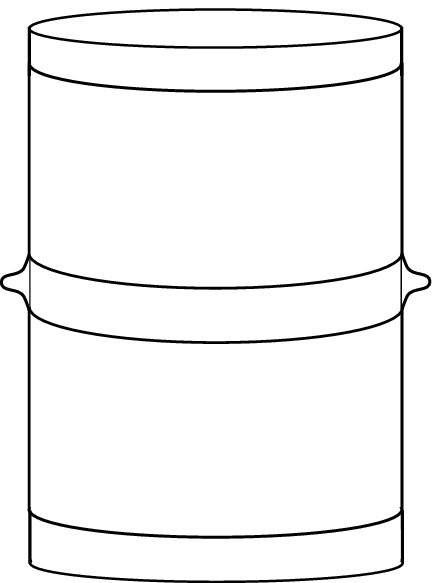}}\hspace{1.5em}\subfloat[]{\includegraphics[height=0.18\textheight]{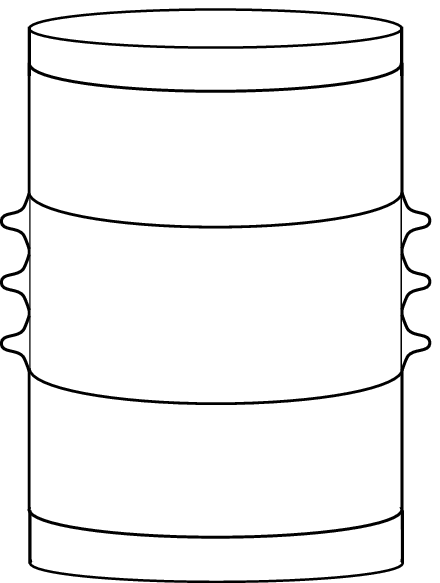}}

\caption{This figure depicts the three types of axisymmetric wrinkling patterns
that achieve the minimum energy scaling laws from \prettyref{thm:FvKlargemandrelscaling}.
In each, a thin elastic cylinder of unit radius and thickness $h$
is compressed axially by amount $\lambda$, and lies entirely outside
of an inner cylindrical mandrel of radius $\varrho>1$. Pattern A
shows the trivial wrinkling pattern, i.e., the unbuckled configuration,
which achieves an excess energy scaling as $\lambda^{2}$. Pattern
B is made up of one wrinkle, and achieves an excess energy scaling
as $(\varrho-1)^{4/7}h^{6/7}\lambda^{5/7}$. Pattern C features many
wrinkles, and achieves an excess energy scaling as $(\varrho-1)^{2/3} h^{2/3}\lambda$.
In this pattern, the number of wrinkles scales as $(\varrho-1)^{1/3}h^{-2/3}\lambda $.
A similar discussion applies for \prettyref{thm:NLlargemandrelscaling},
where $\varrho-1$ is replaced by $(\varrho^{2}-1)\vee h^{2}$. \label{fig:wrinklingpatterns}}
\end{figure}

A similar result can be proved for the nonlinear energy. Define 
\begin{equation}
\cE_{b}^{NL}(\varrho,h)=\left|\Omega\right|\left(\varrho^{2}-1\right)^{2}+\left|\Omega\right|\varrho^{2}h^{2}\label{eq:EbNL}
\end{equation}
and recall the definition of $c_{0}$ given immediately before the
statement of \prettyref{thm:FvKlargemandrelscaling} above.
\begin{thm}
\label{thm:NLlargemandrelscaling} Let $\varrho_{0}\in[1,\infty)$,
and let $h,\lambda\in(0,\frac{1}{2}]$, $\varrho\in[1,\varrho_{0}]$,
and $m\in[1,\infty)$. Then we have that
\[
\min_{A_{\lambda,\varrho,m}^{NL}}E_{h}^{NL}-\cE_{b}^{NL}\sim_{\varrho_{0},m}\min\left\{ \lambda^{2},\max\left\{ [(\varrho^{2}-1)\vee h^{2}]^{4/7}h^{6/7}\lambda^{5/7},[(\varrho^{2}-1)\vee h^{2}]^{2/3}h^{2/3}\lambda \right\} \right\} 
\]
whenever $(\varrho^{2}-1)\vee h^{2}\geq c_{0}(\lambda,h,1)$.\end{thm}
\begin{rem}
\label{rem:apriorislopebound} In contrast with \prettyref{thm:FvKlargemandrelscaling},
we do not address the case $m=\infty$ in this result. As the reader
will observe, our proof of the lower bound part of \prettyref{thm:NLlargemandrelscaling}
rests on the assumption that $m<\infty$. However, in the proof of
the upper bound part, the successful test functions belong to $A_{\lambda,\varrho,1}^{NL}$
uniformly in $h$. It does not appear to us that one can improve the
scaling of these upper bounds by considering test functions with exploding
slopes. This should be contrasted with the blow-up estimates discussed
for the vKD model in \prettyref{rem:FvKslopeexlposion}. %
\begin{comment}
We wonder if of $E_{h}^{NL}$ over $A_{\lambda,\varrho,\infty}^{NL}$:
are minimizers Lipschitz uniformly in $h$? Try an argument a la Chipot+Evans? 
\end{comment}
\end{rem}
\begin{proof}
\begin{comment}
By \prettyref{prop:NLUBs}, we have that 
\[
\min_{A_{\lambda,\varrho,m}^{NL}}E_{h}^{NL}-\cE_{b}^{NL}\lesssim_{\varrho_{0}}\min\left\{ \lambda^{2},\max\left\{ \lambda h,h^{6/7}\lambda^{5/7}[(\varrho^{2}-1)\vee h^{2}]^{4/7},[(\varrho^{2}-1)\vee h^{2}]^{2/3}\lambda h^{2/3}\right\} \right\} 
\]
whenever $h\in(0,\frac{1}{4}]$, $\lambda\in(0,\frac{1}{2}]$, $\varrho\in[1,\varrho_{0}]$,
and $m\in[1,\infty]$. By , we have that
\[
\min\left\{ \max\left\{ \left[(\varrho^{2}-1)\vee h^{2}\right]^{2/3}h^{2/3}\lambda,\lambda^{5/7}[(\varrho^{2}-1)\vee h^{2}]^{4/7}h^{6/7}\right\} ,\lambda^{2}\right\} \lesssim_{m,\varrho_{0}}\min_{A_{\lambda,\varrho,m}^{FvK}}E_{h}^{NL}-\cE_{b}^{NL}
\]
whenever $h,\lambda\in(0,1]$, $\varrho\in[1,\varrho_{0}]$, and $m\in(0,\infty)$.
The result now follows from the observation that 
\[
\lambda h\leq\max\{h^{6/7}\lambda^{5/7}[(\varrho^{2}-1)\vee h^{2}]^{4/7},[(\varrho^{2}-1)\vee h^{2}]^{2/3}\lambda h^{2/3}\}\iff\min\{\lambda^{1/2}h^{1/4},h^{1/2}\}\leq(\varrho^{2}-1)\vee h^{2}.
\]
\end{comment}
{} \prettyref{thm:NLlargemandrelscaling} follows from \prettyref{prop:NLUBs}
and \prettyref{prop:NLLBs_largemandr} once we observe that
\[
\lambda h\leq\max\{h^{6/7}\lambda^{5/7}[(\varrho^{2}-1)\vee h^{2}]^{4/7},[(\varrho^{2}-1)\vee h^{2}]^{2/3}\lambda h^{2/3}\}\iff\min\{\lambda^{1/2}h^{1/4},h^{1/2}\}\leq(\varrho^{2}-1)\vee h^{2}.
\]

\end{proof}

\subsubsection{The neutral mandrel case\label{sub:neutralmandrelresults}}

Next we turn to the borderline case between the large and small mandrel
cases, given by $\varrho=1$. In this case, our methods prove upper
and lower bounds on the minimum energy which fail to match in general,
though they do match in a regime in which the thickness, $h$, is
large as compared to the compression, $\lambda$. %
\begin{comment}
Note that in the neutral mandrel case, the bulk energy vanishes
as \$h\textbackslash{}to0\$.
\end{comment}

We begin with the results for the vKD model.
\begin{thm}
\label{thm:FvKneutralbounds} Let $h,\lambda\in(0,\frac{1}{2}]$ and
$m\in[2,\infty)$. Then we have that 
\[
\min\left\{ \max\{h\lambda^{3/2},(h\lambda)^{12/11}\},\lambda^{2}\right\} \lesssim_{m}\min_{A_{\lambda,1,m}^{vKD}}E_{h}^{vKD}\lesssim\min\left\{ h\lambda,\lambda^{2}\right\} .
\]
In the case that $m=\infty$, we have that 
\[
\min\left\{ \max\{(h\lambda)^{12/11}\},\lambda^{2}\right\} \lesssim\min_{A_{\lambda,1,\infty}^{vKD}}E_{h}^{vKD}\lesssim\min\left\{ h\lambda,\lambda^{2}\right\} .
\]
 \end{thm}
\begin{rem}
Although the lower bound in this result changes when $m=\infty$,
in this case it does not imply a blow-up rate for $\norm{D\phi}_{L^{\infty}}$
as $h\to0$. Indeed, as discussed in \prettyref{rem:unifbddslopes_neutralmandrel},
minimizing sequences need not have exploding slopes in the neutral
mandrel case. \end{rem}
\begin{proof}
Taking $\varrho=1$ in \prettyref{prop:FvKUB} proves the upper bound
part of \prettyref{thm:FvKneutralbounds}. To prove the lower bound
part, we first observe that if we define
\begin{equation}
FS_{h}(\phi)=\int_{\Omega}\,\abs{\epsilon_{\theta\theta}}^{2}+\abs{\epsilon_{zz}}^{2}+h^{2}\abs{D^{2}\phi_{\rho}}^{2}\,d\theta dz,\label{eq:FS}
\end{equation}
then 
\[
E_{h}^{vKD}(\phi)\geq FS_{h}(\phi)\quad\forall\,\phi\in A_{\lambda,\varrho,m}^{vKD}.
\]
\prettyref{prop:FSscalinglaw} identifies the minimum energy scaling
law of $FS_{h}$ over $A_{\lambda,1,m}^{vKD}$, and this proves the
result. 
\end{proof}
As the reader will note, the argument in the proof above uses only the $\theta\theta$-
and $zz$-components of the membrane term. As far as scaling is concerned,
the lower bounds given in \prettyref{thm:FvKneutralbounds} are the
optimal bounds that can be proved by such a method. This is discussed
in more detail in \prettyref{sub:neutralmandrelLB_FvK}; the essential
point is that our lower bounds arise as the minimum energy scaling
law of what we call the \emph{free-shear functional}, defined in \prettyref{eq:FS}
above.

The upper and lower bounds from \prettyref{thm:FvKneutralbounds}
match in a certain regime of the form $h\geq\lambda^{\alpha}$.
\begin{cor}
\label{cor:unbuckled}Let $h,\lambda\in(0,\frac{1}{2}]$ and $m\in[2,\infty)$.
If $h\geq\lambda^{5/6}$, we have that 
\[
\min_{A_{\lambda,1,m}^{vKD}}E_{h}^{vKD}\sim_{m}\lambda^{2}.
\]
The same result holds in the case that $m=\infty$. 
\end{cor}
\begin{comment}
This scaling law is achieved by the test function $\phi=(0,0,-\lambda z)$.
Is it true that in the region where $\min E_{h}^{FvK}\sim\lambda^{2}$,
this test function is the unique minimizer? 
\end{comment}

\begin{rem}
We note here a possible connection between our analysis and that of
\cite{grabovsky2015rigorous,grabovsky2015scaling}, which derives
Koiter's formula for the incipient buckling load of a (perfect) thin
cylinder via an analysis of the fully nonlinear model. Although our
focus is not on buckling as such, \prettyref{cor:unbuckled} proves
that, in the regime $\lambda\leq h^{6/5}$, the minimum energy scales
as that of the unbuckled deformation. In comparison, the buckling
load of a thin elastic cylinder scales linearly with $h$. If the
effect of the neutral mandrel is to improve local to global stability,
then perhaps the upper bound from \prettyref{thm:FvKneutralbounds}
is optimal in its scaling.\end{rem}
\begin{proof}
\prettyref{cor:unbuckled} follows from \prettyref{thm:FvKneutralbounds},
after observing that, since $\lambda\leq1$, 
\[
h\geq\lambda^{5/6}\iff\max\{h\lambda^{3/2},(h\lambda)^{12/11}\}\geq\lambda^{2}.
\]

\end{proof}
Now we state the corresponding results for the nonlinear energy.
\begin{thm}
\label{thm:NLneutralbounds} Let $h,\lambda\in(0,\frac{1}{2}]$ and
$m\in[1,\infty)$. Then we have that 
\[
\min\left\{ \max\left\{ h\lambda^{3/2},(h\lambda)^{12/11}\right\} ,\lambda^{2}\right\} \lesssim_{m}\min_{A_{\lambda,1,m}^{NL}}E_{h}^{NL}-\cE_{b}^{NL}(1,h)\lesssim_{\varrho_{0}}\min\left\{ \lambda h,\lambda^{2}\right\} .
\]
\end{thm}
\begin{rem}
As discussed in \prettyref{rem:apriorislopebound}, the lower bound
in the case that $m=\infty$ is not addressed for the nonlinear model
by our methods. \end{rem}
\begin{proof}
Taking $\varrho=\varrho_{0}=1$ in \prettyref{prop:NLUBs} gives the
upper bound part, once we observe that 
\[
\lambda\leq1\implies\lambda h\geq\min\{h^{2}\lambda^{5/7},\lambda^{2}\}.
\]
 The lower bound part follows from \prettyref{prop:NLneutralLBs}.\end{proof}
\begin{cor}
Let $h,\lambda\in(0,\frac{1}{2}]$ and $m\in[1,\infty)$. If $h\geq\lambda^{5/6}$,
then we have that 
\[
\min_{A_{\lambda,1,\infty}^{NL}}E_{h}^{NL}-\cE_{b}^{NL}(1,h)\sim_{m}\lambda^{2}
\]
\end{cor}
\begin{proof}
Arguing as in the proof of \prettyref{cor:unbuckled}, we see that
the result follows from \prettyref{thm:NLneutralbounds}.
\end{proof}

\subsection{Discussion of the proofs\label{sub:DiscussionofProofs}}

We turn now to a discussion of the mathematical ideas behind the proofs
of these results. To fix ideas, we focus exclusively in this section
on the nonlinear model, given in \prettyref{eq:ENL}. For added clarity,
we consider \textbf{only} the case where $h\to0$ while $\lambda\in(0,\frac{1}{2}]$,
$\varrho\in[1,\infty)$, and $m\in[1,\infty)$ are held fixed. Under
these additional assumptions, \prettyref{thm:NLlargemandrelscaling}
and \prettyref{thm:NLneutralbounds} imply the following results:
\begin{itemize}
\item If $\varrho>1$, there are constants $c,C$ depending only on $\lambda,\varrho,m$
such that
\begin{equation}
ch^{2/3}\leq\min_{A_{\lambda,\varrho,m}^{NL}}E_{h}^{NL}-\cE_{b}^{NL}\leq Ch^{2/3}\quad\text{as}\ h\to0.\label{eq:asymptoticexp}
\end{equation}

\item If $\varrho=1$, there are constants $c,C$ depending only on $\lambda,m$
such that 
\begin{equation}
ch\leq\min_{A_{\lambda,1,m}^{NL}}E_{h}^{NL}-\cE_{b}^{NL}\leq Ch\quad\text{as}\ h\to0.\label{eq:asymptoticexp-1}
\end{equation}

\end{itemize}
%We turn now to discuss the proofs of these results.

\subsubsection{The bulk energy }

We see from \prettyref{eq:EbNL} that $\cE_{b}^{NL}$ is of the form
\[
\cE_{b}^{NL}=b_{m}(\varrho)+b_{\kappa}(\varrho)h^{2}.
\]
The first factor, $b_{m}$, is the ``bulk membrane energy'' that
remains in the limit $h\to0$. The second factor, $b_{\kappa}h^{2}$,
is the ``bulk bending energy'' and appears in $\cE_{b}^{NL}$ due
to our choice of bending term. 

The bulk membrane energy can be found by solving the relaxed problem:
\begin{equation}
b_{m}=\min_{\Phi\in A_{\lambda,\varrho,m}^{NL}}\int_{\Omega}QW(D\Phi)\,dx.\label{eq:relaxedpblm}
\end{equation}
Here, $QW$ is the quasiconvexification of $W(F)=\left|F^{T}F-\id\right|^{2}$.
It follows from the results of \cite{pipkin1994relaxed} that 
\[
QW(F)=(\lambda_{1}^{2}-1)_{+}^{2}+(\lambda_{2}^{2}-1)_{+}^{2}
\]
where $\{\lambda_{i}\}_{i=1,2}$ are the singular values of $F$.

Regardless of whether we consider the large, neutral, or small mandrel
cases, the deformation 
\[
\Phi_{\text{eff}}(\theta,z)=(1+(\varrho-1)_{+},\theta,(1-\lambda)z)
\]
is a minimizer of \prettyref{eq:relaxedpblm}. The effective (first
Piola\textendash Kirchhoff) stress field is given by 
\begin{equation}
\sigma_{\text{eff}}=DQW(D\Phi_{\text{eff}})=4\varrho(\varrho^{2}-1)_{+}E_{\theta}\otimes e_{\theta},\label{eq:effectivestress}
\end{equation}
and the bulk membrane energy satisfies 
\[
b_{m}=|\Omega|(\varrho^{2}-1)_{+}^2.
\]
We note here that in the large mandrel case, where $\varrho>1$,
both $\sigma_{\text{eff}}$ and $b_{m}$ are non-zero, whereas for
the small or neutral mandrels these both vanish. As will become clear,
the appearance of different power laws for the scaling of the excess energy in \prettyref{eq:asymptoticexp}
and \prettyref{eq:asymptoticexp-1} is due precisely to the vanishing
or non-vanishing of $\sigma_{\text{eff}}$.

\subsubsection{Upper bounds}

To achieve the upper bounds from \prettyref{eq:asymptoticexp} and
\prettyref{eq:asymptoticexp-1}, one must construct a good test function
and estimate its elastic energy. The particular test functions that
we use are of the form 
\begin{equation}
\Phi(\theta,z)=(\varrho+w(z),\theta,(1-\lambda)z+u(z)).\label{eq:axisymmpattern}
\end{equation}
We refer to such constructions as ``axisymmetric wrinkling patterns''
(see \prettyref{fig:wrinklingpatterns}). By construction, the metric
tensor $g=D\Phi^{T}D\Phi$ satisfies $g_{\theta z}=0$ and by choosing
$u,w$ suitably we can ensure that $g_{zz}=0$ as well. 

In \prettyref{sec:upperbounds}, we estimate the elastic energy of
\eqref{eq:axisymmpattern}. The result is that the excess energy is
bounded above by a multiple of
\[
\int_{I_{z}}(\varrho^{2}-1)_{+}|w|+|w|^{2}+h^{2}|w''|^{2}\,dz,
\]
where $\norm{w'}_{L^{2}}\geq c(\lambda)$. Minimizing over all such
$w$ leads to the desired upper bounds. Evidently, both the character
of the optimal $w$ and the scaling in $h$ of the resulting upper
bound depend crucially on whether $\varrho>1$.

\subsubsection{Ansatz-free lower bounds}

The proofs of the lower bounds from \prettyref{eq:asymptoticexp}
and \prettyref{eq:asymptoticexp-1} require an ansatz-free argument.
We start by establishing the following claims:
\begin{enumerate}
\item With enough axial confinement, low-energy configurations must buckle;
\item Buckling in the presence of the mandrel induces excess hoop stress,
and costs energy.
\end{enumerate}
The first claim is quantified in \prettyref{cor:NLbucklingcontrol},
with the result being that low-energy configurations must satisfy
\begin{equation}
\norm{D\Phi_{\rho}}_{L^{2}}\geq c(\lambda).\label{eq:LBsketch-inequality1}
\end{equation}
The second claim is quantified in \prettyref{lem:NLmembrLB}; this
result implies in particular that the excess energy is bounded below
by a multiple of 
\begin{equation}
(\varrho^{2}-1)_{+}\norm{\Phi_{\varrho}-\varrho}_{L^{1}(\Omega)}+\norm{\Phi_{\rho}-\varrho}_{L_{z}^{2}L_{\theta}^{1}}^{2}.\label{eq:LBsketch-inequality2}
\end{equation}
The anisotropic norm appearing here is characteristic of our neutral
mandrel analysis. It arises because we consider the stretching of
each $\theta$-hoop individually in this case, a choice that may be
sub-optimal in general as it ignores the cost of shear.

Finally, we prove in \prettyref{lem:NLbendingcontrol} that, for low-energy
configurations, the excess energy is bounded below by a multiple of
\begin{equation}
h^{2}\norm{D^{2}\Phi_{\rho}}_{L^{2}(\Omega)}^{2}.\label{eq:LBsketch-inequality3}
\end{equation}
While such a bound comes for free when we consider $E_{h}^{vKD}$,
it requires some extra work for $E_{h}^{NL}$, due to the 
nonlinearities in the bending term.

Combining \prettyref{eq:LBsketch-inequality1}, \prettyref{eq:LBsketch-inequality2},
and \prettyref{eq:LBsketch-inequality3} with various Gagliardo-Nirenberg
interpolation inequalities (see \prettyref{sec:Appendix}), we conclude
the desired lower bounds.

\subsubsection{The role of $\sigma_{\text{eff}}$ in lower bounds}

As described above, the vanishing of the effective applied stress,
$\sigma_{\text{eff}}$, affects both the scaling law of the excess
energy as well as the character of low energy sequences. We wish now
to present a short argument for the first part of \prettyref{eq:LBsketch-inequality2}.
While this argument is not strictly necessary for the proof of the
main results, we believe that it helps to clarify the role of $\sigma_{\text{eff}}$
in the lower bounds.

It turns out that 
\[
E_{h}^{NL}(\Phi)-\cE_{b}^{NL}\geq\int_{\Omega}W(D\Phi)-b_{m},
\]
i.e., the excess energy can be split into its membrane and bending
parts (see \prettyref{lem:NLexcesssplits}). Since $QW\leq W$, we
have that 
\[
\int_{\Omega}W(D\Phi)-b_{m}\geq\int_{\Omega}QW(D\Phi)-QW(D\Phi_{eff}).
\]
If $\sigma_{\text{eff}}\neq0$, then to first order 
\begin{equation}
QW(D\Phi)-QW(D\Phi_{eff})=\left\langle \sigma_{\text{eff}},D(\Phi-\Phi_{\text{eff}})\right\rangle +h.o.t.,\label{eq:TaylorQW}
\end{equation}
and in fact we have that
\[
QW(D\Phi)-QW(D\Phi_{\text{eff}})\geq\left\langle \sigma_{\text{eff}},D(\Phi-\Phi_{\text{eff}})\right\rangle 
\]
since $QW$ is convex (this also follows from \cite{pipkin1994relaxed}).
Integrating by parts with the formula \prettyref{eq:effectivestress},
and using that $\Phi_{\rho}\geq\varrho$, we conclude that 
\[
\int_{\Omega}\left\langle \sigma_{\text{eff}},D(\Phi-\Phi_{\text{eff}})\right\rangle =\int_{\Omega}|\sigma_{\text{eff}}||\Phi_{\rho}-\varrho|.
\]
Hence, 
\[
E_{h}^{NL}(\Phi)-\cE_{b}^{NL}\geq|\sigma_{\text{eff}}|\norm{\Phi_{\rho}-\varrho}_{L^{1}(\Omega)}\quad\forall\,\Phi\in A_{\lambda,\varrho,\infty}^{NL}.
\]

While this argument succeeds in proving the first part of \prettyref{eq:LBsketch-inequality2},
it fails to prove the second part since, essentially, the expansion
\prettyref{eq:TaylorQW} fails to capture the leading order behavior of $QW$ in
the neutral mandrel case. Nevertheless, one can prove the full power
of \prettyref{eq:LBsketch-inequality2} assuming only that the cylinder
is at least as large as the mandrel, i.e., $\varrho\geq1$. The argument
we give in \prettyref{sub:largemandrelLB_nonlinear} establishes both
parts at once, using only familiar calculus and Sobolev-type inequalities along with
the basic definitions.

\subsection{Outline}

In \prettyref{sec:upperbounds}, we give the proofs of the upper bound
parts of \prettyref{thm:FvKlargemandrelscaling}, \prettyref{thm:NLlargemandrelscaling},
\prettyref{thm:FvKneutralbounds}, and \prettyref{thm:NLneutralbounds}.
In \prettyref{sec:largemandrelLB} we prove the lower bounds in the
large mandrel case, i.e., the lower bound parts of \prettyref{thm:FvKlargemandrelscaling}
and \prettyref{thm:NLlargemandrelscaling}. In \prettyref{sec:neutralmandrelLB},
we consider the analysis of lower bounds in the neutral mandrel case.
There, we prove the lower bound parts of \prettyref{thm:FvKneutralbounds}
and \prettyref{thm:NLneutralbounds}, as well as the energy scaling
law for the free-shear functional. We end with a short appendix in
\prettyref{sec:Appendix} which contains the various interpolation
inequalities that we use.

\subsection{Notation\label{sub:notation}}

The notation $X\lesssim Y$ means that there exists a positive numerical
constant $C$ such that $X\leq CY$, and the notation $X\lesssim_{a}Y$
means that there exists a positive constant $C'$ depending only on
$a$ such that $X\leq C'(a)Y$. The notation $X\sim Y$ means that
$X\lesssim Y$ and $Y\lesssim X$, and similarly for $X\sim_{a}Y$. 

When the meaning is clear, we sometimes abbreviate function spaces on
$\Omega$ by dropping the dependence on the domain, e.g., 
 $H^{k}=H^{k}(\Omega)$. The space $H_{\text{per}}^{k}=H_{\text{per}}^{k}(\Omega)$
is the space of periodic Sobolev functions on $\Omega$ of order $k$ and integrability $2$.
We employ the following notation regarding mixed $L^{p}$-norms: 
\[
\norm{f}_{L_{x_{1}}^{p_{1}}L_{x_{2}}^{p_{2}}}=\left(\int\left(\int|f(x_{1},x_{2})|^{p_{2}}\,dx_{2}\right)^{\frac{p_{1}}{p_{2}}}\,dx_{1}\right)^{\frac{1}{p_{1}}}
\]
and
\[
\norm{f}_{L_{x_{1}}^{p}}(x_{2})=\left(\int|f(x_{1},x_{2})|^{p}\,dx_{1}\right)^{\frac{1}{p}}.
\]

We refer to the unit basis vectors for the reference $\theta,z$-coordinates
on $\Omega$ as $\{e_{i}\}_{i\in\{\theta,z\}}$, and the unit frame
of coordinate vectors for the cylindrical $\rho,\theta,z$-coordinates
on $\R^{3}$ as $\{E_{i}\}_{i\in\{\rho,\theta,z\}}$. Note that $E_{\rho}=E_{\rho}(x)$
and $E_{\theta}=E_{\theta}(x)$ depend on $x\in\R^{3}$ through its
$\theta$-coordinate, $x_{\theta}$; our convention is that $E_{\rho}$
points in the direction of increasing radial coordinate, $\rho$,
and $E_{\theta}$ in the direction of increasing azimuthal coordinate,
$\theta$, so that in particular $x=x_{\rho}E_{\rho}(x)+x_{z}E_{z}$.
We will sometimes perform Lebesgue averages of a function $f:\Omega\to\R$
over the reference $\theta$-coordinate. We denote this by
\[
\overline{f}(z)=\frac{1}{\left|I_{\theta}\right|}\int_{I_{\theta}}f(\theta,z)\,d\theta.
\]
The notation $\left|A\right|$ denotes the Euclidean volume of the
(Lebesgue measurable) set $A$. The set $\mathcal{B}(U)$ denotes
the set of Lebesgue measurable subsets $A\subset U$.

\subsection{Acknowledgements}

We would like to thank our advisor R.~V.\ Kohn for his constant support.
We would like to thank S.\ Conti for many inspirational discussions
during an intermediate phase of this project, and in particular for
his insight into the analysis of the free-shear functional. We would
like to thank the University of Bonn for its hospitality during our
visit in April and May of 2015. This research was conducted while
the author was supported by a National Science Foundation Graduate
Research Fellowship DGE-0813964, and National Science Foundation grants
OISE-0967140 and DMS-1311833.

\section{Elastic energy of axisymmetric wrinkling patterns\label{sec:upperbounds}}

We begin our analysis of the compressed cylinder by estimating the
elastic energy of various axisymmetric wrinkling patterns. This amounts
to considering test functions that depend only on the $z$-coordinate.
The results in this section constitute the upper bound parts of \prettyref{thm:FvKlargemandrelscaling},
\prettyref{thm:NLlargemandrelscaling}, \prettyref{thm:FvKneutralbounds},
and \prettyref{thm:NLneutralbounds}. We consider the vKD model in
\prettyref{sub:FvKUB} and the nonlinear model in \prettyref{sub:NLUB}.

\subsection{vKD model\label{sub:FvKUB}}

Recall the definitions of $E_{h}^{vKD}$, $A_{\lambda,\varrho,m}^{vKD}$,
and $\cE_{b}^{vKD}$, given in \prettyref{eq:EFvK}, \prettyref{eq:AFvK},
and \prettyref{eq:EbFVK} respectively. In this section, we prove
the following upper bound.
\begin{prop}
\label{prop:FvKUB} We have that 
\[
\min_{A_{\lambda,\varrho,m}^{vKD}}E_{h}^{vKD}-\cE_{b}^{vKD}\lesssim\min\left\{ \lambda^{2},\max\left\{ \lambda h,h^{6/7}\lambda^{5/7}(\varrho-1)^{4/7},m^{-1/3}(\varrho-1)^{2/3}\lambda h^{2/3}\right\} \right\} 
\]
whenever $h,\lambda\in(0,\frac{1}{2}]$, $\varrho\in[1,\infty)$,
and $m\in[2,\infty]$.\end{prop}
\begin{proof}
The upper bound of $\lambda^{2}$ is achieved by the unbuckled configuration,
$\phi=(\varrho-1,0,-\lambda z)$. To prove the remainder of the upper
bound, note first that it suffices to achieve it for $(h,\lambda,\varrho,m)\in(0,h_{0}]\times(0,\frac{1}{2}]\times[1,\infty)\times[2,\infty]$
for some $h_{0}\in(0,\frac{1}{2}]$. We apply \prettyref{lem:FvKUB_manywrinkles},
\prettyref{lem:FvKUB_onewrinkle}, and \prettyref{lem:FvKUB_uniforminmandrel}
to deduce the required upper bound in the stated parameter range with
$h_{0}=\frac{1}{2^{4}}$. 
\end{proof}
In the remainder of this section, we will \textbf{assume} that 
\[
h\in(0,\frac{1}{2^{4}}],\ \lambda\in(0,\frac{1}{2}],\ \varrho\in[1,\infty),\ \text{and}\ m\in[2,\infty]
\]
unless otherwise explicitly stated.

We begin by defining a two-scale axisymmetric wrinkling pattern. We
will refer to the parameters $n\in\N$ and $\delta\in(0,1]$, which
are the number of wrinkles and their relative extent. We refer the
reader to \prettyref{fig:radialheightfield} for a schematic of this
construction. 

Fix $f\in C^{\infty}(\R)$ such that
\begin{itemize}
\item $f$ is non-negative and one-periodic
\item $\text{supp}\,f\cap[-\frac{1}{2},\frac{1}{2}]\subset(-\frac{1}{2},\frac{1}{2})$
\item $\norm{f'}_{L^{\infty}}\leq2$
\item $\norm{f'}_{L^{2}(B_{1/2})}^{2}=1$,
\end{itemize}
and define $f_{\delta,n}\in C^{\infty}(\R)$ by
\[
f_{\delta,n}(t)=\frac{\sqrt{\delta}}{n}f(\frac{n}{\delta}\{t\})\ind{\{t\}\in B_{\delta/2}}.
\]
Define $w_{\delta,n,\lambda},u_{\delta,n,\lambda}:\Omega\to\R$ by
\[
w_{\delta,n,\lambda}(\theta,z)=\sqrt{2\lambda}f_{\delta,n}(z)\quad\text{and}\quad u_{\delta,n,\lambda}(\theta,z)=\int_{-\frac{1}{2}\leq z'\leq z}\lambda-\frac{1}{2}(\partial_{z}w_{\delta,n,\lambda}(\theta,z'))^{2}\,dz'.
\]
Finally, define $\phi_{\delta,n,\lambda,\varrho}:\Omega\to\R^{3}$
by
\[
\phi_{\delta,n,\lambda,\varrho}=(w_{\delta,n,\lambda}+\varrho-1,0,-\lambda z+u_{\delta,n,\lambda}),
\]
in cylindrical coordinates. 

\begin{figure}
\includegraphics[height=0.18\textheight]{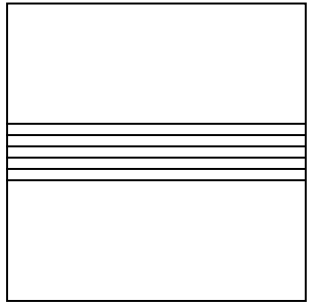}

\caption{This schematic depicts the axisymmetric wrinkle construction used
in the proof of the upper bounds. The pattern features $n$ wrinkles
in the $e_{z}$-direction with volume fraction $\delta$. The optimal
choice of $\delta,n$ depends on the axial compression, $\lambda$,
the thickness, $h$, the mandrel's radius, $\varrho$, and the \emph{a
priori} $L^{\infty}$ slope bound, $m$. \label{fig:radialheightfield}}
\end{figure}

Now, we estimate the elastic energy of this construction in the vKD
model. Define 
\[
m_{1}(\lambda,\delta)=2\max\left\{ \sqrt{\frac{2\lambda}{\delta}},\frac{2\lambda}{\delta}\right\} .
\]

\begin{lem}
\label{lem:FvKadmissability}We have that $\phi_{\delta,n,\lambda,\varrho}\in A_{\lambda,\varrho,m_{1}}^{vKD}$.
Furthermore,
\[
E_{h}^{vKD}(\phi_{\delta,n,\lambda,\varrho})-\cE_{b}^{vKD}\lesssim\max\left\{ (\varrho-1)\frac{\lambda^{1/2}\delta^{3/2}}{n},\frac{\lambda\delta^{2}}{n^{2}},h^{2}\frac{\lambda n^{2}}{\delta^{2}}\right\} .
\]
\end{lem}
\begin{proof}
Abbreviate $\phi_{\delta,n,\lambda,\varrho}$ by $\phi$, $w_{\delta,n,\lambda}$
by $w$, and $u_{\delta,n,\lambda}$ by $u$. We claim that $\phi_{\rho}\in H_{\text{per}}^{2}$,
$\phi_{\theta}\in H_{\text{per}}^{1}$, and $\phi_{z}+\lambda z\in H_{\text{per}}^{1}$.
To see this, observe that 
\[
\int_{I_{z}}\frac{1}{2}|\partial_{z}w_{\delta,n,\lambda}|^{2}dz=\lambda\int_{B_{\delta/2}}|f'_{\delta,n}|^{2}dt=\lambda\int_{B_{1/2}}|f'|^{2}dt=\lambda
\]
for all $\theta\in I_{\theta}$, so that $u\in H_{\text{per}}^{1}$. That
$w\in H_{\text{per}}^{2}$ follows from its definition. Observe also that
$\phi_{\rho}\geq\varrho-1$, since $w\geq0$. 

Now we check the slope bounds. By construction, we have that
\[
\epsilon_{zz}=\partial_{z}\phi_{z}+\frac{1}{2}(\partial_{z}\phi_{\rho})^{2}=0
\]
and that
\[
\partial_{z}\phi_{\rho}=\partial_{z}w=\sqrt{2\lambda}f'_{\delta,n}.
\]
Hence, 
\[
\norm{\partial_{z}\phi_{\rho}}_{L^{\infty}}\le\sqrt{2\lambda}\norm{f_{\delta,n}'}_{L^{\infty}}\leq2\sqrt{\frac{2\lambda}{\delta}}
\]
and
\[
\norm{\partial_{z}\phi_{z}}_{L^{\infty}}\leq\lambda\norm{f'_{\delta,n}}_{L^{\infty}}^{2}\leq\frac{4\lambda}{\delta}.
\]
It follows that
\[
\max_{\substack{i\in\left\{ \theta,z\right\} ,\,j\in\{\rho,\theta,z\}}
}\norm{\partial_{i}\phi_{j}}_{L^{\infty}}\leq m_{1}(\lambda,\delta),
\]
and therefore that $\phi\in A_{\lambda,\varrho,m_{1}}^{vKD}$.

Now we bound the elastic energy of this construction. Since $\epsilon_{zz}=\epsilon_{\theta z}=0$
and $w$ depends only on $z$, we see that
\[
E_{h}^{vKD}(\phi)=\int_{\Omega}\,\abs{w+\varrho-1}^{2}+h^{2}\abs{\partial_{z}^{2}w}^{2}\,d\theta dz
\]
and hence that
\[
E_{h}^{vKD}(\phi)-\cE_{b}^{vKD}\lesssim\max\left\{ (\varrho-1)_{+}\norm{w}_{L^{1}(\Omega)},\norm{w}_{L^{2}(\Omega)}^{2},h^{2}\norm{\partial_{z}^{2}w}_{L^{2}(\Omega)}^{2}\right\} .
\]
Now we conclude the desired result from the elementary bounds
\[
\norm{w}_{L^{1}(\Omega)}\lesssim\frac{\lambda^{1/2}\delta^{3/2}}{n},\quad\norm{w}_{L^{2}(\Omega)}^{2}\lesssim\frac{\lambda\delta^{2}}{n^{2}},\ \text{and}\quad\norm{\partial_{z}^{2}w}_{L^{2}(\Omega)}^{2}\lesssim\frac{\lambda n^{2}}{\delta^{2}}.
\]

\end{proof}
We make three choices of the parameters $n,\delta$ in what follows.
First, we consider a construction which features many wrinkles as
$h\to0$. 
\begin{lem}
\label{lem:FvKUB_manywrinkles} Assume that $m<\infty$ and that 
\[
m^{-1/3}(\varrho-1)^{2/3}\lambda h^{2/3}\geq\max\{\lambda h,h^{6/7}\lambda^{5/7}(\varrho-1)^{4/7}\}.
\]
Let $n\in\N$ and $\delta\in(0,1]$ satisfy
\[
n\in\left[(\varrho-1)^{1/3}\lambda h^{-2/3}m^{-7/6},2(\varrho-1)^{1/3}\lambda h^{-2/3}m^{-7/6}\right]\quad\text{and}\quad\delta=4\lambda m^{-1}.
\]
Then, $\phi_{\delta,n,\lambda,\varrho}\in A_{\lambda,\varrho,m}^{vKD}$
and 
\[
E_{h}^{vKD}(\phi_{\delta,n,\lambda,\varrho})-\cE_{b}^{vKD}\lesssim\frac{(\varrho-1)^{2/3}h^{2/3}\lambda}{m^{1/3}}.
\]
\end{lem}
\begin{proof}
Rearranging the inequality $m^{-1/3}(\varrho-1)^{2/3}\lambda h^{2/3}\geq h^{6/7}\lambda^{5/7}(\varrho-1)^{4/7}$,
we find that $(\varrho-1)^{1/3}\lambda h^{-2/3}m^{-7/6}\geq1$ so
that there exists such an $n\in\N$. Also, with our choice of $\delta$
we have that $m_{1}(\delta,\lambda)=m.$ We note that indeed $\delta\leq1$
since $\lambda\leq\frac{1}{2}$ and $m\geq2$.

It follows from \prettyref{lem:FvKadmissability} that $\phi_{\delta,n,\lambda,\varrho}\in A_{\lambda,\varrho,m}^{vKD}$,
and that
\[
E_{h}^{vKD}(\phi_{\delta,n,\lambda,\varrho})-\cE_{b}^{vKD}\lesssim\max\left\{ (\varrho-1)^{2/3}h^{2/3}m^{7/6}\delta^{3/2}\frac{1}{\lambda^{1/2}},\frac{\delta^{2}h^{4/3}m^{7/3}}{(\varrho-1)^{2/3}\lambda},h^{2/3}\frac{\lambda^{3}(\varrho-1)^{2/3}}{\delta^{2}m^{7/3}}\right\} .
\]
Using that $\delta\sim\frac{\lambda}{m}$, we have that
\[
E_{h}^{vKD}(\phi_{\delta,n,\lambda,\varrho})-\cE_{b}^{vKD}\lesssim\max\left\{ \frac{(\varrho-1)^{2/3}h^{2/3}\lambda}{m^{1/3}},\lambda m^{1/3}\frac{h^{4/3}}{(\varrho-1)^{2/3}}\right\} .
\]
Since
\[
\frac{(\varrho-1)^{2/3}h^{2/3}\lambda}{m^{1/3}}\geq\lambda m^{1/3}\frac{h^{4/3}}{(\varrho-1)^{2/3}}\iff(\varrho-1)^{2/3}\geq m^{1/3}h^{1/3},
\]
the result follows.
\end{proof}
Next, we consider a construction consisting of one wrinkle. 
\begin{lem}
\label{lem:FvKUB_onewrinkle} Assume that 
\[
h^{6/7}\lambda^{5/7}(\varrho-1)^{4/7}\geq\max\{\lambda h,m^{-1/3}(\varrho-1)^{2/3}\lambda h^{2/3}\}.
\]
Let $n=1$ and let $\delta\in(0,1]$ be given by
\[
\delta=4\lambda^{1/7}(\varrho-1)^{-2/7}h^{4/7}.
\]
Then, $\phi_{\delta,n,\lambda,\varrho}\in A_{\lambda,\varrho,m}^{vKD}$
and
\[
E_{h}^{vKD}(\phi_{\delta,n,\lambda,\varrho})-\cE_{b}^{vKD}\lesssim h^{6/7}\lambda^{5/7}(\varrho-1)^{4/7}.
\]
\end{lem}
\begin{proof}
First, we check that $\delta\leq1$. Note that $4\lambda^{1/7}h^{4/7}(\varrho-1)^{-2/7}\leq1$
if and only if $\lambda h^{4}\leq(\varrho-1)^{2}\frac{1}{2^{14}}$.
By assumption, we have that $\lambda h\leq h^{6/7}\lambda^{5/7}(\varrho-1)^{4/7}$
so that $\lambda h^{1/2}\leq(\varrho-1)^{2}$. Since $h\leq\frac{1}{2^{4}}$,
it follows that $h^{4}\leq\frac{1}{2^{14}}h^{1/2}$ and hence that
$\lambda h^{4}\leq\frac{1}{2^{14}}(\varrho-1)^{2}$ as required.

Now we check the slope bounds. We have that
\[
m_{1}(\lambda,\delta)=\max\left\{ \sqrt{2}\lambda^{3/7}(\varrho-1)^{1/7}h^{-2/7},\lambda^{6/7}(\varrho-1)^{2/7}h^{-4/7}\right\} .
\]
By assumption, $m^{-1/3}(\varrho-1)^{2/3}\lambda h^{2/3}\leq h^{6/7}\lambda^{5/7}(\varrho-1)^{4/7}$
so that $(\varrho-1)^{2/7}\lambda^{6/7}h^{-4/7}\leq m$. Since $m\geq2$,
we have that $m^{2}\geq2m$ so that $2(\varrho-1)^{2/7}\lambda^{6/7}h^{-4/7}\leq2m\leq m^{2}$
and hence $\sqrt{2}(\varrho-1)^{1/7}\lambda^{3/7}h^{-2/7}\leq m$.
It follows that $m_{1}(\lambda,\delta)\leq m$. 

Using \prettyref{lem:FvKadmissability}, we conclude that $\phi_{\delta,n,\lambda,\varrho}\in A_{\lambda,\varrho,m}^{vKD}$
and that
\[
E_{h}^{vKD}(\phi)-\cE_{b}^{vKD}\lesssim\max\left\{ (\varrho-1)^{4/7}\lambda^{5/7}h^{6/7},\lambda^{9/7}(\varrho-1)^{-4/7}h^{8/7}\right\} .
\]
Since 
\[
(\varrho-1)^{4/7}\lambda^{5/7}h^{6/7}\geq\lambda^{9/7}(\varrho-1)^{-4/7}h^{8/7}\iff(\varrho-1)^{2}\geq\lambda h^{1/2}
\]
we conclude the desired result.
\end{proof}
The previous two results fail to cover the neutral mandrel case, where
$\varrho=1$. Our next result includes this case.
\begin{lem}
\label{lem:FvKUB_uniforminmandrel} Assume that
\[
\lambda h\geq\max\{m^{-1/3}(\varrho-1)^{2/3}\lambda h^{2/3},h^{6/7}\lambda^{5/7}(\varrho-1)^{4/7}\}.
\]
If $\lambda\leq mh^{1/2}$, then upon taking $n=1$ and $\delta=4h^{1/2}\in(0,1]$
we find that $\phi_{\delta,n,\lambda,\varrho}\in A_{\lambda,\varrho,m}^{vKD}$
and that
\[
E_{h}^{vKD}(\phi_{\delta,n,\lambda,\varrho})-\cE_{b}^{vKD}\lesssim\lambda h.
\]
If $\lambda>mh^{1/2}$, then upon taking $n\in\N$ and $\delta\in(0,1]$
which satisfy 
\[
n\in[\lambda h^{-1/2}m^{-1},2\lambda h^{-1/2}m^{-1}]\quad\text{and}\quad\delta=4\lambda m^{-1},
\]
we find that $\phi_{\delta,n,\lambda,\varrho}\in A_{\lambda,\varrho,m}^{vKD}$
and that
\[
E_{h}^{vKD}(\phi_{\delta,n,\lambda,\varrho})-\cE_{b}^{vKD}\lesssim\lambda h.
\]
\end{lem}
\begin{rem}
\label{rem:unifbddslopes_neutralmandrel} We note here that if $\varrho-1$
is small enough, then the scaling law of $\lambda h$ can be achieved
by a construction with uniformly bounded slopes. Indeed, if one takes
$n\sim h^{-1/2}$ and $\delta=1$, then the resulting $\phi_{\delta,n,\lambda,\varrho}$
belongs to $A_{\lambda,\varrho,m}^{vKD}$ for all $\lambda\in[0,\frac{1}{2}]$
and $m\in[2,\infty]$, and the excess energy is bounded by a multiple
of $\lambda h$ whenever $\varrho-1\leq\lambda^{1/2}h^{1/2}$. \end{rem}
\begin{proof}
We prove this in two parts. Assume first that $\lambda\leq mh^{1/2}$.
Then let $n=1$ and $\delta=4h^{1/2}$. Note that $\delta\in(0,1]$
if and only if $h\leq\frac{1}{2^{4}}$. Also, 
\[
m_{1}(\lambda,\delta)=\max\left\{ 2\sqrt{\frac{2\lambda}{4h^{1/2}}},\frac{4\lambda}{4h^{1/2}}\right\} =\max\left\{ \sqrt{\frac{2\lambda}{h^{1/2}}},\frac{\lambda}{h^{1/2}}\right\} .
\]
Since $m\geq2$, $2m\leq m^{2}$. Thus, $\lambda\leq mh^{1/2}\implies2\lambda\leq2mh^{1/2}\leq m^{2}h^{1/2}$
so that $(2\lambda h^{-1/2})^{1/2}\leq m$. Thus, $m_{1}(\lambda,\delta)\leq m$.
By \prettyref{lem:FvKadmissability}, we have that $\phi_{\delta,n,\lambda,\varrho}\in A_{\lambda,\varrho,m}^{vKD}$
and that 
\[
E_{h}^{vKD}(\phi_{\delta,n,\lambda,\varrho})-\cE_{b}^{vKD}\lesssim\max\left\{ (\varrho-1)\lambda^{1/2}h^{3/4},\lambda h\right\} .
\]
Note that $(\varrho-1)\lambda^{1/2}h^{3/4}\leq\lambda h$ is a rearrangement
of $\lambda h\geq h^{6/7}\lambda^{5/7}(\varrho-1)^{4/7}$. Thus,
\[
E_{h}^{vKD}(\phi_{\delta,n,\lambda,\varrho})-\cE_{b}^{vKD}\lesssim\lambda h.
\]

Now assume that $\lambda>mh^{1/2}$. Let $n\in\N$ and $\delta\in(0,1]$
satisfy 
\[
n\in[\lambda h^{-1/2}m^{-1},2\lambda h^{-1/2}m^{-1}]\quad\text{and}\quad\delta=4\lambda m^{-1}.
\]
Note that $\lambda h^{-1/2}m^{-1}>1$ is a rearrangement of $\lambda>h^{1/2}m$,
so that such an $n$ exists. Also, note that $\delta\leq1$ since
$m\geq2$ and $\lambda\leq\frac{1}{2}$, and that $m_{1}(\delta,\lambda)=m$.
Hence by \prettyref{lem:FvKadmissability}, we have that $\phi_{\delta,n,\lambda,\varrho}\in A_{\lambda,\varrho,m}^{vKD}$
and that 
\[
E_{h}^{vKD}(\phi_{\delta,n,\lambda,\varrho})-\cE_{b}^{vKD}\lesssim\max\left\{ (\varrho-1)\frac{\lambda h^{1/2}}{m^{1/2}},\lambda h\right\} .
\]
Since $(\varrho-1)\frac{\lambda h^{1/2}}{m^{1/2}}\leq\lambda h$ is
a rearrangement of $\lambda h\geq m^{-1/3}(\varrho-1)^{2/3}\lambda h^{2/3}$,
we conclude that 
\[
E_{h}^{vKD}(\phi_{\delta,n,\lambda,\varrho})-\cE_{b}^{vKD}\lesssim\lambda h.
\]

\end{proof}

\subsection{Nonlinear model\label{sub:NLUB}}

Recall the definitions of $E_{h}^{NL}$, $A_{\lambda,\varrho,m}^{NL}$,
and $\cE_{b}^{NL}$, given in \prettyref{eq:ENL}, \prettyref{eq:ANL},
and \prettyref{eq:EbNL}. In this section, we prove the following
upper bound.
\begin{prop}
\label{prop:NLUBs}Let $\varrho_{0}\in[1,\infty)$. Then we have that
\[
\min_{A_{\lambda,\varrho,m}^{NL}}E_{h}^{NL}-\cE_{b}^{NL}\lesssim_{\varrho_{0}}\min\left\{ \lambda^{2},\max\left\{ \lambda h,h^{6/7}\lambda^{5/7}[(\varrho^{2}-1)\vee h^{2}]^{4/7},[(\varrho^{2}-1)\vee h^{2}]^{2/3}\lambda h^{2/3}\right\} \right\} 
\]
whenever $h,\lambda\in(0,\frac{1}{2}]$, $\varrho\in[1,\varrho_{0}]$,
and $m\in[1,\infty]$.\end{prop}
\begin{proof}
Note that since $A_{\lambda,\varrho,m}^{NL}\subset A_{\lambda,\varrho,m'}^{NL}$
if $m\leq m'$, we only need to prove the claim for the case of $m=1$.
The upper bound of $\lambda^{2}$ is achieved by the unbuckled configuration,
$\Phi=(\varrho,\theta,(1-\lambda)z)$. To prove the remainder of the
upper bound, note first that it suffices to achieve it for $(h,\lambda,\varrho)\in(0,h_{0}]\times(0,\frac{1}{2}]\times[1,\varrho_{0}]$
for some $h_{0}\in(0,\frac{1}{2}]$. We apply \prettyref{lem:NLUB_manywrinkles},
\prettyref{lem:NLUB_onewrinkle}, and \prettyref{lem:NLUB_uniforminmandrel}
to deduce the required upper bound in the stated parameter range with
$h_{0}=\frac{1}{4}$. Note that the dependence of the constants in
these lemmas on $f$ can be dropped, since $f$ is fixed in the subsequent
paragraphs.
\end{proof}
In the remainder of this section, we\textbf{ fix} $\varrho_{0}\in[1,\infty)$
as in the claim. Furthermore, we \textbf{assume} that 
\[
h\in(0,\frac{1}{4}],\ \lambda\in(0,\frac{1}{2}],\ \text{and}\quad\varrho\in[1,\varrho_{0}]
\]
unless otherwise explicitly stated. 

As in the analysis of the vKD model, we define a two-scale axisymmetric
wrinkling pattern. We refer to $n\in\N$ and $\delta\in(0,1]$, which
represent the number of wrinkles and their relative extent respectively.
Again, we refer the reader to \prettyref{fig:radialheightfield} for
a schematic of this construction. 

We start by fixing $f\in C^{\infty}(\R)$ such that
\begin{itemize}
\item $f$ is non-negative and one-periodic
\item $\text{supp}\,f\cap[-\frac{1}{2},\frac{1}{2}]\subset(-\frac{1}{2},\frac{1}{2})$
\item $\norm{f'}_{L^{\infty}}<1$
\item $\int_{-\frac{1}{2}}^{\frac{1}{2}}\sqrt{1-f'^{2}}\,dt=\frac{1}{2}$.
\end{itemize}
Define $f_{\delta,n}\in C^{\infty}(\R)$ by
\[
f_{\delta,n}(t)=\frac{\delta}{n}f(\frac{n}{\delta}\{t\})\ind{\{t\}\in B_{\delta/2}}.
\]
Let $S_{f}:[0,1]\to\R$ be defined by
\[
S_{f}(q)=1-\int_{-\frac{1}{2}}^{\frac{1}{2}}\sqrt{1-q^{2}f'^{2}}\,dt,
\]
and observe that $S_{f}$ is a bijection of $[0,1]\leftrightarrow[0,\frac{1}{2}]$.
Hence, \textbf{if} $\delta\in[2\lambda,1]$, we can define $w_{\delta,n,\lambda},u_{\delta,n,\lambda}:\Omega\to\R$
by 
\[
w_{\delta,n,\lambda}(\theta,z)=S_{f}^{-1}\left(\frac{\lambda}{\delta}\right)f_{\delta,n}(z)\quad\text{and}\quad u_{\delta,n,\lambda}(\theta,z)=\int_{-\frac{1}{2}\leq z'\leq z}\sqrt{1-(\partial_{z}w_{\delta,n,\lambda}(\theta,z'))^{2}}-(1-\lambda)\,dz'.
\]
Finally, we define $\Phi_{\delta,n,\lambda,\varrho}:\Omega\to\R^{3}$
by
\[
\Phi_{\delta,n,\lambda,\varrho}=(w_{\delta,n,\lambda}+\varrho,\theta,(1-\lambda)z+u_{\delta,n,\lambda}),
\]
in cylindrical coordinates.

We now estimate the elastic energy of this wrinkling pattern. 
\begin{lem}
\label{lem:NLadmissability} Let $\delta\in[2\lambda,1]$. Then we
have that $\Phi_{\delta,n,\lambda,\varrho}\in A_{\lambda,\varrho,1}^{NL}$.
Furthermore,
\[
E_{h}^{NL}(\Phi_{\delta,n,\lambda,\varrho})-\cE_{b}^{NL}\lesssim_{\varrho_{0},f}\max\left\{ \left[(\varrho^{2}-1)\vee h^{2}\right]\frac{\lambda^{1/2}\delta^{3/2}}{n},\frac{\lambda\delta^{2}}{n^{2}},h^{2}\frac{\lambda n^{2}}{\delta^{2}}\right\} .
\]
\end{lem}
\begin{proof}
Abbreviate $\Phi_{\delta,n,\lambda,\varrho}$ by $\Phi$, $w_{\delta,n,\lambda}$
by $w$, and $u_{\delta,n,\lambda}$ by $u$. By its definition, $\Phi_{\rho}\in H_{\text{per}}^{2}$,
$\Phi_{\theta}-\theta\in H_{\text{per}}^{2}$, and $\Phi_{z}-(1-\lambda)z\in H_{\text{per}}^{2}$.
To see these, note that $w,u\in H_{\text{per}}^{2}$. Indeed, we have that
\begin{align*}
\int_{-\frac{1}{2}}^{\frac{1}{2}}\sqrt{1-(\partial_{z}w(\theta,z))^{2}}\,dz & =\int_{[-\frac{1}{2},\frac{1}{2}]\backslash B_{\delta/2}}1\,dt+\int_{B_{\delta/2}}\sqrt{1-\left(S_{f}^{-1}\left(\frac{\lambda}{\delta}\right)f_{\delta,n}'(t)\right)^{2}}\,dt\\
 & =2(\frac{1}{2}-\frac{\delta}{2})+\delta\int_{-\frac{1}{2}}^{\frac{1}{2}}\sqrt{1-(S_{f}^{-1}\left(\frac{\lambda}{\delta}\right))^{2}\left(f'(t)\right)^{2}}\,dt\\
 & =1-\delta S_{f}\circ S_{f}^{-1}(\frac{\lambda}{\delta})=1-\lambda
\end{align*}
for each $\theta\in I_{\theta}$. Also, we have that $\Phi_{\rho}\geq\varrho$,
since $w\geq0$, and that 
\[
\partial_{z}\Phi_{z}=1-\lambda+\partial_{z}u=\sqrt{1-(\partial_{z}w)^{2}}\geq0.
\]

Now we check the slope bounds. Note that
\[
\partial_{z}\Phi_{\rho}=\partial_{z}w=S_{f}^{-1}\left(\frac{\lambda}{\delta}\right)f_{\delta,n}'(z)
\]
so that 
\[
\norm{\partial_{z}\Phi_{\rho}}_{L^{\infty}}\leq\left|S_{f}^{-1}\left(\frac{\lambda}{\delta}\right)\right|\norm{f_{\delta,n}'}_{L^{\infty}}\leq\norm{f'}_{L^{\infty}}<1.
\]
Also, by the above, we have that
\[
\partial_{z}\Phi_{z}=\sqrt{1-(\partial_{z}w)^{2}}\in[0,1].
\]
Hence, 
\[
\max_{\substack{i\in\left\{ \theta,z\right\} ,\,j\in\{\rho,\theta,z\}}
}\norm{\partial_{i}\Phi_{j}}_{L^{\infty}}\leq1
\]
and it follows that $\Phi\in A_{\lambda,\varrho,1}^{NL}$.

Now we bound the energy of this construction. Since $g_{zz}=1$, $g_{\theta z}=0$,
and $u,w$ are functions of $z$ alone, we have that
\[
E_{h}^{NL}(\Phi)=\int_{\Omega}\,\left|(\varrho+w)^{2}-1\right|^{2}+h^{2}(\left|\varrho+w\right|^{2}+|\partial_{z}^{2}w|^{2}+2|\partial_{z}w|^{2}+|\partial_{z}^{2}u|^{2})\,d\theta dz.
\]
Hence,
\begin{align*}
E_{h}^{NL}(\Phi)-\cE_{b}^{NL}\lesssim_{\varrho_{0}} & \max\{\left[(\varrho^{2}-1)\vee h^{2}\right]\norm{w}_{L^{1}(\Omega)},\norm{w}_{L^{2}(\Omega)}^{2},\\
 & \qquad h^{2}\left(\norm{\partial_{z}^{2}w}_{L^{2}(\Omega)}^{2}\vee\norm{\partial_{z}w}_{L^{2}(\Omega)}^{2}\vee\norm{\partial_{z}^{2}u}_{L^{2}(\Omega)}^{2}\right)\}.
\end{align*}
(Here we used that $\norm{w}_{L^{\infty}}\leq1$, which follows from
its definition and our choice of $f$.) By definition, we have that
\[
\partial_{z}^{2}u=-\frac{\partial_{z}w\partial_{z}^{2}w}{\sqrt{1-(\partial_{z}w)^{2}}}
\]
so that 
\[
\norm{\partial_{z}u}_{L^{2}(\Omega)}\lesssim_{f}\norm{\partial_{z}^{2}w}_{L^{2}(\Omega)}.
\]
Also, we have that
\begin{align*}
 & \norm{w}_{L^{1}(\Omega)}\lesssim S_{f}^{-1}(\frac{\lambda}{\delta})\frac{\delta^{2}}{n},\quad\norm{w}_{L^{2}(\Omega)}^{2}\lesssim\left(S_{f}^{-1}(\frac{\lambda}{\delta})\right)^{2}\frac{\delta^{3}}{n^{2}},\\
 & \norm{\partial_{z}w}_{L^{2}(\Omega)}^{2}\lesssim\left(S_{f}^{-1}(\frac{\lambda}{\delta})\right)^{2}\delta,\ \text{and}\quad\norm{\partial_{z}^{2}w}_{L^{2}(\Omega)}^{2}\lesssim\left(S_{f}^{-1}(\frac{\lambda}{\delta})\right)^{2}\frac{n^{2}}{\delta}.
\end{align*}
Since
\[
\frac{q^{2}}{2}\norm{f'}_{L^{2}([-\frac{1}{2},\frac{1}{2}])}^{2}\leq S_{f}(q)
\]
it follows that 
\[
S_{f}^{-1}(\frac{\lambda}{\delta})\lesssim_{f}\left(\frac{\lambda}{\delta}\right)^{1/2}.
\]
Combining the above, we conclude that
\[
E_{h}^{NL}(\Phi)-\cE_{b}^{NL}\lesssim_{\varrho_{0},f}\max\left\{ \left[(\varrho^{2}-1)\vee h^{2}\right]\frac{\lambda^{1/2}\delta^{3/2}}{n},\frac{\lambda\delta^{2}}{n^{2}},h^{2}\left(\frac{\lambda n^{2}}{\delta^{2}}\vee\lambda\right)\right\} 
\]
and the result immediately follows.
\end{proof}
Next, we choose $n,\delta$ which are optimal for our construction
in various regimes. Our first choice exhibits many wrinkles, and is
the nonlinear analog of \prettyref{lem:FvKUB_manywrinkles}.
\begin{lem}
\label{lem:NLUB_manywrinkles}Assume that 
\[
[(\varrho^{2}-1)\vee h^{2}]^{2/3}\lambda h^{2/3}\geq\max\{\lambda h,h^{6/7}\lambda^{5/7}[(\varrho^{2}-1)\vee h^{2}]^{4/7}\}.
\]
Let $n\in\N$ and $\delta\in(0,1]$ satisfy
\[
n\in\left[[(\varrho^{2}-1)\vee h^{2}]^{1/3}\lambda h^{-2/3},2[(\varrho^{2}-1)\vee h^{2}]^{1/3}\lambda h^{-2/3}\right]\quad\text{and}\quad\delta=2\lambda.
\]
Then, $\Phi_{\delta,n,\lambda,\varrho}\in A_{\lambda,\varrho,1}^{NL}$
and 
\[
E_{h}^{NL}(\Phi_{\delta,n,\lambda,\varrho})-\cE_{b}^{NL}\lesssim_{\varrho_{0},f}[(\varrho^{2}-1)\vee h^{2}]^{2/3}\lambda h^{2/3}.
\]
\end{lem}
\begin{proof}
Rearranging the inequality $[(\varrho^{2}-1)\vee h^{2}]^{2/3}\lambda h^{2/3}\geq h^{6/7}\lambda^{5/7}[(\varrho^{2}-1)\vee h^{2}]^{4/7}$,
we find that $[(\varrho^{2}-1)\vee h^{2}]^{1/3}\lambda h^{-2/3}\geq1$
so that there exists such an $n\in\N$. Also, with our choice of $\delta$
we have that $\delta\in[2\lambda,1]$. It follows immediately from
\prettyref{lem:NLadmissability} that $\Phi_{\delta,n,\lambda,\varrho}\in A_{\lambda,\varrho,1}^{NL}$.
Finally, the bound on the energy follows from \prettyref{lem:NLadmissability}
as in the proof of \prettyref{lem:FvKUB_manywrinkles}, where $\varrho-1$
is replaced by $(\varrho^{2}-1)\vee h^{2}$ and $m$ is replaced by
the number $1$. 
\end{proof}
Next, we consider a pattern consisting of one wrinkle.
\begin{lem}
\label{lem:NLUB_onewrinkle}Assume that 
\[
h^{6/7}\lambda^{5/7}[(\varrho^{2}-1)\vee h^{2}]^{4/7}\geq\max\{\lambda h,[(\varrho^{2}-1)\vee h^{2}]^{2/3}\lambda h^{2/3}\}.
\]
Let $n=1$ and let $\delta\in[2\lambda,1]$ be given by
\[
\delta=2\lambda^{1/7}[(\varrho^{2}-1)\vee h^{2}]^{-2/7}h^{4/7}.
\]
Then, $\Phi_{\delta,n,\lambda,\varrho}\in A_{\lambda,\varrho,1}^{NL}$
and
\[
E_{h}^{NL}(\Phi_{\delta,n,\lambda,\varrho})-\cE_{b}^{NL}\lesssim_{\varrho_{0},f}h^{6/7}\lambda^{5/7}[(\varrho^{2}-1)\vee h^{2}]^{4/7}.
\]
\end{lem}
\begin{proof}
First, we check that $\delta\in[2\lambda,1]$. For the upper bound,
note that $2\lambda^{1/7}[(\varrho^{2}-1)\vee h^{2}]^{-2/7}h^{4/7}\leq1$
if and only if $\lambda h^{4}\leq\frac{1}{2^{7}}[(\varrho^{2}-1)\vee h^{2}]^{2}$.
By assumption, we have that $\lambda h\leq h^{6/7}\lambda^{5/7}[(\varrho^{2}-1)\vee h^{2}]^{4/7}$
so that $\lambda h^{1/2}\leq[(\varrho^{2}-1)\vee h^{2}]^{2}$. Since
$h\leq\frac{1}{4}$, it follows that $h^{4}\leq\frac{1}{2^{7}}h^{1/2}$
and hence that $\lambda h^{4}\leq\frac{1}{2^{7}}[(\varrho^{2}-1)\vee h^{2}]^{2}$
as required. For the lower bound, we note that $2\lambda^{1/7}[(\varrho^{2}-1)\vee h^{2}]^{-2/7}h^{4/7}\geq2\lambda$
if and only if $h^{4}\geq\lambda^{6}[(\varrho^{2}-1)\vee h^{2}]^{2}$.
As this is a rearrangement of $[(\varrho^{2}-1)\vee h^{2}]^{2/3}\lambda h^{2/3}\leq h^{6/7}\lambda^{5/7}[(\varrho^{2}-1)\vee h^{2}]^{4/7}$,
we conclude the lower bound.

It follows from \prettyref{lem:NLadmissability} that $\Phi_{\delta,n,\lambda,\varrho}\in A_{\lambda,\varrho,1}^{NL}$.
The bound on the energy also follows from \prettyref{lem:NLadmissability},
as in the proof of \prettyref{lem:FvKUB_onewrinkle} but where $\varrho-1$
is replaced by $(\varrho^{2}-1)\vee h^{2}$. 
\end{proof}
Finally, we discuss the neutral mandrel case, where $\varrho=1$.
\begin{lem}
\label{lem:NLUB_uniforminmandrel} Assume that
\[
\lambda h\geq\max\{[(\varrho^{2}-1)\vee h^{2}]^{2/3}\lambda h^{2/3},h^{6/7}\lambda^{5/7}[(\varrho^{2}-1)\vee h^{2}]^{4/7}\}.
\]
If $\lambda\leq h^{1/2}$, then upon taking $n=1$ and $\delta=2h^{1/2}\in[2\lambda,1]$
we find that $\Phi_{\delta,n,\lambda,\varrho}\in A_{\lambda,\varrho,1}^{NL}$
and that
\[
E_{h}^{NL}(\Phi_{\delta,n,\lambda,\varrho})-\cE_{b}^{NL}\lesssim_{\varrho_{0},f}\lambda h.
\]
If $\lambda>h^{1/2}$, then upon taking $n\in\N$ and $\delta\in[2\lambda,1]$
which satisfy 
\[
n\in[\lambda h^{-1/2},2\lambda h^{-1/2}]\quad\text{and}\quad\delta=2\lambda,
\]
we find that $\Phi_{\delta,n,\lambda,\varrho}\in A_{\lambda,\varrho,1}^{NL}$
and that
\[
E_{h}^{NL}(\Phi_{\delta,n,\lambda,\varrho})-\cE_{b}^{NL}\lesssim_{\varrho_{0},f}\lambda h.
\]
\end{lem}
\begin{proof}
We prove this in two parts. Assume first that $\lambda\leq h^{1/2}$.
Then let $n=1$ and $\delta=2h^{1/2}$. Note that $\delta\in[2\lambda,1]$
if and only if $h\leq\frac{1}{4}$ and $h^{1/2}\geq\lambda$. It follows
from \prettyref{lem:NLadmissability} that $\Phi_{\delta,n,\lambda,\varrho}\in A_{\lambda,\varrho,1}^{NL}$,
and the bound on the energy follows from \prettyref{lem:NLadmissability}
as in the proof of \prettyref{lem:FvKUB_uniforminmandrel}, where
$\varrho-1$ is replaced by $(\varrho^{2}-1)\vee h^{2}$. 

Now assume that $\lambda>h^{1/2}$. Let $n\in\N$ and $\delta\in[2\lambda,1]$
which satisfy 
\[
n\in[\lambda h^{-1/2},2\lambda h^{-1/2}]\quad\text{and}\quad\delta=2\lambda.
\]
Note that $\lambda h^{-1/2}>1$ is a rearrangement of $\lambda>h^{1/2}$,
so that such an $n$ exists. It follows immediately from \prettyref{lem:NLadmissability}
that $\Phi_{\delta,n,\lambda,\varrho}\in A_{\lambda,\varrho,1}^{NL}$.
The bound on the energy follows from \prettyref{lem:NLadmissability}
as in the proof of \prettyref{lem:FvKUB_uniforminmandrel}, where
$\varrho-1$ is replaced by $(\varrho^{2}-1)\vee h^{2}$ and $m$
is replaced by the number $1$. 
\end{proof}

\section{Ansatz-free lower bounds in the large mandrel case \label{sec:largemandrelLB}}

We turn now to prove the ansatz-free lower bounds from \prettyref{thm:FvKlargemandrelscaling}
and \prettyref{thm:NLlargemandrelscaling}. The key idea behind their
proof is that buckling in the presence of the mandrel requires ``outwards''
displacement, i.e., displacement in the direction of increasing $\rho$,
and that this results in the presence of non-trivial tensile hoop
stresses. This observation leads to lower bounds on $E_{h}^{vKD}$
in \prettyref{sub:largemandrelLB_FvK} and on $E_{h}^{NL}$ in \prettyref{sub:largemandrelLB_nonlinear}.
These bounds are optimal in certain regimes of the form $\varrho-1\geq c_{m}(\lambda,h)>0$
(for the precise statement, we refer the reader to \prettyref{sub:largemandrelresults}
in the introduction).

\subsection{vKD model\label{sub:largemandrelLB_FvK}}

Recall the definitions of $E_{h}^{vKD}$, $A_{\lambda,\varrho,m}^{vKD}$,
and $\cE_{b}^{vKD}$ from \prettyref{eq:EFvK}, \prettyref{eq:AFvK},
and \prettyref{eq:EbFVK}. In \prettyref{sub:largemandrelLBproof},
we prove the following lower bound.
\begin{prop}
\label{prop:FvKlargemandrelLB} We have that
\[
\min\left\{ \max\left\{ m^{-2/3}(\varrho-1)^{2/3}h^{2/3}\lambda,\lambda^{5/7}(\varrho-1)^{4/7}h^{6/7}\right\} ,\lambda^{2}\right\} \lesssim\min_{A_{\lambda,\varrho,m}^{vKD}}E_{h}^{vKD}-\cE_{b}^{vKD}
\]
whenever $h,\lambda\in(0,\infty)$, $\varrho\in[1,\infty)$, and $m\in(0,\infty]$. \end{prop}
\begin{proof}
This follows from \prettyref{cor:FvKLB_largemandr} and \prettyref{cor:FvKLB_largemandrel2},
which combine to prove the equivalent statement that
\[
\min_{A_{\lambda,\varrho,m}^{vKD}}E_{h}^{vKD}-\cE_{b}^{vKD}\gtrsim\max\left\{ \min\{m^{-2/3}(\varrho-1)^{2/3}h^{2/3}\lambda,\lambda^{2}\},\min\{\lambda^{5/7}(\varrho-1)^{4/7}h^{6/7},\lambda^{2}\}\right\} .
\]

\end{proof}
In \prettyref{sub:Blowuprate_largemandrel}, we prove an estimate
on the blow-up rate of $D\phi$ as $h\to0$ for the minimizers of
the $m=\infty$ problem.

\subsubsection{Proof of the ansatz-free lower bound\label{sub:largemandrelLBproof}}

We begin by controlling various features of the radial displacement,
$\phi_{\rho}$. Given $\phi\in A_{\lambda,\varrho,m}^{vKD}$ we call
\[
\Delta^{vKD}=E_{h}^{vKD}(\phi)-\cE_{b}^{vKD},
\]
which is the excess elastic energy in the vKD model.
\begin{lem}
\label{lem:FvKLBs_largemandr} Let $\phi\in A_{\lambda,\varrho,\infty}^{vKD}$.
Then we have that 
\[
\Delta^{vKD}\geq\max\left\{ (\varrho-1)\norm{\phi_{\rho}-(\varrho-1)}_{L^{1}(\Omega)},h^{2}\norm{D^{2}\phi_{\rho}}_{L^{2}(\Omega)}^{2},\norm{\frac{1}{2}\norm{\partial_{z}\phi_{\rho}}_{L_{z}^{2}}^{2}-\lambda}_{L_{\theta}^{2}}^{2}\right\} .
\]
\end{lem}
\begin{proof}
Make the substitution
\[
\phi=(w+\varrho-1,u_{\theta},u_{z}-\lambda z),
\]
given in cylindrical coordinates. By definition, the vKD strain tensor,
$\epsilon$, satisfies
\[
\epsilon_{\theta\theta}=\partial_{\theta}u_{\theta}+\frac{1}{2}(\partial_{\theta}w)^{2}+w+(\varrho-1)\quad\text{and}\quad\epsilon_{zz}=\partial_{z}u_{z}-\lambda+\frac{1}{2}(\partial_{z}w)^{2}.
\]
Since $u_{\theta}\in H_{\text{per}}^{1}$, we have that
\begin{align*}
E_{h}^{vKD}(\phi) & \geq\int_{\Omega}\abs{\epsilon_{\theta\theta}}^{2}+\abs{\epsilon_{zz}}^{2}+h^{2}\abs{D^{2}w}^{2}\\
 & \geq\int_{\Omega}(\varrho-1)^{2}+2(\varrho-1)(\partial_{\theta}u_{\theta}+\frac{1}{2}(\partial_{\theta}w)^{2}+w)+\abs{\epsilon_{zz}}^{2}+h^{2}\abs{D^{2}w}^{2}\\
 & \geq\cE_{b}^{vKD}+\int_{\Omega}2(\varrho-1)w+\abs{\epsilon_{zz}}^{2}+h^{2}\abs{D^{2}w}^{2}.
\end{align*}
Since $w$ is non-negative, we conclude that 
\[
\Delta^{vKD}\geq\max\left\{ 2(\varrho-1)\norm{w}_{L^{1}(\Omega)},\norm{\epsilon_{zz}}_{L^{2}(\Omega)}^{2},h^{2}\norm{D^{2}w}_{L^{2}(\Omega)}^{2}\right\} .
\]
By applying Jensen's inequality and using that $u_{z}\in H_{\text{per}}^{1}$,
it follows that
\[
\norm{\epsilon_{zz}}_{L^{2}(\Omega)}^{2}\geq\frac{1}{\abs{I_{z}}}\int_{I_{\theta}}\abs{\int_{I_{z}}\epsilon_{zz}\,dz}^{2}\,d\theta=\frac{1}{\abs{I_{z}}}\norm{\frac{1}{2}\norm{\partial_{z}w}_{L_{z}^{2}}^{2}-\lambda}_{L_{\theta}^{2}}^{2}.
\]
Since $\abs{I_{z}}=1$, the result follows.
\end{proof}
Now, we will apply the Gagliardo-Nirenberg interpolation inequalities
from \prettyref{sec:Appendix} to deduce the desired lower bounds.
\begin{cor}
\label{cor:FvKLB_largemandr} If $\phi\in A_{\lambda,\varrho,m}^{vKD}$,
then
\[
\Delta^{vKD}\gtrsim\min\{m^{-2/3}(\varrho-1)^{2/3}h^{2/3}\lambda,\lambda^{2}\}.
\]
In fact, if $\phi\in A_{\lambda,\varrho,\infty}^{vKD}$, then 
\[
\Delta^{vKD}\gtrsim\min\{\norm{D\phi_{\rho}}_{L^{\infty}}^{-2/3}(\varrho-1)^{2/3}h^{2/3}\lambda,\lambda^{2}\}.
\]
\end{cor}
\begin{proof}
Observe that by \prettyref{lem:FvKLBs_largemandr} and an application
of H\"older's inequality, we have that 
\[
(\Delta^{vKD})^{1/2}\geq\abs{I_{z}}^{-1/2}|I_{\theta}|^{-1/2}\norm{\frac{1}{2}\norm{\partial_{z}\phi_{\rho}}_{L_{z}^{2}}^{2}-\lambda}_{L_{\theta}^{1}}.
\]
Hence, by the triangle inequality,
\[
\frac{1}{2}\norm{\partial_{z}\phi_{\rho}}_{L^{2}(\Omega)}^{2}+|\Omega|^{1/2}(\Delta^{vKD})^{1/2}\geq\lambda|I_{\theta}|.
\]
Now we perform a case analysis. If $\phi$ satisfies $\norm{\partial_{z}\phi_{\rho}}_{L^{2}(\Omega)}^{2}\leq\lambda|I_{\theta}|$,
then we conclude by the above that $\Delta^{vKD}\gtrsim\lambda^{2}$.

If, on the other hand, $\phi$ satisfies $\norm{\partial_{z}\phi_{\rho}}_{L^{2}(\Omega)}^{2}>\lambda|I_{\theta}|$,
then we can combine the interpolation inequality from \prettyref{lem:2dinterp-1}
(applied to $f=\phi_{\rho}-(\varrho-1)$) with \prettyref{lem:FvKLBs_largemandr}
to conclude that
\[
\lambda\lesssim\norm{D\phi_{\rho}}_{L^{\infty}(\Omega)}^{2/3}\left(\frac{1}{\varrho-1}\Delta^{vKD}\right)^{2/3}\left(\frac{1}{h^{2}}\Delta^{vKD}\right)^{1/3}\lesssim m^{2/3}(\varrho-1)^{-2/3}h^{-2/3}\Delta^{vKD}.
\]
These observations combine to prove the desired result.\end{proof}
\begin{cor}
\label{cor:FvKLB_largemandrel2} If $\phi\in A_{\lambda,\varrho,m}^{vKD}$,
then
\[
\Delta^{vKD}\gtrsim\min\{\lambda^{5/7}(\varrho-1)^{4/7}h^{6/7},\lambda^{2}\}.
\]
\end{cor}
\begin{proof}
Evidently, it suffices to prove that 
\[
\Delta^{vKD}\leq|I_{\theta}|\lambda^{2}\implies\Delta^{vKD}\gtrsim\lambda^{5/7}(\varrho-1)^{4/7}h^{6/7}.
\]
Assume that $\Delta^{vKD}\leq|I_{\theta}|\lambda^{2}$, and define
the set 
\[
Z=\left\{ \theta\in I_{\theta}\ :\ \abs{\frac{1}{2}\norm{\partial_{z}\phi_{\rho}}_{L_{z}^{2}}^{2}-\lambda}\geq\sqrt{2}\lambda\right\} .
\]
We claim that $|I_{\theta}\backslash Z|\geq\frac{1}{2}|I_{\theta}|$.
Indeed, by Chebyshev's inequality and \prettyref{lem:FvKLBs_largemandr},
we have that
\[
2\lambda^{2}\abs{Z}\leq\norm{\frac{1}{2}\norm{\partial_{z}\phi_{\rho}}_{L_{z}^{2}}^{2}-\lambda}_{L_{\theta}^{2}}^{2}\leq|I_{\theta}|\lambda^{2}
\]
so that $|Z|\leq\frac{1}{2}|I_{\theta}|$ as desired. It follows that
\[
\lambda^{5/7}\abs{I_{\theta}}\lesssim\int_{I_{\theta}\backslash Z}\norm{\partial_{z}\phi_{\rho}}_{L_{z}^{2}}^{10/7}\,d\theta\leq\int_{I_{\theta}}\norm{\partial_{z}\phi_{\rho}}_{L_{z}^{2}}^{10/7}\,d\theta.
\]

Applying the first interpolation inequality from \prettyref{lem:1dinterp}
to $f=\phi_{\rho}-(\varrho-1)$, we conclude that 
\[
\lambda^{5/7}\abs{I_{\theta}}\lesssim\int_{I_{\theta}}\norm{f}_{L_{z}^{1}}^{4/7}\norm{\partial_{z}^{2}f}_{L_{z}^{2}}^{6/7}\,d\theta\leq\norm{\phi_{\rho}-(\varrho-1)}_{L^{1}(\Omega)}^{4/7}\norm{D^{2}\phi_{\rho}}_{L^{2}(\Omega)}^{6/7}.
\]
Note that we used H\"older's inequality in the second step. Finally,
\prettyref{lem:FvKLBs_largemandr} proves that 
\[
\lambda^{5/7}\lesssim\left(\frac{1}{\varrho-1}\Delta^{vKD}\right)^{4/7}\left(\frac{1}{h^{2}}\Delta^{vKD}\right)^{3/7}=(\varrho-1)^{-4/7}h^{-6/7}\Delta^{vKD}
\]
and the lower bound follows.
\end{proof}

\subsubsection{Blow-up rate of $D\phi$ as $h\to0$\label{sub:Blowuprate_largemandrel}}

We can now make \prettyref{rem:FvKslopeexlposion} precise, regarding
the claim that $E_{h}^{vKD}$ prefers exploding slopes in the limit
$h\to0$. The following result can be seen to justify the introduction
of the parameter $m$ in the definition of the admissible set, $A_{\lambda,\varrho,m}^{vKD}$. 
\begin{cor}
Let $\left\{ (h_{\alpha},\lambda_{\alpha},\varrho_{\alpha})\right\} _{\alpha\in\R_{+}}$
be such that $h_{\alpha},\lambda_{\alpha}\in(0,\frac{1}{2}]$ and
$\varrho_{\alpha}\geq1+\lambda_{\alpha}^{1/2}h_{\alpha}^{1/4}$. Assume
that $h_{\alpha}\ll(\varrho_{\alpha}-1)^{-2/3}\lambda_{\alpha}^{3/2}$
as $\alpha\to\infty$, and let $\{\phi^{\alpha}\}_{\alpha\in\R_{+}}$
satisfy 
\[
\phi^{\alpha}\in A_{\lambda_{\alpha},\varrho_{\alpha},\infty}^{vKD}\quad\text{and}\quad E_{h_{\alpha}}^{vKD}(\phi^{\alpha})=\min_{A_{\lambda_{\alpha},\varrho_{\alpha},\infty}^{vKD}}E_{h_{\alpha}}^{vKD}.
\]
 Then we have that 
\[
(\varrho_{\alpha}-1)^{1/7}h_{\alpha}^{-2/7}\lambda_{\alpha}^{3/7}\lesssim\norm{D\phi_{\rho}^{\alpha}}_{L^{\infty}}\quad\text{as}\ \alpha\to\infty.
\]
\end{cor}
\begin{proof}
For ease of notation, we omit the index $\alpha$ in what follows.
By \prettyref{prop:FvKUB} we have that
\[
E_{h}^{vKD}(\phi)-\cE_{b}^{vKD}\lesssim h^{6/7}\lambda^{5/7}(\varrho-1)^{4/7}.
\]
Hence, by \prettyref{cor:FvKLB_largemandr}, it follows that 
\[
\lambda^{2}\lesssim h^{6/7}\lambda^{5/7}(\varrho-1)^{4/7}\quad\text{or}\quad\norm{D\phi_{\rho}}_{L^{\infty}}^{-2/3}(\varrho-1)^{2/3}h^{2/3}\lambda\lesssim h^{6/7}\lambda^{5/7}(\varrho-1)^{4/7}.
\]
Rearranging, we have that
\[
h\gtrsim(\varrho-1)^{-2/3}\lambda^{3/2}\quad\text{or}\quad(\varrho-1)^{1/7}h^{-2/7}\lambda^{3/7}\lesssim\norm{D\phi_{\rho}}_{L^{\infty}}.
\]
By assumption the first inequality does not hold, and the result follows. 
\end{proof}

\subsection{Nonlinear model\label{sub:largemandrelLB_nonlinear}}

Recall the definitions of $E_{h}^{NL}$, $A_{\lambda,\varrho,m}^{NL}$,
and $\cE_{b}^{NL}$ given in \prettyref{eq:ENL}, \prettyref{eq:ANL},
and \prettyref{eq:EbNL}. In this section, we prove the following
lower bound.
\begin{prop}
\label{prop:NLLBs_largemandr}Let $\varrho_{0}\in[1,\infty)$. Then
we have that
\[
\min\left\{ \max\left\{ \left[(\varrho^{2}-1)\vee h^{2}\right]^{2/3}h^{2/3}\lambda,\lambda^{5/7}[(\varrho^{2}-1)\vee h^{2}]^{4/7}h^{6/7}\right\} ,\lambda^{2}\right\} \lesssim_{m,\varrho_{0}}\min_{A_{\lambda,\varrho,m}^{NL}}E_{h}^{NL}-\cE_{b}^{NL}
\]
whenever $h,\lambda\in(0,1]$, $\varrho\in[1,\varrho_{0}]$, and $m\in(0,\infty)$.
\end{prop}
The reader may notice that, although it is certainly more involved, the following argument shares the same overall structure as the one given for the vKD model in \prettyref{sub:largemandrelLB_FvK}. For more on this, we refer to the discussion in \prettyref{sub:DiscussionofProofs}.

In the remainder of this section, we \textbf{assume} that 
\[
0<h,\lambda\leq1,\quad1\leq\varrho\leq\varrho_{0}<\infty,\ \text{and}\quad0<m<\infty.
\]
Given $\Phi\in A_{\lambda,\varrho,m}^{NL}$ we call
\begin{equation}
\Delta^{NL}=E_{h}^{NL}(\Phi)-\cE_{b}^{NL},\label{eq:NLexcess}
\end{equation}
which is the excess elastic energy in the nonlinear model. Observe
we may \textbf{assume} that 
\[
\Phi\ \text{satisfies}\ \Delta^{NL}\leq1,
\]
since otherwise the desired bound is clear. As the reader will note,
this assumption simplifies the discussion throughout.

We will make frequent use of the following identities concerning the
components of the metric tensor, $g=D\Phi^{T}D\Phi$, in $(\theta,z)$-coordinates:
\begin{align}\label{eq:g_components}
\begin{split}
g_{\theta\theta}&=\left(\partial_{\theta}\Phi_{\rho}\right)^{2}+\Phi_{\rho}^{2}\left(\partial_{\theta}\Phi_{\theta}\right)^{2}+\left(\partial_{\theta}\Phi_{z}\right)^{2}\\
g_{zz}&=\left(\partial_{z}\Phi_{\rho}\right)^{2}+\Phi_{\rho}^{2}\left(\partial_{z}\Phi_{\theta}\right)^{2}+\left(\partial_{z}\Phi_{z}\right)^{2}\\
g_{\theta z}&=\partial_{\theta}\Phi_{\rho}\partial_{z}\Phi_{\rho}+\Phi_{\rho}^{2}\partial_{\theta}\Phi_{\theta}\partial_{z}\Phi_{\theta}+\partial_{\theta}\Phi_{z}\partial_{z}\Phi_{z}
\end{split}
\end{align}We will also make use of the following identities concerning the components
of $D^{2}\Phi$ in $(\theta,z)$-coordinates: \begin{align}\label{eq:DDPhi}
\begin{split}
\partial_{\theta}^{2}\Phi&=(\partial_{\theta}^{2}\Phi_{\rho}-\Phi_{\rho}(\partial_{\theta}\Phi_{\theta})^{2})E_{\rho}(\Phi)+(2\partial_{\theta}\Phi_{\rho}\partial_{\theta}\Phi_{\theta}+\Phi_{\rho}\partial_{\theta}^{2}\Phi_{\theta})E_{\theta}(\Phi)+\partial_{\theta}^{2}\Phi_{z}E_{z}\\
\partial_{z}^{2}\Phi&=(\partial_{z}^{2}\Phi_{\rho}-\Phi_{\rho}(\partial_{z}\Phi_{\theta})^{2})E_{\rho}(\Phi)+(2\partial_{z}\Phi_{\rho}\partial_{z}\Phi_{\theta}+\Phi_{\rho}\partial_{z}^{2}\Phi_{\theta})E_{\theta}(\Phi)+\partial_{z}^{2}\Phi_{z}E_{z}\\
\partial_{\theta z}\Phi&=(\partial_{\theta z}\Phi_{\rho}-\Phi_{\rho}\partial_{\theta}\Phi_{\theta}\partial_{z}\Phi_{\theta})E_{\rho}(\Phi)+(\partial_{\theta}\Phi_{\rho}\partial_{z}\Phi_{\theta}+\partial_{\theta}\Phi_{\theta}\partial_{z}\Phi_{\rho}+\Phi_{\rho}\partial_{\theta z}\Phi_{\theta})E_{\theta}(\Phi)+\partial_{\theta z}\Phi_{z}E_{z}
\end{split}
\end{align}Here, $\{E_{i}\}_{i\in\{\rho,\theta,z\}}$ denotes the unit frame
of coordinate vectors for the cylindrical $\rho,\theta,z$-coordinates
on $\R^{3}$ (as defined in \prettyref{sub:notation}).

\subsubsection{Controlling the radial deformation}

We begin by proving that the excess energy controls the membrane and
bending terms individually.
\begin{lem}
\label{lem:NLexcesssplits} If $\Phi\in A_{\lambda,\varrho,\infty}^{NL}$,
then 
\begin{align*}
\Delta^{NL} & \geq\max\left\{ \int_{\Omega}\abs{g_{\theta\theta}-1}^{2}-(\varrho^{2}-1)^{2},\norm{g_{\theta z}}_{L^{2}(\Omega)}^{2},\norm{g_{zz}-1}_{L^{2}(\Omega)}^{2}\right\} \\
\Delta^{NL} & \geq h^{2}\max\left\{ \int_{\Omega}\abs{\partial_{\theta}^{2}\Phi}^{2}-\varrho^{2},\norm{\partial_{\theta z}\Phi}_{L^{2}(\Omega)}^{2},\norm{\partial_{z}^{2}\Phi}_{L^{2}(\Omega)}^{2}\right\} .
\end{align*}
\end{lem}
\begin{proof}
By the definition of $\Delta^{NL}$ in \prettyref{eq:NLexcess}, it
suffices to prove the following two inequalities to conclude the result:
\[
\int_{\Omega}\abs{g_{\theta\theta}-1}^{2}-(\varrho^{2}-1)^{2}\geq0\quad\text{and}\quad\int_{\Omega}\abs{\partial_{\theta}^{2}\Phi}^{2}-\varrho^{2}\geq0.
\]
To see the first inequality, we begin by noting that
\begin{equation}
(g_{\theta\theta}-1)^{2}-(\varrho^{2}-1)^{2}=2(\varrho^{2}-1)(g_{\theta\theta}-\varrho^{2})+(g_{\theta\theta}-\varrho^{2})^{2}\label{eq:excessthetaenergy}
\end{equation}
and 
\begin{equation}
g_{\theta\theta}-\varrho^{2}=(\partial_{\theta}\Phi_{\rho})^{2}+\Phi_{\rho}^{2}(\partial_{\theta}\Phi_{\theta})^{2}+(\partial_{\theta}\Phi_{z})^{2}-\varrho^{2}\label{eq:excessthetastrain}
\end{equation}
by \prettyref{eq:g_components}. It follows that 
\begin{equation}
(g_{\theta\theta}-1)^{2}-(\varrho^{2}-1)^{2}\geq2(\varrho^{2}-1)(\Phi_{\rho}^{2}(\partial_{\theta}\Phi_{\theta})^{2}-\varrho^{2}+(\partial_{\theta}\Phi_{\rho})^{2}+(\partial_{\theta}\Phi_{z})^{2}).\label{eq:excessenergylarge}
\end{equation}
Using the hypothesis that $\Phi_{\rho}\geq\varrho$ and applying Jensen's
inequality, we see that 
\begin{equation}
\int_{\Omega}\Phi_{\rho}^{2}(\partial_{\theta}\Phi_{\theta})^{2}-\varrho^{2}\geq\frac{\varrho^{2}}{\left|\Omega\right|}\left(\left(\int_{\Omega}\partial_{\theta}\Phi_{\theta}\right)^{2}-\left|\Omega\right|^{2}\right)=\frac{\varrho^{2}}{\left|\Omega\right|}\left(\left|\Omega\right|^{2}-\left|\Omega\right|^{2}\right)=0.\label{eq:excessenergylarge-1}
\end{equation}
Since $\varrho\geq1$, the first inequality follows.

To see the second inequality, note that by \prettyref{eq:DDPhi} we
have that 
\[
\abs{\partial_{\theta}^{2}\Phi}\geq\abs{\partial_{\theta}^{2}\Phi_{\rho}-\Phi_{\rho}(\partial_{\theta}\Phi_{\theta})^{2}}.
\]
Hence, by Jensen's inequality and since $\Phi_{\rho}\in H_{\text{per}}^{2}$,
it follows that 
\[
\int_{\Omega}\abs{\partial_{\theta}^2\Phi}^{2}-\varrho^{2}\geq\frac{1}{\abs{\Omega}}\left(\int_{\Omega}\partial_{\theta}^2\Phi_{\rho}-\Phi_{\rho}(\partial_{\theta}\Phi_{\theta})^{2}\right)^{2}-\abs{\Omega}\varrho^{2}=\frac{1}{\abs{\Omega}}\left(\int_{\Omega}\Phi_{\rho}(\partial_{\theta}\Phi_{\theta})^{2}\right)^{2}-\abs{\Omega}\varrho^{2}.
\]
Using that $\Phi_{\rho}\geq\varrho$ and applying Jensen's inequality
again, we conclude that 
\[
\int_{\Omega}\abs{\partial_{\theta}^2\Phi}^{2}-\varrho^{2}\geq\frac{\varrho^{2}}{|\Omega|}\left(\left(\int_{\Omega}(\partial_{\theta}\Phi_{\theta})^{2}\right)^{2}-\abs{\Omega}^{2}\right)\geq\frac{\varrho^{2}}{|\Omega|}\left(|\Omega|^{2}-|\Omega|^{2}\right)=0
\]
as desired.
\end{proof}
Next, we establish control on the radial component of the deformation,
$\Phi_{\rho}$. As we will require the uniform-in-mandrel estimates
from this result to complete the proof of \prettyref{prop:NLLBs_largemandr},
we record these alongside the large mandrel estimates now.
\begin{lem}
\label{lem:NLmembrLB}Let $\Phi\in A_{\lambda,\varrho,\infty}^{NL}$.
Then we have that
\begin{align*}
\Delta^{NL} & \gtrsim(\varrho^{2}-1)\max\{\norm{\Phi_{\rho}-\varrho}_{L^{1}(\Omega)},\norm{\partial_{\theta}\Phi_{\rho}}_{L^{2}(\Omega)}^{2},\norm{\partial_{\theta}\Phi_{\theta}-1}_{L^{2}(\Omega)}^{2},\norm{\partial_{\theta}\Phi_{z}}_{L^{2}(\Omega)}^{2}\}\\
(\Delta^{NL})^{1/2} & \gtrsim\max\left\{ \norm{\Phi_{\rho}-\varrho}_{L_{z}^{2}L_{\theta}^{1}},\norm{\partial_{\theta}\Phi_{\rho}}_{L_{z}^{4}L_{\theta}^{2}}^{2},\norm{\partial_{\theta}\Phi_{\theta}-1}_{L_{z}^{4}L_{\theta}^{2}}^{2},\norm{\partial_{\theta}\Phi_{z}}_{L_{z}^{4}L_{\theta}^{2}}^{2}\right\} .
\end{align*}
\end{lem}
\begin{proof}
We begin by proving the first estimate. Recall \prettyref{lem:NLexcesssplits}
and equations \prettyref{eq:excessenergylarge} and \prettyref{eq:excessenergylarge-1}.
Altogether, these imply that 
\begin{equation}
\Delta^{NL}\geq2(\varrho^{2}-1)\max\left\{ \int_{\Omega}\Phi_{\rho}^{2}(\partial_{\theta}\Phi_{\theta})^{2}-\varrho^{2},\norm{\partial_{\theta}\Phi_{\rho}}_{L^{2}(\Omega)}^{2},\norm{\partial_{\theta}\Phi_{z}}_{L^{2}(\Omega)}^{2}\right\} .\label{eq:NLmembrLB-1}
\end{equation}
Introduce the displacements $\phi_{\rho}=\Phi_{\rho}-\varrho$ and
$\phi_{\theta}=\Phi_{\theta}-\theta$. In these variables, 
\begin{equation}
\Phi_{\rho}^{2}(\partial_{\theta}\Phi_{\theta})^{2}-\varrho^{2}\geq\varrho^{2}\left(2\partial_{\theta}\phi_{\theta}+(\partial_{\theta}\phi_{\theta})^{2}\right)+2\varrho\phi_{\rho}(\partial_{\theta}\phi_{\theta}+1)^{2}.\label{eq:NLmembrLB-0}
\end{equation}
Since the second term is non-negative, and since $\phi_{\theta}\in H_{\text{per}}^{2}$
and $\varrho\geq1$, we conclude from \prettyref{eq:NLmembrLB-0}
that
\begin{equation}
I:=\int_{\Omega}\Phi_{\rho}^{2}(\partial_{\theta}\Phi_{\theta})^{2}-\varrho^{2}\geq\int_{\Omega}\varrho^{2}\left(2\partial_{\theta}\phi_{\theta}+(\partial_{\theta}\phi_{\theta})^{2}\right)\geq\norm{\partial_{\theta}\phi_{\theta}}_{L^{2}(\Omega)}^{2}.\label{eq:NLmembrLB-2}
\end{equation}
In a similar manner, we can conclude from \prettyref{eq:NLmembrLB-0}
that 
\[
I\geq\int_{\Omega}2\varrho\phi_{\rho}(\partial_{\theta}\phi_{\theta}+1)^{2}\geq\int_{\Omega}\phi_{\rho}(2\partial_{\theta}\phi_{\theta}+1)
\]
and, since $\phi_{\rho}\geq0$, that 
\[
I+\left|\int_{\Omega}\phi_{\rho}\partial_{\theta}\phi_{\theta}\right|\gtrsim\norm{\phi_{\rho}}_{L^{1}(\Omega)}.
\]

Recall the notation $\overline{f}$ for the $\theta$-average of a
function $f$, introduced in \prettyref{sub:notation}. Integrating
by parts and applying Poincare's inequality, we see that 
\begin{align*}
\abs{\int_{\Omega}\phi_{\rho}\partial_{\theta}\phi_{\theta}} & =\abs{\int_{\Omega}\partial_{\theta}\phi_{\rho}(\phi_{\theta}-\overline{\phi_{\theta}})}\leq\norm{\partial_{\theta}\phi_{\rho}}_{L^{2}(\Omega)}\norm{\phi_{\theta}-\overline{\phi_{\theta}}}_{L^{2}(\Omega)}\\
 & \lesssim\norm{\partial_{\theta}\phi_{\rho}}_{L^{2}(\Omega)}\norm{\partial_{\theta}\phi_{\theta}}_{L^{2}(\Omega)}.
\end{align*}
Hence, 
\begin{equation}
I+\norm{\partial_{\theta}\phi_{\rho}}_{L^{2}(\Omega)}\norm{\partial_{\theta}\phi_{\theta}}_{L^{2}(\Omega)}\gtrsim\norm{\phi_{\rho}}_{L^{1}(\Omega)}.\label{eq:NLmembrLB-3}
\end{equation}
Combining \prettyref{eq:NLmembrLB-1}, \prettyref{eq:NLmembrLB-2},
and \prettyref{eq:NLmembrLB-3} gives the required bound.

We turn now to prove the second estimate. First, we observe that by
\prettyref{eq:excessthetastrain} and \prettyref{eq:excessenergylarge-1},
\[
\int_{\Omega}g_{\theta\theta}-\varrho^{2}\geq\int_{\Omega}\Phi_{\rho}^{2}(\partial_{\theta}\Phi_{\theta})^{2}-\varrho^{2}\geq0.
\]
Hence, by \prettyref{lem:NLexcesssplits}, \prettyref{eq:excessthetaenergy},
and since $\varrho\geq1$, we have that
\[
\Delta^{NL}\geq\int_{\Omega}|g_{\theta\theta}-1|^{2}-(\varrho^{2}-1)^{2}\geq\int_{\Omega}(g_{\theta\theta}-\varrho^{2})^{2}.
\]
Applying Jensen's inequality along the slices $\{z\}\times I_{\theta}$,
we find that 
\begin{equation}
(\Delta^{NL})^{1/2}\gtrsim\norm{\overline{g_{\theta\theta}-\varrho^{2}}}_{L_{z}^{2}}.\label{eq:NLmembrLB-4}
\end{equation}
Now we estimate the integrand in the line above. It follows from \prettyref{eq:excessthetastrain}
that

\[
\overline{g_{\theta\theta}-\varrho^{2}}\geq\max\left\{ \overline{\Phi_{\rho}^{2}(\partial_{\theta}\Phi_{\theta})^{2}-\varrho^{2}},\norm{\partial_{\theta}\Phi_{\rho}}_{L_{\theta}^{2}}^{2},\norm{\partial_{\theta}\Phi_{z}}_{L_{\theta}^{2}}^{2}\right\} 
\]
for a.e.\ $z\in I_{z}$. Here we used that 
\[
II=\overline{\Phi_{\rho}^{2}(\partial_{\theta}\Phi_{\theta})^{2}-\varrho^{2}}\geq0
\]
for a.e.\ $z\in I_{z}$, which follows from Jensen's inequality (as
in the proof of \prettyref{eq:excessenergylarge-1}).

Now, we apply the same reasoning to $II$ as for $I$ above. The analog
of \prettyref{eq:NLmembrLB-2} is that 
\[
II\geq\norm{\partial_{\theta}\phi_{\theta}}_{L_{\theta}^{2}}^{2}\quad a.e.,
\]
and this is implied by \prettyref{eq:NLmembrLB-0}. The analog of
\prettyref{eq:NLmembrLB-3} is that 
\[
II+\norm{\partial_{\theta}\phi_{\rho}}_{L_{\theta}^{2}}\norm{\partial_{\theta}\phi_{\theta}}_{L_{\theta}^{2}}\gtrsim\norm{\phi_{\rho}}_{L_{\theta}^{1}}\quad a.e.
\]
This also follows from \prettyref{eq:NLmembrLB-0}, by an integration
by parts argument and Poincare's inequality. It follows that
\[
\overline{g_{\theta\theta}-\varrho^{2}}\gtrsim\max\left\{ \norm{\phi_{\rho}}_{L_{\theta}^{1}},\norm{\partial_{\theta}\phi_{\theta}}_{L_{\theta}^{2}}^{2},\norm{\partial_{\theta}\Phi_{\rho}}_{L_{\theta}^{2}}^{2},\norm{\partial_{\theta}\Phi_{z}}_{L_{\theta}^{2}}^{2}\right\} \quad a.e.
\]
Combining this with \prettyref{eq:NLmembrLB-4} proves the required
bound. 
\end{proof}
Now, we turn to quantify the observation that if $\lambda$ is large
enough, the cylinder should buckle. 
\begin{lem}
\label{lem:bucklingestimate_witherror}Let $\Phi\in A_{\lambda,1,\infty}^{NL}$.
Then we have that 
\[
\lambda\abs{A}\lesssim\max\{\int_{A}\norm{\partial_{z}\Phi_{\rho}}_{L_{z}^{2}}^{2}\,d\theta,(\Delta^{NL})^{1/2},\norm{\Phi_{\rho}\partial_{z}\Phi_{\theta}}_{L^{2}(\Omega)}^{2}\}
\]
 for all $A\in\cB(I_{\theta})$.\end{lem}
\begin{rem}
\label{rem:invertibilityhypothesis}It is precisely in the proof of
this lemma where the hypothesis on the sign of $\partial_{z}\Phi_{z}$
from the definition of $A_{\lambda,\varrho,m}^{NL}$ is used. We note
that this can be relaxed, the crucial hypothesis being that $\partial_{z}\Phi_{z}$
``stays away'' from the well at $-1$. Indeed, the lemma would remain
true if the statement that $\partial_{z}\Phi_{z}\geq0$ from \prettyref{eq:ANL}
were replaced with the statement that there exists a constant $c>0$
such that $|\partial_{z}\Phi_{z}+1|\geq c>1$. \end{rem}
\begin{proof}
Since $\Phi\in A_{\lambda,1,\infty}^{NL}$, we have that
\[
\int_{I_{z}}\partial_{z}\Phi_{z}-1\,dz=1-\lambda-1=-\lambda
\]
for a.e.\ $\theta\in I_{\theta}$. Since we have assumed that $\partial_{z}\Phi_{z}\geq0$
a.e., it follows that 
\[
\lambda\leq\int_{I_{z}}|\partial_{z}\Phi_{z}-1||1+\partial_{z}\Phi_{z}|\,dz=\norm{(\partial_{z}\Phi_{z})^{2}-1}_{L_{z}^{1}}
\]
for a.e.\ $\theta\in I_{\theta}$. By the identity for $g_{zz}$
in \prettyref{eq:g_components}, we see that
\[
\lambda\leq\norm{g_{zz}-1}_{L_{z}^{1}}+\norm{\partial_{z}\Phi_{\rho}}_{L_{z}^{2}}^{2}+\norm{\Phi_{\rho}\partial_{z}\Phi_{\theta}}_{L_{z}^{2}}^{2}.
\]
Now the result follows from \prettyref{lem:NLexcesssplits} by an
application of H\"older's inequality.
\end{proof}
Now we control the cross-term, $\Phi_{\rho}\partial_{z}\Phi_{\theta}$. 
\begin{lem}
\label{lem:NLcross-term}Let $\Phi\in A_{\lambda,\varrho,m}^{NL}.$
Then we have that
\[
\norm{\Phi_{\rho}\partial_{z}\Phi_{\theta}}_{L^{2}(\Omega)}\lesssim_{\varrho_{0},m}(\Delta^{NL})^{1/4}.
\]
\end{lem}
\begin{proof}
Since $\Phi_{\rho}\geq1$, we have that 
\[
\abs{\Phi_{\rho}\partial_{z}\Phi_{\theta}}\leq\abs{\Phi_{\rho}\partial_{z}\Phi_{\theta}\partial_{\theta}\Phi_{\theta}}+\abs{\Phi_{\rho}\partial_{z}\Phi_{\theta}(\partial_{\theta}\Phi_{\theta}-1)}\leq\Phi_{\rho}^{2}\abs{\partial_{z}\Phi_{\theta}\partial_{\theta}\Phi_{\theta}}+\abs{\Phi_{\rho}}\abs{\partial_{z}\Phi_{\theta}}\abs{\partial_{\theta}\Phi_{\theta}-1}.
\]
From the definition of $g_{\theta z}$ in \prettyref{eq:g_components},
we see that 
\[
\Phi_{\rho}^{2}\abs{\partial_{z}\Phi_{\theta}\partial_{\theta}\Phi_{\theta}}\leq\abs{g_{\theta z}}+\abs{\partial_{\theta}\Phi_{\rho}}\abs{\partial_{z}\Phi_{\rho}}+\abs{\partial_{\theta}\Phi_{z}}\abs{\partial_{z}\Phi_{z}}.
\]
Using a Lipschitz bound along with \prettyref{lem:NLmembrLB} and
H\"older's inequality, we see that
\begin{align*}
\norm{\Phi_{\rho}}_{L^{\infty}(\Omega)} & \lesssim\norm{\Phi_{\rho}}_{L^{1}(\Omega)}+\norm{D\Phi_{\rho}}_{L^{\infty}(\Omega)}\leq\varrho|\Omega|+\norm{\Phi_{\rho}-\varrho}_{L^{1}(\Omega)}+\norm{D\Phi_{\rho}}_{L^{\infty}(\Omega)}\\
 & \lesssim\varrho+(\Delta^{NL})^{1/2}+\norm{D\Phi_{\rho}}_{L^{\infty}(\Omega)}.
\end{align*}
Combining the above with the definition of $A_{\lambda,\varrho,m}^{NL}$
and the hypotheses that $\varrho\leq\varrho_{0}$ and $\Delta^{NL}\leq1$
gives that 
\[
\abs{\Phi_{\rho}\partial_{z}\Phi_{\theta}}\lesssim_{\varrho_{0},m}\max\{\abs{g_{\theta z}},\abs{\partial_{\theta}\Phi_{\rho}},\abs{\partial_{\theta}\Phi_{z}},\abs{\partial_{\theta}\Phi_{\theta}-1}\}.
\]
It follows that
\[
\norm{\Phi_{\rho}\partial_{z}\Phi_{\theta}}_{L^{2}(\Omega)}\lesssim_{\varrho_{0},m}\max\{\norm{g_{\theta z}}_{L^{2}(\Omega)},\norm{\partial_{\theta}\Phi_{\rho}}_{L^{2}(\Omega)},\norm{\partial_{\theta}\Phi_{\theta}-1}_{L^{2}(\Omega)},\norm{\partial_{\theta}\Phi_{z}}_{L^{2}(\Omega)}\}.
\]
Thus, after applying \prettyref{lem:NLexcesssplits}, \prettyref{lem:NLmembrLB},
and using H\"older's inequality, we find that
\[
\norm{\Phi_{\rho}\partial_{z}\Phi_{\theta}}_{L^{2}(\Omega)}\lesssim_{\varrho_{0},m}\max\{(\Delta^{NL})^{1/2},(\Delta^{NL})^{1/4}\}=(\Delta^{NL})^{1/4}
\]
as desired.
\end{proof}
Combining \prettyref{lem:bucklingestimate_witherror} and \prettyref{lem:NLcross-term}
gives the following result.
\begin{cor}
\label{cor:NLbucklingcontrol}Let $\Phi\in A_{\lambda,\varrho,m}^{NL}$.
Then we have that
\[
\lambda\abs{A}\lesssim_{\varrho_{0},m}\max\{\int_{A}\norm{\partial_{z}\Phi_{\rho}}_{L_{z}^{2}}^{2}\,d\theta,(\Delta^{NL})^{1/2}\}
\]
for all $A\in\mathcal{B}(I_{\theta})$.
\end{cor}
Finally, we consider the bending term. 
\begin{lem}
\label{lem:NLbendingcontrol}Let $\Phi\in A_{\lambda,\varrho,m}^{NL}$.
Then we have that
\[
\max\left\{ \frac{1}{h^{2}}\Delta^{NL},(\Delta^{NL})^{1/2}\right\} \gtrsim_{\varrho_{0}m}\max\left\{ \norm{D^{2}\Phi_{\rho}}_{L^{2}(\Omega)}^{2},\norm{\Phi_{\rho}-\varrho}_{L^{1}(\Omega)}\right\} .
\]
\end{lem}
\begin{proof}
First, we consider the $\theta z$- and $zz$-components of $D^{2}\Phi_{\rho}$.
From \prettyref{eq:DDPhi}, it follows that 
\begin{align*}
\abs{\partial_{\theta z}\Phi} & \geq\abs{\partial_{\theta z}\Phi_{\rho}-\Phi_{\rho}\partial_{\theta}\Phi_{\theta}\partial_{z}\Phi_{\theta}}\\
\abs{\partial_z^2\Phi} & \geq\abs{\partial_z^2\Phi_{\rho}-\Phi_{\rho}\left(\partial_{z}\Phi_{\theta}\right)^{2}}
\end{align*}
so that
\begin{align*}
\norm{\partial_{\theta z}\Phi_{\rho}}_{L^{2}(\Omega)} & \leq\norm{\partial_{\theta z}\Phi}_{L^{2}(\Omega)}+\norm{\Phi_{\rho}\partial_{\theta}\Phi_{\theta}\partial_{z}\Phi_{\theta}}_{L^{2}(\Omega)}\\
\norm{\partial_z^2\Phi_{\rho}}_{L^{2}(\Omega)} & \leq\norm{\partial_z^2\Phi}_{L^{2}(\Omega)}+\norm{\Phi_{\rho}\left(\partial_{z}\Phi_{\theta}\right)^{2}}_{L^{2}(\Omega)}.
\end{align*}
Using \prettyref{lem:NLcross-term}, we can bound the error terms
in the same manner: 
\begin{align*}
\norm{\Phi_{\rho}\partial_{\theta}\Phi_{\theta}\partial_{z}\Phi_{\theta}}_{L^{2}(\Omega)} & \leq\norm{\Phi_{\rho}\partial_{z}\Phi_{\theta}}_{L^{2}(\Omega)}\norm{\partial_{\theta}\Phi_{\theta}}_{L^{\infty}(\Omega)}\lesssim_{\varrho_{0},m}(\Delta^{NL})^{1/4}\\
\norm{\Phi_{\rho}\left(\partial_{z}\Phi_{\theta}\right)^{2}}_{L^{2}(\Omega)} & \leq\norm{\Phi_{\rho}\partial_{z}\Phi_{\theta}}_{L^{2}(\Omega)}\norm{\partial_{z}\Phi_{\theta}}_{L^{\infty}(\Omega)}\lesssim_{\varrho_{0},m}(\Delta^{NL})^{1/4}.
\end{align*}
Combining this with \prettyref{lem:NLexcesssplits}, we find that
\[
\norm{\partial_{\theta z}\Phi_{\rho}}_{L^{2}(\Omega)}\vee\norm{\partial_z^2\Phi_{\rho}}_{L^{2}(\Omega)}\lesssim_{\varrho_{0},m}(\frac{1}{h^{2}}\Delta^{NL})^{1/2}\vee(\Delta^{NL})^{1/4}.
\]
This completes the $\theta z$- and $zz$-components of the result.

Now we consider the $\theta\theta$-component of $D^{2}\Phi$, which
requires a more careful estimate. We begin by using \prettyref{eq:DDPhi}
to write that
\begin{equation}
\abs{\partial_{\theta}^2\Phi}^{2}-\varrho^{2}\geq\abs{\partial_{\theta}^2\Phi_{\rho}-\Phi_{\rho}\left(\partial_{\theta}\Phi_{\theta}\right)^{2}}^{2}+\abs{2\partial_{\theta}\Phi_{\rho}\partial_{\theta}\Phi_{\theta}+\Phi_{\rho}\partial_{\theta}^2\Phi_{\theta}}^{2}-\varrho^{2}=\abs{\partial_{\theta}^2\Phi_{\rho}}^{2}+I+II\label{eq:NLbendingeq1}
\end{equation}
where
\begin{align*}
I & =\abs{\Phi_{\rho}\left(\partial_{\theta}\Phi_{\theta}\right)^{2}}^{2}-\varrho^{2}\\
II & =\abs{\Phi_{\rho}\partial_{\theta}^2\Phi_{\theta}}^{2}+4\abs{\partial_{\theta}\Phi_{\rho}\partial_{\theta}\Phi_{\theta}}^{2}+4\partial_{\theta}\Phi_{\rho}\partial_{\theta}\Phi_{\theta}\Phi_{\rho}\partial_{\theta}^2\Phi_{\theta}-2\Phi_{\rho}\partial_{\theta}^2\Phi_{\rho}\left(\partial_{\theta}\Phi_{\theta}\right)^{2}.
\end{align*}
First, we discuss $I$. Introducing the displacement $\phi_{\rho}=\Phi_{\rho}-\varrho$,
which is non-negative, we have that
\[
I=(\phi_{\rho}+\varrho)^{2}\left(\partial_{\theta}\Phi_{\theta}\right)^{4}-\varrho^{2}\geq\varrho^{2}((\partial_{\theta}\Phi_{\theta})^{4}-1)+2\varrho|\phi_{\rho}|(\partial_{\theta}\Phi_{\theta})^{4}.
\]
By Jensen's inequality and since $\varrho\geq1$, 
\[
\int_{\Omega}I\geq2\varrho\int_{\Omega}|\phi_{\rho}|(\partial_{\theta}\Phi_{\theta})^{4}\geq\norm{\phi_{\rho}(\partial_{\theta}\Phi_{\theta})^{4}}_{L^{1}(\Omega)}.
\]
In particular, this shows that $\int_{\Omega}I\geq0$. Continuing,
we have that
\begin{align*}
\norm{\phi_{\rho}}_{L^{1}(\Omega)} & \leq\norm{\phi_{\rho}((\partial_{\theta}\Phi_{\theta})^{4}-1)}_{L^{1}(\Omega)}+\int_{\Omega}I \\
& \leq\norm{\phi_{\rho}}_{L^{2}(\Omega)}\norm{(\partial_{\theta}\Phi_{\theta})^{4}-1)}_{L^{2}(\Omega)}+\int_{\Omega}I
  \lesssim_{m}\norm{\phi_{\rho}}_{L^{2}(\Omega)}\norm{\partial_{\theta}\Phi_{\theta}-1}_{L^{2}(\Omega)}+\int_{\Omega}I \\
 &\lesssim(\norm{\partial_{\theta}\phi_{\rho}}_{L^{2}(\Omega)}\vee\norm{\phi_{\rho}}_{L_{z}^{2}L_{\theta}^{1}})\norm{\partial_{\theta}\Phi_{\theta}-1}_{L^{2}(\Omega)}+\int_{\Omega}I
\end{align*}
where in the last step we used Poincare's inequality. So by \prettyref{lem:NLmembrLB},
H\"older's inequality, and our assumption that $\Delta^{NL}\leq1$,
it follows that 
\begin{equation}
\norm{\phi_{\rho}}_{L^{1}(\Omega)}\lesssim_{m}(\Delta^{NL})^{1/2}\vee\left|\int_{\Omega}I\right|.\label{eq:NLbendingeq2}
\end{equation}

Next, we discuss $II$. An integration by parts argument shows that
\[
\int_{\Omega}\Phi_{\rho}\partial_{\theta}^2\Phi_{\rho}\left(\partial_{\theta}\Phi_{\theta}\right)^{2}=-\int_{\Omega}\left(\partial_{\theta}\Phi_{\rho}\partial_{\theta}\Phi_{\theta}\right)^{2}+2\Phi_{\rho}\partial_{\theta}\Phi_{\rho}\partial_{\theta}\Phi_{\theta}\partial_{\theta}^2\Phi_{\theta},
\]
so that by an elementary Young's inequality we have that
\[
\int_{\Omega}II=\int_{\Omega}\abs{\Phi_{\rho}\partial_{\theta}^2\Phi_{\theta}}^{2}+6\abs{\partial_{\theta}\Phi_{\rho}\partial_{\theta}\Phi_{\theta}}^{2}+8\partial_{\theta}\Phi_{\rho}\partial_{\theta}\Phi_{\theta}\Phi_{\rho}\partial_{\theta}^2\Phi_{\theta}\geq-10\int_{\Omega}\abs{\partial_{\theta}\Phi_{\rho}\partial_{\theta}\Phi_{\theta}}^{2}.
\]
Hence, by H\"older's inequality and \prettyref{lem:NLmembrLB}, it
follows that
\[
\int_{\Omega}II\gtrsim_{m}-\norm{\partial_{\theta}\Phi_{\rho}}_{L_{z}^{4}L_{\theta}^{2}}^{2}\gtrsim-(\Delta^{NL})^{1/2}.
\]

Now we combine the estimates. Using \prettyref{lem:NLexcesssplits}
along with \prettyref{eq:NLbendingeq1} and the fact that $\int_{\Omega}\,I\geq 0$, we have that
\[
\frac{1}{h^{2}}\Delta^{NL}\geq\int_{\Omega}\abs{\partial_{\theta}^2\Phi}^{2}-\varrho^{2}\geq\norm{\partial_{\theta}^2\Phi_{\rho}}_{L^{2}(\Omega)}^{2}+\left|\int_{\Omega}I\right|+\int_{\Omega}II
\]
and hence that
\begin{equation}
\left|\int_{\Omega}I\right|+\norm{\partial_{\theta}^2\Phi_{\rho}}_{L^{2}(\Omega)}^{2}\leq\frac{1}{h^{2}}\Delta^{NL}-\int_{\Omega}II\lesssim_{m}(\frac{1}{h^{2}}\Delta^{NL})\vee(\Delta^{NL})^{1/2}.\label{eq:NLbendingeq3}
\end{equation}
Combining \prettyref{eq:NLbendingeq2}
and \prettyref{eq:NLbendingeq3} gives the desired result.
\end{proof}

\subsubsection{Proof of the ansatz-free lower bound}

We now combine the above estimates with the Gagliardo-Nirenberg interpolation inequalities from \prettyref{sec:Appendix} to prove the desired lower bound. At this stage, the argument is more-or-less parallel to the one given for the vKD model in \prettyref{sub:largemandrelLB_FvK}.

\begin{proof}[Proof of \prettyref{prop:NLLBs_largemandr}] Introduce
the radial displacement, $\phi_{\rho}=\Phi_{\rho}-\varrho$. As a
result of \prettyref{lem:NLmembrLB}, \prettyref{cor:NLbucklingcontrol},
and \prettyref{lem:NLbendingcontrol}, we have the following estimates:
\[
\Delta^{NL}\gtrsim(\varrho^{2}-1)\norm{\phi_{\rho}}_{L^{1}(\Omega)},
\]
\[
\max\left\{ \frac{1}{h^{2}}\Delta^{NL},(\Delta^{NL})^{1/2}\right\} \gtrsim_{\varrho_{0},m}\max\left\{ \norm{D^{2}\phi_{\rho}}_{L^{2}(\Omega)}^{2},\norm{\phi_{\rho}}_{L^{1}(\Omega)}\right\} ,
\]
and 
\[
\max\{\int_{A}\norm{\partial_{z}\phi_{\rho}}_{L_{z}^{2}}^{2}\,d\theta,(\Delta^{NL})^{1/2}\}\gtrsim_{\varrho_{0},m}\lambda\abs{A}\quad\forall\,A\in\cB(I_{\theta}).
\]
We now conclude the proof by a case analysis. 

First, consider the case that $\frac{1}{h^{2}}\Delta^{NL}\leq(\Delta^{NL})^{1/2}$.
In this case, we conclude by Poincare's inequality (since $\phi_{\rho}\in H_{\text{per}}^{2}$)
that
\[
(\Delta^{NL})^{1/2}\gtrsim_{\varrho_{0},m}\norm{D^{2}\phi_{\rho}}_{L^{2}(\Omega)}^{2}\gtrsim\norm{\partial_{z}\phi_{\rho}}_{L^{2}(\Omega)}^{2}
\]
and hence that 
\[
\Delta^{NL}\gtrsim_{\varrho_{0},m}\lambda^{2}
\]
upon taking $A=I_{\theta}$.

In the opposite case, we have the lower bound 
\[
\Delta^{NL}\gtrsim_{\varrho_{0},m}\max\left\{ \left[(\varrho^{2}-1)\vee h^{2}\right]\norm{\phi_{\rho}}_{L^{1}(\Omega)},h^{2}\norm{D^{2}\phi_{\rho}}_{L^{2}(\Omega)}^{2}\right\} .
\]
Now, we give two separate arguments that combine to give the desired
result. First, we apply the interpolation inequality from \prettyref{lem:2dinterp-1}
to $\phi_{\rho}$ to conclude that
\begin{align*}
\norm{D\phi_{\rho}}_{L^{2}(\Omega)}^{2} &\lesssim_{\varrho_{0},m}\norm{D\phi_{\rho}}_{L^{\infty}(\Omega)}^{2/3}\left(\frac{1}{(\varrho^{2}-1)\vee h^{2}}\Delta^{NL}\right)^{2/3}\left(\frac{1}{h^{2}}\Delta^{NL}\right)^{1/3}\\
&\lesssim_{\varrho_{0},m}\left[(\varrho^{2}-1)\vee h^{2}\right]^{-2/3}h^{-2/3}\Delta^{NL}.
\end{align*}
Taking $A=I_{\theta}$ gives that
\[
\max\{\norm{\partial_{z}\phi_{\rho}}_{L^{2}(\Omega)}^{2},(\Delta^{NL})^{1/2}\}\gtrsim_{\varrho_{0},m}\lambda
\]
so that
\[
\max\left\{ \left[(\varrho^{2}-1)\vee h^{2}\right]^{-2/3}h^{-2/3}\Delta^{NL},(\Delta^{NL})^{1/2}\right\} \gtrsim_{\varrho_{0},m}\lambda.
\]
Therefore, we conclude by this argument that
\[
\Delta^{NL}\gtrsim_{\varrho_{0},m}\min\left\{ \lambda^{2},h^{2/3}\left[(\varrho^{2}-1)\vee h^{2}\right]^{2/3}\lambda\right\} .
\]

For the second argument, we begin by defining the sets 
\[
Z_{\epsilon}=\left\{ \theta\in I_{\theta}\ :\ \norm{\partial_{z}\phi_{\rho}}_{L_{z}^{2}}^{2}\geq\epsilon\lambda\right\} 
\]
for $\epsilon\in\R_{+}$. Choosing $A=I_{\theta}\backslash Z_{\epsilon}$
gives that
\[
\max\{\epsilon\lambda\abs{I_{\theta}\backslash Z_{\epsilon}},(\Delta^{NL})^{1/2}\}\geq c_{1}(\varrho_{0},m)\lambda\abs{I_{\theta}\backslash Z_{\epsilon}}.
\]
In particular, taking $\epsilon=c_{1}/2$, we conclude that 
\[
\Delta^{NL}\geq c_{1}^{2}\abs{I_{\theta}\backslash Z_{c_{1}/2}}^{2}\lambda^{2}.
\]
Now if $\abs{I_{\theta}\backslash Z_{c_{1}/2}}\geq\frac{1}{2}\abs{I_{\theta}}$,
we conclude that
\[
\Delta^{NL}\geq\frac{c_{1}^{2}}{4}\abs{I_{\theta}}^{2}\lambda^{2}.
\]
Otherwise, we are in the case where $\abs{Z_{c_{1}/2}}>\frac{1}{2}\abs{I_{\theta}}$. 

In this final case, we have that 
\[
\lambda^{5/7}\lesssim_{\varrho_{0},m}\frac{1}{2}\abs{I_{\theta}}(\frac{c_{1}}{2}\lambda)^{5/7}\leq\int_{Z_{c_{1}/2}}\norm{\partial_{z}\phi_{\rho}}_{L_{z}^{2}}^{10/7}\,d\theta\leq\int_{I_{\theta}}\norm{\partial_{z}\phi_{\rho}}_{L_{z}^{2}}^{10/7}\,d\theta.
\]
Applying the first interpolation inequality in \prettyref{lem:1dinterp}
to $\phi_{\rho}$, we get that 
\begin{align*}
\lambda^{5/7} & \lesssim_{\varrho_{0},m}\int_{I_{\theta}}\left(\norm{\phi_{\rho}}_{L_{z}^{1}}^{2/5}\norm{\partial_z^2\phi_{\rho}}_{L_{z}^{2}}^{3/5}\right)^{10/7}\,d\theta=\int_{I_{\theta}}\norm{\phi_{\rho}}_{L_{z}^{1}}^{4/7}\norm{\partial_z^2\phi_{\rho}}_{L_{z}^{2}}^{6/7}\,d\theta\\
 & \leq\norm{\phi_{\rho}}_{L^{1}(\Omega)}^{4/7}\norm{\partial_z^2\phi_{\rho}}_{L^{2}(\Omega)}^{6/7}
\end{align*}
after an application of H\"older's inequality. It follows that
\[
\lambda^{5/7}\lesssim_{\varrho,m}\left(\frac{1}{(\varrho^{2}-1)\vee h^{2}}\Delta^{NL}\right)^{4/7}\left(\frac{1}{h^{2}}\Delta^{NL}\right)^{3/7}=\left[(\varrho^{2}-1)\vee h^{2}\right]^{-4/7}h^{-6/7}\Delta^{NL}
\]
and so we conclude the second result:
\[
\Delta^{NL}\gtrsim_{\varrho_{0},m}\min\left\{ \lambda^{2},\lambda^{5/7}[(\varrho^{2}-1)\vee h^{2}]^{4/7}h^{6/7}\right\} .
\]

In conclusion, we have proved that
\[
\Delta^{NL}\gtrsim_{\varrho_{0},m}\min\left\{ \lambda^{2},\min\left\{ \lambda^{2},h^{2/3}\left[(\varrho^{2}-1)\vee h^{2}\right]^{2/3}\lambda\right\} \vee\min\left\{ \lambda^{2},\lambda^{5/7}[(\varrho^{2}-1)\vee h^{2}]^{4/7}h^{6/7}\right\} \right\} ,
\]
which is simply a restatement of the desired result.

\end{proof}

\section{Ansatz-free lower bounds in the neutral mandrel case \label{sec:neutralmandrelLB}}

In this section, we prove the lower bounds from \prettyref{thm:FvKneutralbounds}
and \prettyref{thm:NLneutralbounds}. We begin with the vKD model
in \prettyref{sub:neutralmandrelLB_FvK}. There, we introduce the
free-shear functional from \prettyref{eq:FS} as a bounding device
and prove its minimum energy scaling law. Then, we turn to the nonlinear
model in \prettyref{sub:neutralmandrelLB_NL}.

\subsection{vKD model\label{sub:neutralmandrelLB_FvK}}

In the neutral mandrel case, where $\varrho=1$, the estimates proved
in \prettyref{sub:largemandrelLB_FvK} do not lead to useful lower
bounds on $E_{h}^{vKD}$. Nevertheless, buckling in the presence of
the mandrel continues to induce tensile hoop stresses when $\varrho=1$,
and this can still be used to prove non-trivial lower bounds. We emphasize
here that it is not clear at first the degree of success that we should
expect from this approach: indeed, the magnitude of the hoop stresses
induced by the mandrel vanish as $h\to0$ in the neutral mandrel case.
This is in stark contrast with the large mandrel case, where the
effective hoop streses are of order one and the excess hoop stresses
set the minimum energy scaling law. For more on this, we refer the
reader to the discussion in \prettyref{sub:DiscussionofProofs}.

Let us briefly recall from \prettyref{sub:neutralmandrelresults}
our approach to \prettyref{thm:FvKneutralbounds}: introducing the
free-shear functional,
\[
FS_{h}(\phi)=\int_{\Omega}\,\abs{\epsilon_{\theta\theta}}^{2}+\abs{\epsilon_{zz}}^{2}+h^{2}\abs{D^{2}\phi_{\rho}}^{2}\,d\theta dz,
\]
we observe that 
\[
E_{h}^{vKD}(\phi)\geq FS_{h}(\phi)\quad\forall\,\phi\in A_{\lambda,\varrho,m}^{vKD}
\]
since in the definition of $FS_{h}$ we have simply neglected the
cost of shear in the membrane term. Thus, lower bounds on the minimum
of $FS_{h}$ give lower bounds on the minimum of $E_{h}^{vKD}$. In
the present section, we give the optimal argument along these lines.
To do so, we answer the following question: what is the minimum energy
scaling law of the free-shear functional? 

Let $A_{\lambda,m}=A_{\lambda,1,m}^{vKD}$.
\begin{prop}
\label{prop:FSscalinglaw} Let $h,\lambda\in(0,\frac{1}{2}]$ and
$m\in[2,\infty)$. Then we have that
\[
\min_{A_{\lambda,m}}\,FS_{h}\sim_{m}\min\left\{ \max\{h\lambda^{3/2},(h\lambda)^{12/11}\},\lambda^{2}\right\} 
\]
In the case that $m=\infty$, we have that
\[
\min_{A_{\lambda,\infty}}\,FS_{h}\sim\min\left\{ (h\lambda)^{12/11},\lambda^{2}\right\} .
\]
\end{prop}
\begin{rem}
\label{rem:blowupFS} As in the analysis of the large mandrel case,
we can quantify the blow-up rate of $\norm{D\phi}_{L^{\infty}}$ for
the free-shear functional as $h\to0$. See \prettyref{sub:BlowuprateFS}
for the precise statement of this result. \end{rem}
\begin{proof}
The asserted lower bounds follow from \prettyref{cor:FSLB-1} and
\prettyref{cor:FSLB-2}. The upper bound of $\lambda^{2}$ is achieved
by the unbuckled configuration, $\phi=(0,0,-\lambda z)$. To prove
the remainder of the upper bound, note first that it suffices to achieve
it for $(h,\lambda,m)\in(0,h_{0}]\times(0,\frac{1}{2}]\times[2,\infty)$
for some $h_{0}\in(0,\frac{1}{2}]$. So, we take $h_{0}=\frac{1}{2^{10}}$
and apply \prettyref{lem:FSUB_manywrinkles_tilted}, \prettyref{lem:FSUB_onewrinkle_tilted_long},
and \prettyref{lem:FSUB_onewrinkle_tilted} to get that 
\[
\min_{A_{\lambda,m}}\,FS_{h}\lesssim\min\left\{ \lambda^{2},\max\left\{ m^{-1/2}h\lambda^{3/2},(h\lambda)^{12/11},h^{6/5}\lambda\right\} \right\} 
\]
in the stated parameter range. Since 
\[
\min\left\{ \lambda^{2},\max\left\{ (h\lambda)^{12/11},h^{6/5}\lambda\right\} \right\} =\min\left\{ \lambda^{2},(h\lambda)^{12/11}\right\} 
\]
the result follows.
\end{proof}
This result shows that the free-shear functional prefers three types
of low-energy patterns if $m<\infty$, and two if $m=\infty$. See
\prettyref{fig:FSheightfield} for a schematic of these patterns.

\subsubsection{Lower bounds on the free-shear functional}

Here, we prove the lower bound from \prettyref{prop:FSscalinglaw}.
Our first result is the free-shear version of \prettyref{lem:FvKLBs_largemandr}.
\begin{lem}
\label{lem:FSLBs}Let $\phi\in A_{\lambda,\infty}$. Then we have
that 
\[
FS_{h}(\phi)\gtrsim\max\left\{ \norm{\phi_{\rho}}_{L_{z}^{2}L_{\theta}^{1}}^{2},\norm{\partial_{\theta}\phi_{\rho}}_{L_{z}^{4}L_{\theta}^{2}}^{4},h^{2}\norm{D^{2}\phi_{\rho}}_{L^{2}(\Omega)}^{2},\norm{\frac{1}{2}\norm{\partial_{z}\phi_{\rho}}_{L_{z}^{2}}^{2}-\lambda}_{L_{\theta}^{2}}^{2}\right\} .
\]
\end{lem}
\begin{proof}
By the definition of $FS_{h}$ in \prettyref{eq:FS}, we have that
\[
FS_{h}(\phi)=\int_{\Omega}\,\abs{\partial_{\theta}\phi_{\theta}+\frac{1}{2}(\partial_{\theta}\phi_{\rho})^{2}+\phi_{\rho}}^{2}+\abs{\partial_{z}\phi_{z}+\frac{1}{2}(\partial_{z}\phi_{\rho})^{2}}^{2}+h^{2}\abs{D^{2}\phi_{\rho}}^{2}\,d\theta dz.
\]
Applying Jensen's inequality in the $\theta$-direction and using
that $\phi_{\theta}\in H_{\text{per}}^{1}$ and that $\phi_{\rho}\geq0$
we see that
\[
\norm{\partial_{\theta}\phi_{\theta}+\frac{1}{2}(\partial_{\theta}\phi_{\rho})^{2}+\phi_{\rho}}_{L^{2}(\Omega)}\gtrsim\norm{\int_{I_{\theta}}\partial_{\theta}\phi_{\theta}+\frac{1}{2}(\partial_{\theta}\phi_{\rho})^{2}+\phi_{\rho}\,d\theta}_{L_{z}^{2}}\gtrsim\norm{\partial_{\theta}\phi_{\rho}}_{L_{z}^{4}L_{\theta}^{2}}^{2}\vee\norm{\phi_{\rho}}_{L_{z}^{2}L_{\theta}^{1}}.
\]
Applying Jensen's inequality in the $z$-direction and using that
$\phi_{z}+\lambda z\in H_{\text{per}}^{1}$ we see that
\[
\norm{\partial_{z}\phi_{z}+\frac{1}{2}(\partial_{z}\phi_{\rho})^{2}}_{L^{2}(\Omega)}\gtrsim\norm{\int_{I_{z}}\partial_{z}\phi_{z}+\frac{1}{2}(\partial_{z}\phi_{\rho})^{2}\,dz}_{L_{\theta}^{2}}=\norm{\frac{1}{2}\norm{\partial_{z}\phi_{\rho}}_{L_{z}^{2}}^{2}-\lambda}_{L_{\theta}^{2}}.
\]
The result now follows.
\end{proof}
Now, we apply the Gagliardo-Nirenberg interpolation inequalities from
\prettyref{sec:Appendix} to deduce the desired lower bounds.
\begin{cor}
\label{cor:FSLB-1} If $\phi\in A_{\lambda,m}$, then
\[
FS_{h}(\phi)\gtrsim\min\{ m^{-1}h\lambda^{3/2},\lambda^{2}\} 
\]
whenever $h,\lambda\in(0,\infty)$ and $m\in(0,\infty]$.

In fact, if $\phi\in A_{\lambda,\infty}$, then 
\[
FS_{h}(\phi)\gtrsim\min\{\norm{D\phi_{\rho}}_{L^{\infty}(\Omega)}^{-1}h\lambda^{3/2},\lambda^{2}\}.
\]
\end{cor}
\begin{proof}
Observe that by \prettyref{lem:FSLBs} and H\"older's inequality,
we have that 
\[
c_{1}\left(FS_{h}(\phi)\right)^{1/2}\geq\norm{\frac{1}{2}\norm{\partial_{z}\phi_{\rho}}_{L_{z}^{2}}^{2}-\lambda}_{L_{\theta}^{1}}
\]
for some numerical constant $c_{1}$.
Hence, by the triangle inequality,
\[
\frac{1}{2}\norm{\partial_{z}\phi_{\rho}}_{L^{2}(\Omega)}^{2}+c_{1}(FS_{h}(\phi))^{1/2}\geq\lambda|I_{\theta}|.
\]
Now we perform a case analysis. If $\phi$ satisfies $\norm{\partial_{z}\phi_{\rho}}_{L^{2}(\Omega)}^{2}\leq\lambda|I_{\theta}|$,
then we conclude by the above that $FS_{h}(\phi)\gtrsim\lambda^{2}$.

On the other hand, suppose that $\phi$ satisfies $\norm{\partial_{z}\phi_{\rho}}_{L^{2}(\Omega)}^{2}>\lambda|I_{\theta}|$.
Then, observe that by \prettyref{lem:FSLBs} and H\"older's inequality,
\[
FS_{h}(\phi)\gtrsim\max\left\{ \norm{\phi_{\rho}}_{L^{1}(\Omega)}^{2},h^{2}\norm{D^{2}\phi_{\rho}}_{L^{2}(\Omega)}^{2}\right\} .
\]
Combining this with the interpolation inequality from \prettyref{lem:2dinterp-1},
we conclude that 
\[
\lambda^{1/2}\lesssim\norm{D\phi_{\rho}}_{L^{2}(\Omega)}\lesssim\norm{D\phi_{\rho}}_{L^{\infty}(\Omega)}^{1/3}\norm{\phi_{\rho}}_{L^{1}(\Omega)}^{1/3}\norm{D^{2}\phi_{\rho}}_{L^{2}(\Omega)}^{1/3}\lesssim m^{1/3}h^{-1/3}(FS_{h}(\phi))^{1/3}
\]
and the result follows.\end{proof}
\begin{cor}
\label{cor:FSLB-2} If $\phi\in A_{\lambda,m}$, then
\[
FS_{h}(\phi)\gtrsim\min\left\{ (h\lambda)^{12/11},\lambda^{2}\right\} 
\]
whenever $h,\lambda\in(0,1]$ and $m\in(0,\infty]$.\end{cor}
\begin{proof}
As in the proof of \prettyref{cor:FSLB-1}, it suffices to prove that
\[
\norm{\partial_{z}\phi_{\rho}}_{L^{2}(\Omega)}^{2}\gtrsim\lambda\implies FS_{h}(\phi)\gtrsim(h\lambda)^{12/11}.
\]
Combining the third interpolation inequality from \prettyref{lem:2dinterp}
with the anisotropic interpolation inequality from \prettyref{lem:mixedinterp_1},
we find that
\begin{align*}
\norm{D\phi_{\rho}}_{L^{2}(\Omega)} & \lesssim\norm{\phi_{\rho}}_{L^{2}(\Omega)}^{1/2}\norm{D^{2}\phi_{\rho}}_{L^{2}(\Omega)}^{1/2}\lesssim(\norm{\partial_{\theta}\phi_{\rho}}_{L_{z}^{4}L_{\theta}^{2}}^{1/3}\norm{\phi_{\rho}}_{L_{z}^{2}L_{\theta}^{1}}^{2/3}+\norm{\phi_{\rho}}_{L_{z}^{2}L_{\theta}^{1}})^{1/2}\norm{D^{2}\phi_{\rho}}_{L_{\theta z}^{2}}^{1/2}\\
 & \lesssim\max\left\{ \norm{\partial_{\theta}\phi_{\rho}}_{L_{z}^{4}L_{\theta}^{2}}^{1/6}\norm{\phi_{\rho}}_{L_{z}^{2}L_{\theta}^{1}}^{1/3}\norm{D^{2}\phi_{\rho}}_{L_{\theta z}^{2}}^{1/2},\norm{\phi_{\rho}}_{L_{z}^{2}L_{\theta}^{1}}^{1/2}\norm{D^{2}\phi_{\rho}}_{L_{\theta z}^{2}}^{1/2}\right\} .
\end{align*}
Hence, by \prettyref{lem:FSLBs}, we conclude that
\[
h\lambda\lesssim\max\left\{ FS_{h}^{11/12},FS_{h}\right\} .
\]
It follows immediately that
\[
FS_{h}\gtrsim\min\left\{ (h\lambda)^{12/11},h\lambda\right\} =(h\lambda)^{12/11}
\]
as desired.
\end{proof}

\subsubsection{Upper bounds on the free-shear functional}

In this section, we prove the upper bound from \prettyref{prop:FSscalinglaw}.
Since this upper bound matches the lower bounds from the previous
section, our analysis of the free-shear functional is optimal as far
as scaling laws are concerned. In the remainder of this section, we
will \textbf{assume} that 
\[
h\in(0,\frac{1}{2^{10}}],\ \lambda\in(0,\frac{1}{2}],\ \text{and}\ m\in[2,\infty)
\]
unless otherwise explicitly stated.

We begin by defining a two-scale wrinkling pattern along a to-be-chosen
direction. We refer to the parameters $n,k\in\N$ and $\delta\in(0,1]$,
which are the number of wrinkles, the number of times each wrinkle
wraps about the cylinder, and the relative extent of the wrinkles.
See \prettyref{fig:FSheightfield} for a schematic of this construction.

To define the construction, we fix $f\in C^{\infty}(\R)$ such that
\begin{itemize}
\item $f$ is non-negative and one-periodic
\item $\text{supp}\,f\cap[-\frac{1}{2},\frac{1}{2}]\subset(-\frac{1}{2},\frac{1}{2})$
\item $\norm{f'}_{L^{\infty}}\leq2$
\item $\norm{f'}_{L^{2}(B_{1/2})}^{2}=1$.
\end{itemize}
Define $f_{\delta,n}\in C^{\infty}(\R)$ by
\[
f_{\delta,n}(t)=\frac{\sqrt{\delta}}{n}f(\frac{n}{\delta}\{t\})\ind{\{t\}\in B_{\delta/2}}
\]
and $w_{\delta,n,k,\lambda}:\Omega\to\R$ by 
\[
w_{\delta,n,k,\lambda}(\theta,z)=\frac{\sqrt{2\lambda}}{k}f_{\delta,n}(\frac{\theta}{2\pi}+kz).
\]
Recall that we write $\overline{f}$ to denote the $\theta$-average
of $f$, as given in \prettyref{sub:notation}. Define $u^{\delta,n,k,\lambda}=(u_{\theta}^{\delta,n,k,\lambda},u_{z}^{\delta,n,k,\lambda}):\Omega\to\R^{2}$
by
\begin{align*}
u_{\theta}^{\delta,n,k,\lambda}(\theta,z) & =\int_{0\leq\theta'\leq\theta}\left[\left(\overline{\frac{1}{2}(\partial_{\theta}w)^{2}+w}\right)(z)-\frac{1}{2}(\partial_{\theta}w(\theta',z))^{2}-w(\theta',z)\right]\,d\theta'\\
u_{z}^{\delta,n,k,\lambda}(\theta,z) & =\int_{-\frac{1}{2}\leq z'\leq z}\left[\lambda-\frac{1}{2}(\partial_{z}w(\theta,z'))^{2}\right]\,dz'
\end{align*}
where $w=w_{\delta,n,k,\lambda}$. Finally, define $\phi_{\delta,n,k,\lambda}:\Omega\to\R^{3}$
by
\[
\phi_{\delta,n,k,\lambda}=(w_{\delta,n,k,\lambda},u_{\theta}^{\delta,n,k,\lambda},-\lambda z+u_{z}^{\delta,n,k,\lambda}),
\]
in cylindrical coordinates. 

\begin{figure}
\includegraphics[height=0.18\textheight]{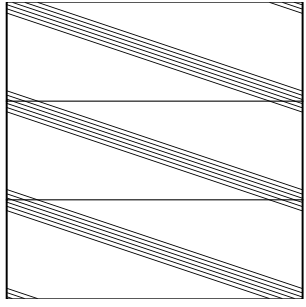}

\caption{This schematic depicts the free-shear construction. The pattern features
$n$ wrinkles which wrap $k$ times about the cylinder, with total
volume fraction $\delta$. The optimal choice of $n,k\delta$ depends
on the axial compression, $\lambda$, the thickness, $h$, and the
\emph{a priori} $L^{\infty}$ slope bound, $m$. \label{fig:FSheightfield}}
\end{figure}

Now, we estimate the energy of this construction. Let
\[
m_{2}(\delta,n,k,\lambda)=2\max\left\{ \sqrt{\frac{2\lambda}{\delta}},\frac{2\lambda}{\delta},\frac{2\lambda}{\pi k\delta}+\frac{2\pi\sqrt{2\lambda\delta}}{n}\right\} .
\]

\begin{lem}
\label{lem:FSadmissability}We have that $\phi_{\delta,n,k,\lambda}\in A_{\lambda,m_{2}}$.
Furthermore,
\[
FS_{h}(\phi_{\delta,n,k,\lambda})\lesssim\max\left\{ \frac{\lambda\delta^{3}}{k^{2}n^{2}},\frac{\lambda^{2}}{k^{4}},h^{2}\frac{\lambda k^{2}n^{2}}{\delta^{2}}\right\} .
\]
\end{lem}
\begin{proof}
Abbreviate $\phi_{\delta,n,k,\lambda}$ by $\phi$, $w_{\delta,n,k,\lambda}$
by $w$, and $u^{\delta,n,k,\lambda}$ by $u$. By its definition,
$\phi_{\rho}\in H_{\text{per}}^{2}$, $\phi_{\theta}\in H_{\text{per}}^{1}$, and
$\phi_{z}+\lambda z\in H_{\text{per}}^{1}$. In particular, we note that
\[
\int_{-\frac{1}{2}\leq z'\leq\frac{1}{2}}\frac{1}{2}|\partial_{z}w(\theta,z')|^{2}dz=\lambda\int_{B_{\delta/2}}|f'_{\delta,n}|^{2}dt=\lambda\int_{B_{1/2}}|f'|^{2}dt=\lambda
\]
for all $\theta\in I_{\theta}$, so that $u_{z}^{\delta,n,k,\lambda}\in H_{\text{per}}^{1}$.
Also, we have that $w\geq0$ so that $\phi_{\rho}\geq0$. 

Now we obtain the slope bounds. Since
\begin{align*}
\epsilon_{\theta\theta} & =\partial_{\theta}\phi_{\theta}+\frac{1}{2}(\partial_{\theta}\phi_{\rho})^{2}+\phi_{\rho}=\overline{\frac{1}{2}(\partial_{\theta}\phi_{\rho})^{2}+\phi_{\rho}}\\
\epsilon_{zz} & =\partial_{z}\phi_{z}+\frac{1}{2}(\partial_{z}\phi_{\rho})^{2}=0
\end{align*}
and 
\begin{align*}
\partial_{\theta}\phi_{\rho}(\theta,z) & =\partial_{\theta}w(\theta,z)=\frac{1}{2\pi}\frac{\sqrt{2\lambda}}{k}f'_{\delta,n}(\frac{\theta}{2\pi}+kz)\\
\partial_{z}\phi_{\rho}(\theta,z) & =\partial_{z}w(\theta,z)=\sqrt{2\lambda}f'_{\delta,n}(\frac{\theta}{2\pi}+kz),
\end{align*}
we find that 
\begin{align*}
\norm{\partial_{\theta}\phi_{\rho}}_{L^{\infty}(\Omega)} & \leq\frac{1}{2\pi}\frac{\sqrt{2\lambda}}{k}\norm{f_{\delta,n}'}_{L^{\infty}}\leq\frac{1}{\pi k}\sqrt{\frac{2\lambda}{\delta}}\\
\norm{\partial_{z}\phi_{\rho}}_{L^{\infty}(\Omega)} & \le\sqrt{2\lambda}\norm{f_{\delta,n}'}_{L^{\infty}}\leq2\sqrt{\frac{2\lambda}{\delta}}\\
\norm{\partial_{z}\phi_{z}}_{L^{\infty}(\Omega)} & \leq\lambda\norm{f'_{\delta,n}}_{L^{\infty}}^{2}\leq\frac{4\lambda}{\delta},
\end{align*}
and that
\begin{align*}
\norm{\partial_{\theta}\phi_{\theta}}_{L^{\infty}(\Omega)} & \leq\norm{\overline{\frac{1}{2}(\partial_{\theta}\phi_{\rho})^{2}+\phi_{\rho}}-\frac{1}{2}(\partial_{\theta}\phi_{\rho})^{2}-\phi_{\rho}}_{L^{\infty}(\Omega)}\leq2\norm{\frac{1}{2}(\partial_{\theta}\phi_{\rho})^{2}+\phi_{\rho}}_{L^{\infty}(\Omega)}\\
 & \leq2(\frac{1}{4\pi^{2}}\frac{\lambda}{k^{2}}\norm{f'_{\delta,n}}_{L^{\infty}}^{2}+\frac{\sqrt{2\lambda}}{k}\norm{f_{\delta,n}}_{L^{\infty}})\leq2\left(\frac{\lambda}{\pi^{2}k^{2}\delta}+\frac{\sqrt{2\lambda\delta}}{kn}\right).
\end{align*}
Here, we used that $\norm{f}_{L^{\infty}}\leq1$, which follows from
its definition.

Now we deal with the shear terms. We have that 
\begin{align*}
\partial_{\theta}\phi_{z}(\theta,z) & =\partial_{\theta}u_{z}(\theta,z)=-\int_{-\frac{1}{2}\leq z'\leq z}\partial_{z}w\partial_{\theta z}w(\theta,z')\,dz'\\
\partial_{z}\phi_{\theta}(\theta,z) & =\partial_{z}u_{\theta}(\theta,z)=\int_{0\leq\theta'\leq\theta}\left[\overline{\partial_{\theta}w\partial_{z\theta}w+\partial_{z}w}(z)-\partial_{\theta}w\partial_{z\theta}w(\theta',z)-\partial_{z}w(\theta',z)\right]\,d\theta'.
\end{align*}
Since
\[
\partial_{\theta z}w(\theta,z)=\frac{\sqrt{2\lambda}}{2\pi}f''_{\delta,n}(\frac{\theta}{2\pi}+kz),
\]
we see that
\begin{align*}
\partial_{\theta}\phi_{z}(\theta,z) & =-\int_{-\frac{1}{2}\leq z'\leq z}\frac{2\lambda}{2\pi}f'_{\delta,n}f''_{\delta,n}(\frac{\theta}{2\pi}+kz')\,dz'=-\int_{-\frac{1}{2}\leq t\leq z}\frac{\lambda}{2\pi}\frac{1}{k}\frac{d}{dt}\left[\left(f'_{\delta,n}\right)^{2}(\frac{\theta}{2\pi}+kt)\right]\,dt\\
 & =\frac{1}{2\pi}\frac{\lambda}{k}\left(\left(f'_{\delta,n}\right)^{2}(\frac{\theta}{2\pi}-\frac{k}{2})-\left(f'_{\delta,n}\right)^{2}(\frac{\theta}{2\pi}+kz)\right)
\end{align*}
so that
\[
\norm{\partial_{\theta}\phi_{z}}_{L^{\infty}(\Omega)}\leq2\frac{1}{2\pi}\frac{\lambda}{k}\norm{f'_{\delta,n}}_{L^{\infty}}^{2}\leq\frac{4\lambda}{\pi k\delta}.
\]
Similarly, we have that
\begin{align*}
 & \int_{0\leq\theta'\leq\theta}\left[\partial_{\theta}w\partial_{z\theta}w(\theta',z)+\partial_{z}w(\theta',z)\right]\,d\theta'\\
 & \qquad=\int_{0\leq\theta'\leq\theta}\left[\frac{2\lambda}{(2\pi)^{2}}\frac{1}{k}f'_{\delta,n}f''_{\delta,n}(\frac{\theta'}{2\pi}+kz)+\sqrt{2\lambda}f'_{\delta,n}(\frac{\theta'}{2\pi}+kz)\right]\,d\theta'\\
 & \qquad=\int_{0\leq t\leq\theta}\frac{1}{2\pi}\frac{\lambda}{k}\frac{d}{dt}\left[\left(f'_{\delta,n}\right)^{2}(\frac{t}{2\pi}+kz)\right]+2\pi\sqrt{2\lambda}\frac{d}{dt}\left[f{}_{\delta,n}(\frac{t}{2\pi}+kz)\right]\,dt\\
 & \qquad=\frac{1}{2\pi}\frac{\lambda}{k}\left(\left(f'_{\delta,n}(\frac{\theta}{2\pi}+kz)\right)^{2}-\left(f'_{\delta,n}(kz)\right)^{2}\right)+2\pi\sqrt{2\lambda}\left(f{}_{\delta,n}(\frac{\theta}{2\pi}+kz)-f{}_{\delta,n}(kz)\right).
\end{align*}
Hence,
\[
\partial_{z}\phi_{\theta}(\theta,z)=-\frac{1}{2\pi}\frac{\lambda}{k}\left(\left(f'_{\delta,n}(\frac{\theta}{2\pi}+kz)\right)^{2}-\left(f'_{\delta,n}(kz)\right)^{2}\right)-2\pi\sqrt{2\lambda}\left(f{}_{\delta,n}(\frac{\theta}{2\pi}+kz)-f{}_{\delta,n}(kz)\right)
\]
so that
\[
\norm{\partial_{z}\phi_{\theta}}_{L^{\infty}(\Omega)}\leq2\left(\frac{1}{2\pi}\frac{\lambda}{k}\norm{f'_{\delta,n}}_{L^{\infty}}^{2}+2\pi\sqrt{2\lambda}\norm{f_{\delta,n}}_{L^{\infty}}\right)\leq2\left(\frac{2\lambda}{\pi k\delta}+\frac{2\pi\sqrt{2\lambda\delta}}{n}\right).
\]

Combining the above, we have shown that 
\[
\max_{\substack{i\in\left\{ \theta,z\right\} ,\,j\in\{\rho,\theta,z\}}
}\norm{\partial_{i}\phi_{j}}_{L^{\infty}(\Omega)}\leq2\max\left\{ \sqrt{\frac{2\lambda}{\delta}},\frac{2\lambda}{\delta},\frac{2\lambda}{\pi k\delta}+\frac{2\pi\sqrt{2\lambda\delta}}{n}\right\} =m_{2}
\]
and it follows that $\phi\in A_{\lambda,m_{2}}$.

Now we bound the free-shear energy of this construction. Since $\epsilon_{\theta\theta}=\overline{\epsilon_{\theta\theta}}$
and $\epsilon_{zz}=0$, we have that
\[
FS_{h}(\phi)=\int_{\Omega}\,\abs{\overline{\frac{1}{2}(\partial_{\theta}w)^{2}+w}}^{2}+h^{2}\abs{D^{2}w}^{2}\,d\theta dz
\]
so that
\[
FS_{h}(\phi)\lesssim\max\left\{ \norm{w}_{L_{z}^{2}L_{\theta}^{1}}^{2},\norm{\partial_{\theta}w}_{L_{z}^{4}L_{\theta}^{2}}^{4},h^{2}\norm{D^{2}w}_{L^{2}(\Omega)}^{2}\right\} .
\]
Since
\[
\norm{w}_{L_{z}^{2}L_{\theta}^{1}}^{2}\lesssim\frac{\lambda\delta^{3}}{k^{2}n^{2}},\quad\norm{\partial_{\theta}w}_{L_{z}^{4}L_{\theta}^{2}}^{4}\lesssim\frac{\lambda^{2}}{k^{4}},\ \text{and}\quad\norm{D^{2}w}_{L^{2}(\Omega)}^{2}\lesssim\frac{\lambda k^{2}n^{2}}{\delta^{2}},
\]
it follows that 
\[
FS_{h}(\phi)\lesssim\max\left\{ \frac{\lambda\delta^{3}}{k^{2}n^{2}},\frac{\lambda^{2}}{k^{4}},h^{2}\frac{\lambda k^{2}n^{2}}{\delta^{2}}\right\} .
\]

\end{proof}
Next, we choose $n,k,\delta$ to optimize this bound. Note that each
of the following three choices is optimal in a different parameter
regime. First, we consider a construction made of up many wrinkles,
each of which wraps many times about the cylinder.
\begin{lem}
\label{lem:FSUB_manywrinkles_tilted} Assume that 
\[
m^{-1/2}h\lambda^{3/2}\geq\max\{h^{6/5}\lambda,(h\lambda)^{12/11}\}.
\]
Let $n,k\in\N$ and $\delta\in(0,1]$ satisfy 
\begin{align*}
n&\in\left[7\lambda^{9/8}h^{-1/4}m^{-11/8},8\lambda^{9/8}h^{-1/4}m^{-11/8}\right] \\ k&\in\left[7h^{-1/4}\lambda^{1/8}m^{1/8},8h^{-1/4}\lambda^{1/8}m^{1/8}\right] \\ \delta&=4\lambda m^{-1}.
\end{align*}
Then, $\phi_{\delta,n,k,\lambda}\in A_{\lambda,m}$ and 
\[
FS_{h}(\phi_{\delta,n,k,\lambda})\lesssim\frac{1}{m^{1/2}}h\lambda^{3/2}.
\]
\end{lem}
\begin{proof}
Rearranging the inequality $m^{-1/2}h\lambda^{3/2}\geq(h\lambda)^{12/11}$,
we find that $\lambda^{9/8}h^{-1/4}m^{-11/8}\geq1$ so that there
exists such an $n\in\N$. Rearranging the inequality $m^{-1/2}h\lambda^{3/2}\geq h^{6/5}\lambda$,
we find that $\lambda^{5/8}\geq h^{1/4}m^{5/8}$. Since $m\geq1$
and $\lambda\leq1$, it follows that $\lambda^{1/8}m^{1/8}h^{-1/4}\geq1$.
Hence, there exists such a $k\in\N$. Also, we have that $\delta\leq1$,
since $\lambda\leq\frac{1}{2}$ and $m\geq2$. Now we check the slope
bound. We claim that $m_{2}(\delta,n,k,\lambda)=m.$ Indeed, we have
that
\[
m_{2}=2\max\left\{ \sqrt{\frac{m}{2}},\frac{m}{2},\frac{1}{2\pi}\frac{m}{k}+2\pi\frac{2\sqrt{2}\lambda}{nm^{1/2}}\right\} =2\max\left\{ \frac{m}{2},\frac{1}{2\pi}\frac{m}{k}+2\pi\frac{2\sqrt{2}\lambda}{nm^{1/2}}\right\} ,
\]
and using that $m\geq2$, $\lambda\leq\frac{1}{2}$, and $n,k\geq7$
we see that 
\[
\frac{1}{2\pi}\frac{m}{k}+2\pi\frac{2\sqrt{2}\lambda}{nm^{1/2}}\leq\frac{m}{2}
\]
so that $m_{2}\leq m$ as required.

It follows from \prettyref{lem:FSadmissability} that $\phi_{\delta,n,k,\lambda}\in A_{\lambda,m}$,
and that
\[
FS_{h}(\phi_{\delta,n,k,\lambda})\lesssim\max\left\{ \frac{hm^{5/2}\delta^{3}}{\lambda^{3/2}},\frac{1}{m^{1/2}}h\lambda^{3/2},\frac{h\lambda^{7/2}}{m^{5/2}\delta^{2}}\right\} .
\]
Using that $\delta\sim\frac{\lambda}{m}$, we have that
\[
FS_{h}(\phi_{\delta,n,k,\lambda})\lesssim\frac{1}{m^{1/2}}h\lambda^{3/2}.
\]

\end{proof}
We now consider a construction made up of a few wrinkles, each of
which wraps many times about the cylinder.
\begin{lem}
\label{lem:FSUB_onewrinkle_tilted_long} Assume that 
\[
(h\lambda)^{12/11}\geq\max\{h^{6/5}\lambda,m^{-1/2}h\lambda^{3/2}\}.
\]
Let $n,k\in\N$ and $\delta\in(0,1]$ satisfy 
\[
n=12,\quad k\in\left[12h^{-3/11}\lambda^{5/22},13h^{-3/11}\lambda^{5/22}\right],\ \text{and}\quad\delta=4(h\lambda)^{2/11}.
\]
Then, $\phi_{\delta,n,k,\lambda}\in A_{\lambda,m}$ and 
\[
FS_{h}(\phi_{\delta,n,k,\lambda})\lesssim(h\lambda)^{12/11}.
\]
\end{lem}
\begin{proof}
Rearranging the inequality $(h\lambda)^{12/11}\geq h^{6/5}\lambda$,
we find that $h^{-3/11}\lambda^{5/22}\geq1$ so that there exists
such a $k\in\N$. Also we note that $\delta\leq1$ since $\lambda\leq\frac{1}{2}$
and $h\leq\frac{1}{2^{10}}$. Now we check the slope bound. We have
that
\[
m_{2}=2\max\left\{ \sqrt{\frac{\lambda^{9/11}}{2h{}^{2/11}}},\frac{\lambda^{9/11}}{2h{}^{2/11}},\frac{1}{\pi k}\frac{\lambda^{9/11}}{2h{}^{2/11}}+2\pi\frac{2\sqrt{2}h^{1/11}\lambda^{13/22}}{n}\right\} .
\]
Rearranging the inequality $(h\lambda)^{12/11}\geq m^{-1/2}h\lambda^{3/2}$,
we find that $m\geq\lambda^{9/11}h^{-2/11}$ so that
\[
m_{2}\leq2\max\left\{ \sqrt{\frac{m}{2}},\frac{m}{2},\frac{1}{2\pi}\frac{m}{k}+2\pi\frac{2\sqrt{2}}{n}h^{1/11}\lambda^{13/22}\right\} =2\max\left\{ \frac{m}{2},\frac{1}{2\pi}\frac{m}{k}+2\pi\frac{2\sqrt{2}}{n}h^{1/11}\lambda^{13/22}\right\} .
\]
Using that $h^{-3/11}\lambda^{5/22}\geq1$ we see that
\[
m_{2}\leq2\max\left\{ \frac{m}{2},\frac{1}{2\pi}\frac{m}{k}+2\pi\frac{2\sqrt{2}}{n}\lambda^{2/3}\right\} .
\]
Since $m\geq2$ , $\lambda\leq\frac{1}{2}$, and $n,k\geq12$ we find
that 
\[
\frac{1}{2\pi}\frac{m}{k}+2\pi\frac{2\sqrt{2}}{n}\lambda^{2/3}\leq\frac{m}{2}
\]
so that $m_{2}\leq m$ as required.

It follows from \prettyref{lem:FSadmissability} that $\phi_{\delta,n,k,\lambda}\in A_{\lambda,m}$,
and that
\[
FS_{h}(\phi_{\delta,n,k,\lambda})\lesssim(h\lambda)^{12/11}.
\]

\end{proof}
Finally, we consider a construction made up of a few wrinkles, each
of which wraps a few times about the cylinder.
\begin{lem}
\label{lem:FSUB_onewrinkle_tilted}Assume that
\[
h^{6/5}\lambda\geq\max\{m^{-1/2}h\lambda^{3/2},(h\lambda)^{12/11}\}.
\]
Let $n,k\in\N$ and $\delta\in(0,1]$ satisfy 
\[
n=2,\quad k=2,\ \text{and}\quad\delta=4h^{2/5}.
\]
Then, $\phi_{\delta,n,k,\lambda}\in A_{\lambda,m}$ and 
\[
FS_{h}(\phi_{\delta,n,k,\lambda})\lesssim h^{6/5}\lambda.
\]
\end{lem}
\begin{rem}
Although this choice of $n,k,\delta$ is sometimes optimal with respect
to the wrinkling construction considered in this section, it is suboptimal
at the level of the free-shear functional. More precisely, in the
regime of this result, one can achieve significantly less free-shear
energy by not wrinkling at all. Indeed, the scaling law of $h^{6/5}\lambda$
is not present in the statement of \prettyref{prop:FSscalinglaw}. \end{rem}
\begin{proof}
Note that $\delta\leq1$ since $h\leq\frac{1}{2^{5}}$. Now we check
the slope bound. We have that
\[
m_{2}=2\max\left\{ \sqrt{\frac{\lambda}{2h^{2/5}}},\frac{\lambda}{2h^{2/5}},\frac{1}{2\pi}\frac{1}{k}\frac{\lambda}{h^{2/5}}+2\pi\frac{2\sqrt{2}\lambda^{1/2}h^{1/5}}{n}\right\} .
\]
Rearranging the inequality $h^{6/5}\lambda\geq m^{-1/2}h\lambda^{3/2}$,
we find that $m\geq\lambda h^{-2/5}$ so that
\[
m_{2}\leq2\max\left\{ \sqrt{\frac{m}{2}},\frac{m}{2},\frac{1}{2\pi}\frac{m}{k}+2\pi\frac{2\sqrt{2}}{n}\lambda^{1/2}h^{1/5}\right\} =2\max\left\{ \frac{m}{2},\frac{1}{2\pi}\frac{m}{k}+2\pi\frac{2\sqrt{2}}{n}\lambda^{1/2}h^{1/5}\right\} .
\]
Rearranging the inequality $h^{6/5}\lambda\geq(h\lambda)^{12/11}$
we find that $\lambda\leq h^{6/5}$, and hence that
\[
m_{2}\leq2\max\left\{ \frac{m}{2},\frac{1}{2\pi}\frac{m}{k}+2\pi\frac{2\sqrt{2}}{n}h^{4/5}\right\} .
\]
Using that $h\leq\frac{1}{2^{5}}$, $m\geq2$, and $n,k\geq2$ we
see that 
\[
\frac{1}{2\pi}\frac{m}{k}+2\pi\frac{2\sqrt{2}}{n}h^{4/5}\leq\frac{m}{2}
\]
so that $m_{2}\leq m$ as required.

It follows from \prettyref{lem:FSadmissability} that $\phi_{\delta,n,k,\lambda}\in A_{\lambda,m}$,
and that
\[
FS_{h}(\phi_{\delta,n,k,\lambda})\lesssim\max\left\{ \lambda h^{6/5},\lambda^{2}\right\} =\lambda h^{6/5}.
\]

\end{proof}

\subsubsection{Blow-up rate of $D\phi$ as $h\to0$ for the free-shear functional\label{sub:BlowuprateFS}}

We can now make \prettyref{rem:blowupFS}  precise.
\begin{cor}
Let $\left\{ (h_{\alpha},\lambda_{\alpha})\right\} _{\alpha\in\R_{+}}$
be such that $h_{\alpha},\lambda_{\alpha}\in(0,\frac{1}{2}]$. Assume
that $h_{\alpha}\ll\lambda_{\alpha}^{5/6}$ as $\alpha\to\infty$,
and let $\{\phi^{\alpha}\}_{\alpha\in\R_{+}}$ satisfy 
\[
\phi^{\alpha}\in A_{\lambda_{\alpha},\infty}\quad\text{and}\quad FS_{h_{\alpha}}(\phi^{\alpha})=\min_{A_{\lambda_{\alpha},\infty}}FS_{h_{\alpha}}.
\]
 Then we have that
\[
h_{\alpha}^{-1/11}\lambda_{\alpha}^{9/22}\lesssim\norm{D\phi_{\rho}^{\alpha}}_{L^{\infty}(\Omega)}\quad\text{as}\ \alpha\to\infty.
\]
\end{cor}
\begin{proof}
For ease of notation, we omit the index $\alpha$ in what follows.
By \prettyref{prop:FSscalinglaw} we have that
\[
FS_{h}(\phi)\lesssim(h\lambda)^{12/11}.
\]
Hence, by \prettyref{cor:FSLB-1}, it follows that 
\[
\lambda^{2}\lesssim(h\lambda)^{12/11}\quad\text{or}\quad\norm{D\phi_{\rho}}_{L^{\infty}(\Omega)}^{-1}h\lambda^{3/2}\lesssim(h\lambda)^{12/11}.
\]
Rearranging, we have that
\[
\lambda^{5/6}\lesssim h\quad\text{or}\quad h^{-1/11}\lambda^{9/22}\lesssim\norm{D\phi_{\rho}}_{L^{\infty}(\Omega)}.
\]
By assumption the first inequality does not hold, so the result follows.
\end{proof}

\subsection{Nonlinear model\label{sub:neutralmandrelLB_NL}}

By combining the interpolation inequalities used in the analysis of
the free-shear functional above and the uniform-in-mandrel lower bounds
from \prettyref{sub:largemandrelLB_nonlinear}, we obtain the following
lower bound in the neutral mandrel case.
\begin{prop}
\label{prop:NLneutralLBs}We have that
\[
\min_{A_{\lambda,1,m}^{NL}}\,E_{h}^{NL}-\cE_{b}^{NL}(1,h)\gtrsim_{m}\min\left\{ \max\{m^{-1}h\lambda^{3/2},(h\lambda)^{12/11}\},\lambda^{2}\right\} 
\]
whenever $h,\lambda\in(0,1]$ and $m\in(0,\infty)$.\end{prop}
\begin{proof}
Let $\Phi\in A_{\lambda,1,m}^{NL}$ and introduce the radial displacement
$\phi_{\rho}=\Phi_{\rho}-1$. Recall the definition of the excess
energy given in \prettyref{eq:NLexcess}. Applying \prettyref{lem:NLmembrLB},
\prettyref{cor:NLbucklingcontrol}, and \prettyref{lem:NLbendingcontrol}
in the case $\varrho=\varrho_{0}=1$, we obtain the following estimates:
\[
\Delta^{NL}\gtrsim\norm{\phi_{\rho}}_{L_{z}^{2}L_{\theta}^{1}}^{2}\vee\norm{\partial_{\theta}\phi_{\rho}}_{L_{z}^{4}L_{\theta}^{2}}^{4},
\]
\[
\max\left\{ \frac{1}{h^{2}}\Delta^{NL},(\Delta^{NL})^{1/2}\right\} \gtrsim_{m}\norm{D^{2}\phi_{\rho}}_{L^{2}(\Omega)}^{2},
\]
and
\[
\max\{\norm{\partial_{z}\phi_{\rho}}_{L^{2}(\Omega)}^{2},(\Delta^{NL})^{1/2}\}\gtrsim_{m}\lambda.
\]
As in the proof of \prettyref{prop:NLLBs_largemandr}, we see that
either $\Delta^{NL}\gtrsim_{m}\lambda^{2}$ or else 
\[
\Delta^{NL}\gtrsim_{m}\max\left\{ \norm{\phi_{\rho}}_{L_{z}^{2}L_{\theta}^{1}}^{2},\norm{\partial_{\theta}\phi_{\rho}}_{L_{z}^{4}L_{\theta}^{2}}^{4},h^{2}\norm{D^{2}\phi_{\rho}}_{L^{2}(\Omega)}^{2}\right\} 
\]
and 
\[
\norm{\partial_{z}\phi_{\rho}}_{L^{2}(\Omega)}^{2}\gtrsim_{m}\lambda.
\]
Now the result follows from the interpolation inequalities in \prettyref{sec:Appendix},
just as in the proofs of \prettyref{cor:FSLB-1} and \prettyref{cor:FSLB-2}.
\end{proof}

\section{Appendix\label{sec:Appendix}}

In this appendix, we collect the interpolation inequalities that
were used in \prettyref{sec:largemandrelLB} and \prettyref{sec:neutralmandrelLB}.
We call $I=[-\frac{1}{2},\frac{1}{2}]$ and $Q=[-\frac{1}{2},\frac{1}{2}]^{2}$.

\subsection{Isotropic interpolation inequalities}

The following periodic Gagliardo-Nirenberg inequalities are standard. They can, for example, be easily deduced from their non-periodic analogs (see, e.g., \cite{friedman1969partial} for the non-periodic case).

\begin{lem}
\label{lem:1dinterp} \label{lem:2dinterp} We have that
\[
\norm{f}_{L^{1}(I)}^{2/5}\norm{f''}_{L^{2}(I)}^{3/5}\gtrsim\norm{f'}_{L^{2}(I)}
\]
for all $f\in H_{\text{per}}^{2}(I)$, and that 
\begin{align*}
\norm{f}_{L^{1}(Q)}^{1/2}\norm{D^{2}f}_{L^{2}(Q)}^{1/2} & \gtrsim\norm{Df}_{L^{4/3}(Q)}\\
\norm{f}_{L^{2}(Q)}^{1/2}\norm{D^{2}f}_{L^{2}(Q)}^{1/2} & \gtrsim\norm{Df}_{L^{2}(Q)}
\end{align*}
for all $f\in H_{\text{per}}^{2}(Q)$. 
\end{lem}
Combing H\"older's inequality with the second inequality above, we deduce the following result.
\begin{lem}
\label{lem:2dinterp-1} We have that 
\[
\norm{Df}_{L^{\infty}(Q)}^{1/3}\norm{f}_{L^{1}(Q)}^{1/3}\norm{D^{2}f}_{L^{2}(Q)}^{1/3}\gtrsim\norm{Df}_{L^{2}(Q)}
\]
for all $f\in H_{\text{per}}^{2}(Q)$.\end{lem}
%\begin{proof}
%This follows from \prettyref{lem:2dinterp} and H\"older's inequality.
%In particular, we use here that
%\[
%\norm{Df}_{L^{\infty}(Q)}^{1/3}\norm{Df}_{L^{4/3}(Q)}^{2/3}\geq\norm{Df}_{L^{2}(Q)}.
%\]
%
%\end{proof}

\subsection{An anisotropic interpolation inequality}

The next lemma was used to interpolate between the mixed
norms appearing in the discussion of the neutral mandrel case (see \prettyref{sec:neutralmandrelLB}).
Here, we refer to a point $x\in Q$ by its coordinates, i.e., $x=(x_{1},x_{2})$
where $x_{i}\in I$, $i=1,2$. Recall the notation for mixed $L^{p}$-norms given in \prettyref{sub:notation}.

\begin{lem}
\label{lem:mixedinterp_1} We have that
\[
\norm{f}_{L_{x_{2}}^{2}L_{x_{1}}^{1}}+\norm{\partial_{x_{1}}f}_{L_{x_{2}}^{4}L_{x_{1}}^{2}}^{1/3}\norm{f}_{L_{x_{2}}^{2}L_{x_{1}}^{1}}^{2/3}\gtrsim\norm{f}_{L^{2}(Q)}
\]
for all $f\in W^{1,4}(Q)$.\end{lem}
\begin{proof}
By a standard one-dimensional Gagliardo-Nirenberg interpolation inequality,
we have that
\[
\norm{f}_{L_{x_{1}}^{2}}\lesssim\norm{\partial_{x_{1}}f}_{L_{x_{1}}^{2}}^{1/3}\norm{f}_{L_{x_{1}}^{1}}^{2/3}+\norm{f}_{L_{x_{1}}^{1}}
\]
for a.e.\ $x_2 \in I$.
After integrating and applying H\"older's inequality, it follows that 
\[
\norm{f}_{L_{x_{2}}^{2}L_{x_{1}}^{2}}\lesssim\norm{\norm{\partial_{x_{1}}f}_{L_{x_{1}}^{2}}^{1/3}\norm{f}_{L_{x_{1}}^{1}}^{2/3}}_{L_{x_{2}}^{2}}+\norm{f}_{L_{x_{2}}^{2}L_{x_{1}}^{1}}\lesssim \norm{\partial_{x_{1}}f}_{L_{x_{2}}^{4}L_{x_{1}}^{2}}^{1/3}\norm{f}_{L_{x_{2}}^{2}L_{x_{1}}^{1}}^{2/3} + \norm{f}_{L_{x_{2}}^{2}L_{x_{1}}^{1}}.
\]
%%Finally, by H\"older's inequality,
%%\[
%%\norm{\norm{\partial_{x_{1}}f}_{L_{x_{1}}^{2}}^{1/3}\norm{f}_{L_{x_{1}}^{1}}^{2/3}}_{L_{x_{2}}^{2}}^{2}=\int_{0}^{1}\norm{\partial_{x_{1}}f}_{L_{x_{1}}^{2}}^{2/3}\norm{f}_{L_{x_{1}}^{1}}^{4/3}\,dx_{2}\lesssim\norm{\partial_{x_{1}}f}_{L_{x_{2}}^{4}L_{x_{1}}^{2}}^{2/3}\norm{f}_{L_{x_{2}}^{2}L_{x_{1}}^{1}}^{4/3}.
%%\]

\end{proof}
\bibliographystyle{amsplain}
\bibliography{cylreferences_abbr}

\end{document}